\documentclass[reprint]{revtex4-1}
\usepackage{amsmath,amssymb,amsthm}
\usepackage{bm}
\usepackage[shortlabels]{enumitem}
\usepackage{tabularx}
\usepackage{graphicx}
\usepackage{tensor}
\usepackage{tikz-cd}
\usepackage{xcolor}
\usepackage{xspace}
\usepackage{hyperref}

\DeclareMathOperator{\tr}{tr}

\DeclareMathOperator{\spn}{span}
\DeclareMathOperator{\ran}{ran}

\DeclareMathOperator{\var}{var}
\DeclareMathOperator{\diag}{diag}
\DeclareMathOperator{\AD}{AD}
\DeclareMathOperator{\Ad}{Ad}
\DeclareMathOperator{\ad}{ad}
\DeclareMathOperator{\hor}{hor}
\DeclareMathOperator{\ver}{ver}
\DeclareMathOperator{\Id}{Id}
\DeclareMathOperator{\End}{End}
\DeclareMathOperator{\Hom}{Hom}

\newtheorem{thm}{Theorem}
\newtheorem{lem}[thm]{Lemma}
\newtheorem{prop}[thm]{Proposition}

\theoremstyle{remark}
\newtheorem*{rk*}{Remark}
\newcommand{\UU}{\text{U}(1)}
\newcommand{\SO}{\text{SO}^+(1,1)}
\newcommand{\so}{\mathfrak{so}^+(1,1)}
\newcommand{\GLC}{\text{GL}(1, \mathbb{C})}
\newcommand{\GLF}{\text{GL}(F)}
\newcommand{\glc}{\mathfrak{gl}(1,\mathbb{C})}
\newcommand{\glf}{\mathfrak{gl}(F)}

\begin{document}
\title{Quantum dynamics of the classical harmonic oscillator}
\author{Dimitrios Giannakis}
\date{\today}
\affiliation{Department of Mathematics and Center for Atmosphere Ocean Science, Courant Institute of Mathematical Sciences, New York University, New York, New York 10012, USA}

\begin{abstract}
    A correspondence is established between measure-preserving, ergodic dynamics of a classical harmonic oscillator and a quantum mechanical gauge theory on two-dimensional Minkowski space. This correspondence is realized through an isometric embedding of the $L^2(\mu)$ space on the circle associated with the oscillator's invariant measure, $ \mu $, into a Hilbert space $\mathcal{H}$ of sections of a $\mathbb{C}$-line bundle over Minkowski space. This bundle is equipped with a covariant derivative induced from an $\SO$ gauge field (connection 1-form) on the corresponding inertial frame bundle, satisfying the Yang-Mills equations. Under this embedding, the Hamiltonian operator of a Lorentz-invariant quantum system, constructed as a geometrical Laplace-type operator on bundle sections, pulls back to the generator of the unitary group of Koopman operators governing the evolution of classical observables of the harmonic oscillator, with Koopman eigenfunctions of zero, positive, and negative eigenfrequency corresponding to quantum eigenstates of zero (``vacuum''), positive (``matter''), and negative (``antimatter'') energy. The embedding also induces a pair of operators acting on classical observables of the harmonic oscillator, exhibiting canonical position--momentum commutation relationships. These operators have the structure of order-$1/2$ fractional derivatives, and therefore display a form of non-locality. 
    
    In a second part of this work, we study a quantum mechanical representation of the classical harmonic oscillator using a one-parameter family of reproducing kernel Hilbert spaces,  $ \hat{\mathcal{K}}_\tau $, associated with the time-$\tau$ diffusion kernel of an order-$1/2$ fractional diffusion on the circle. As shown in recent work, in addition to being Hilbert spaces with the reproducing property, these spaces are unital Banach algebras of functions. It is found that the evolution of classical observables in these spaces takes place via a strongly-continuous, unitary Koopman evolution group,  which exhibits a stronger form of classical--quantum consistency than the $L^2(\mu)$ case. Specifically, for every real-valued classical observable in $ \hat{ \mathcal K}_\tau $, there exists a quantum mechanical observable, whose expectation value is consistent with classical function evaluation. This allows for a description of classical state space dynamics, classical statistics, and quantum statistics of the harmonic oscillator within a unified framework.  
\end{abstract}
    
\maketitle

\section{\label{secIntro}Introduction and statement of main results}

The classical harmonic oscillator, or circle rotation, is arguably among the most widely encountered dynamical systems in science, with a vast range of applications in mechanics, wave propagation, signal processing, and many other areas. From the point of view of ergodic theory, it also provides one of the simplest non-trivial examples of measure-preserving, ergodic dynamics, characterized by a discrete spectrum of frequencies at integer multiples of the oscillator's natural frequency. In this work, we establish a correspondence between this simple classical dynamical system and quantum dynamics of a gauge field theory on two-dimensional Minkowski space, having the proper orthochronous Lorentz group $\SO$ as the structure group. This correspondence is realized through the operator-theoretic formulation of ergodic theory \cite{Baladi00,EisnerEtAl15}, which characterizes dynamical systems through the action of intrinsically linear evolution operators, called Koopman operators \cite{Koopman31,KoopmanVonNeumann32}, acting on appropriate linear spaces of observables by composition with the dynamics. In particular, for a continuous-time, measure-preserving, ergodic dynamical flow $ \Phi^t : S \to S$ on a state space $S$, the Koopman operators $U^t : f \mapsto f \circ \Phi^t$ act by unitary transformations on the $L^2(\mu)$ Hilbert space associated with the invariant measure $ \mu$, analogously to the unitary Heisenberg operators of quantum mechanics. 

Analogies of this type have been at the focus of a number of recent studies on the connections between quantum mechanics and operator-theoretic ergodic theory \cite{Mauro02,BondarEtAl12,Klein18,BondarEtAl19,Giannakis19b,Mezic19,Morgan19,Morgan19b,SlawinskaEtAl19}. Among these is a scheme for sequential data assimilation (filtering) of partially observed classical dynamical systems, which maps these systems into abstract quantum systems by means of the Koopman formalism, and employs the density matrix formulation of quantum dynamics and measurement to perform sequential statistical inference for the evolution of observables \cite{Giannakis19b,SlawinskaEtAl19}. Building on this framework, our goal is to explore the geometrical and algebraic properties of a classical--quantum correspondence in the specific case of the harmonic oscillator. In particular, we seek to address the question: Can a classical harmonic oscillator of fixed frequency (energy) be naturally mapped into a quantum system with a geometrical Hamiltonian operator?    

Our approach is inspired by the following basic observations: 
\begin{enumerate}
    \item The Koopman eigenfrequency spectrum of a classical harmonic oscillator of frequency $\alpha$ consists of all integer multiples $ \alpha_j =  j \alpha $, $ j \in \mathbb{Z} $. In particular, the spacing between two successive eigenfrequencies is constant and equal to the natural frequency, $ \alpha_{j+1} -\alpha_j =\alpha$.              
    \item The energy spectrum of a non-relativistic quantum harmonic oscillator of frequency $ \alpha $ takes the form $ E_j = ( 2 j + 1 ) \alpha / 2 $, $ j \in \mathbb{N}_0 $ (in units with $\hbar = 1$). That is, similarly to the classical harmonic oscillator, $ E_{j+1} - E_j $ is equal to $ \alpha$, but the quantum harmonic oscillator has a zero-point energy $ E_0 = \alpha / 2 $, and all energies are positive.   
    \item A \emph{pair} of quantum harmonic oscillators, consisting of one oscillator as above and another oscillator of the same frequency $ \alpha $, but sign-inverted Hamiltonian, has energy spectrum $ E_j = j \alpha $, with $ j$ now an arbitrary integer. Such a pair is therefore spectrally isomorphic to the Koopman group of a  classical harmonic oscillator. 
\end{enumerate}
In what follows, we will see that such a pair of quantum harmonic operators with oppositely-signed Hamiltonians arises naturally as a coordinate representation of a connection Laplacian of an $\SO$ gauge theory on two-dimensional Minkowski space, where wavefunctions correspond to sections of a $\mathbb{C}$-line bundle, and the positive and negative parts of the energy spectrum can be interpreted as corresponding to ``matter'' and ``antimatter'' states, respectively. Further, by mapping the eigenfunctions of the Koopman operator for the harmonic oscillator, which in this case coincide with Fourier functions on the circle, to the Hermite eigenstates of the quantum harmonic oscillator, with positive (negative) Koopman eigenfrequencies corresponding to positive (negative) energies, we will construct an isometric embedding of the $L^2(\mu) $ space of classical observables of the harmonic oscillator to a natural Hilbert space of sections associated with the gauge field theory. 

This embedding allows one to pull back quantum mechanical observables (self-adjoint operators) on Minkowski space to operators on classical observables of the harmonic oscillator. In particular, it is possible to pull back the quantum mechanical position and momentum operators, and we will see that the resulting operators take the form of fractional derivatives, exhibiting canonical commutation relationships. These results are summarized in the following two theorems.

\begin{thm}[Classical-quantum correspondence based on $L^2(\mu)$]
    \label{thmL2} Given a  rotation on the circle with frequency $ \alpha $, there exists a smooth $\mathbb{C}$-line bundle $E \to M$ over two-dimensional Minkowski space $M$, equipped with an $\SO$ Yang-Mills connection, and a gauge-covariant Hilbert space homomorphism $\mathcal{U} : L^2(\mu) \to \mathcal{H}$, where $ \mu $ is the Haar probability measure on $S^1$ and  $ \mathcal{H}$ a Hilbert space of sections $M \to E$, such that the following hold: 
    \begin{enumerate}[(i), wide]
        \item The skew-adjoint generator $V$ of the unitary Koopman evolution group on $L^2(\mu) $ induced by the circle rotation can be expressed as the pullback under $\mathcal{U}$ of a Hamiltonian operator $H $ on $\mathcal{H} $, i.e., 
            \begin{displaymath}
                V = i \mathcal{U}^* H \mathcal{U}
            \end{displaymath}
            on $C^\infty $ functions. In particular, $H$ has the structure of a Lorentz-invariant, gauge-covariant Laplace-type operator, given by the sum 
        \begin{displaymath}
            H = \Delta + v
        \end{displaymath}
        of the connection Laplacian $\Delta$ associated with the bundle and a quadratic potential $v$ taking negative (positive) values along timelike (spacelike) affine coordinates on $M$. Moreover, up to multiplication by $i $, $V$ and $H$ have the same spectra, consisting of integer multiples of the frequency $ \alpha $.   
    \item  For every $ \tau > 0 $, there exists a continuous, injective map $ \Psi_\tau : S^1 \to \mathcal Q(L^2(\mu))$ from the circle into the space $\mathcal{Q}(L^2(\mu))$ of regular quantum states on $L^2(\mu)$ (i.e., the set of positive, trace-class operators of unit trace), which is compatible with the dynamical flow  and the unitary evolution $ \tilde \Phi^t : \rho \mapsto U^{t*} \rho U^t $ on $ \mathcal{Q}(L^2(\mu))$ induced by the Koopman group, in the sense that
        \begin{displaymath}
            \Psi_\tau \circ \Phi^t = \tilde \Phi^t \circ \Psi_\tau, \quad \forall t \in \mathbb R.
        \end{displaymath} 
        As a result, for every $ \tau > 0 $, there exists an injective map $ \theta \mapsto \mathcal U^* \Psi_\tau(\theta) \mathcal U$ from $S^1$ into the space of quantum states $\mathcal{Q}(\mathcal{H})$ associated with the gauge theory on Minkowski space, mapping classical states (points in the circle) to pure quantum states evolving periodically under the unitary evolution $Z^t : \rho \mapsto e^{-iHt} \rho e^{iHt}$ generated by $H$. 
    \item For every $ \tau > 0$, there exists a continuous linear map $ \Omega_\tau : \mathcal A(L^2(\mu)) \to C_{\mathbb R}(S^1)$ from the Abelian Banach algebra $ \mathcal A(L^2(\mu))$ of bounded, self-adjoint operators on $L^2(\mu)$ (equipped with the operator norm and a symmetrized operator product for the multiplication operation) and the canonical Banach algebra $C_{\mathbb R}(S^1)$ of real-valued continuous functions on the circle. This map is compatible with the evolution $ \tilde U^t : A \mapsto U^t A U^{t*}$ on $ \mathcal A(L^2(\mu))$ induced by the Koopman group on $L^2(\mu)$, i.e., 
        \begin{displaymath}
            \Omega_\tau \circ \tilde U^t = U^t \circ \Omega_\tau, \quad \forall t \in \mathbb R.
        \end{displaymath}
        As a result, $ A \mapsto \Omega_\tau( \mathcal U A \mathcal U^*) $ is a continuous linear map from the space of bounded observables of the quantum system on Minkowski space, $ \mathcal A(\mathcal H)$, to classical observables in $ C_{\mathbb R}(S^1)$, which is compatible with the unitary Koopman evolution $  f \mapsto f \circ \Phi^t  $ on $C_{\mathbb R}(S^1)$ and the Heisenberg evolution $  A \mapsto e^{it H} A e^{-it H}$ on $ \mathcal{A}(\mathcal H)$.
    \end{enumerate}
\end{thm}

\begin{thm}[Canonically commuting operators for the circle rotation]
    \label{thmOps} With the notation of Theorem~\ref{thmL2}, the following hold:
    \begin{enumerate}[(i),wide]
        \item The two creation operators associated with the timelike and spacelike degrees of freedom of $H$ pull back under $\mathcal{U}$ to densely defined operators $A_-^+ $ and $A_+^+ $ on $ L^2(\mu) $, which act as creation operators for negative- (positive)-frequency Koopman eigenfunctions, respectively. These operators and their corresponding annihilation operators, $ A_- $ and $A_+ $, respectively, have the structure of order-$1/2 $ fractional differentiation operators. Moreover, they induce number operators $N_- = A_-^+ A_- $ and $ N_+ = A_+^+ A_+ $, exhibiting canonical commutation relationships, and leading to a decomposition of the generator as a difference between the positive- and negative-frequency number operators, 
        \begin{displaymath}
            V = i \alpha ( -N_- + N_+ ).
        \end{displaymath}
    \item The densely defined operators $ \tilde X_{\pm}$ and $\tilde P_{\pm}$ on $L^2(\mu) $ with
        \begin{displaymath}
            \tilde X_{\pm} = \frac{1}{\sqrt{2\alpha}} (A_{\pm} + A^+_{\pm}), \quad \tilde P_{\pm} = - i \sqrt{\frac{\alpha}{2}} (A_{\pm} - A^+_{\pm})
        \end{displaymath}
        satisfy canonical position-momentum commutation relationships, i.e., 
        \begin{displaymath}
            [ \tilde X_-, \tilde X_+ ] = [ \tilde P_-, \tilde P_+] = 0, \quad [ \tilde X_-, \tilde P_-] = [ \tilde X_+, \tilde P_+ ] = \Id.
        \end{displaymath}
        In particular, $ \tilde X_- $ and $ \tilde P_-$ (resp., $\tilde X_+$ and $ \tilde P_+$) are pullbacks under $\mathcal U$ of the position and momentum operators, respectively, associated with the timelike (resp., spacelike) degrees of freedom of $H$, and generate fractional diffusion semigroups on $L^2(\mu)$. Moreover, the pure quantum states in $\mathcal{Q}(L^2(\mu))$ associated with the Koopman eigenfunctions of the circle rotation satisfy canonical position-momentum uncertainty relationships with respect to these operators. 
    \end{enumerate}
\end{thm}

In Theorem~\ref{thmL2}, by gauge-covariance for $ \mathcal U$ we mean that under a gauge transformation this operator transforms as $ \mathcal U \mapsto \Xi \mathcal U $, where $ \Xi $ is a unitary multiplication operator on $ \mathcal H$ that commutes with $H$.  Moreover, by the connection on $E \to M $ being an $\SO$ Yang-Mills connection, we mean that it is induced from a connection 1-form on an $\SO$ frame bundle over $M$, whose corresponding field strength (curvature tensor) satisfies the Yang-Mills equations, though its Yang-Mills action is infinite due to non-compactness of $M$. In particular, $\SO$ is an Abelian group that may be identified with the universal covering group of $\UU$, so that the gauge theory employed in this work can be thought of as a two-dimensional analog of Maxwell electromagnetism. In Theorem~\ref{thmL2}(ii, iii), the one-parameter families of maps $ \Psi_\tau$ and $ \Omega_\tau$ are constructed using so-called feature maps \cite{FerreiraMenegatto13} associated with the reproducing kernel Hilbert spaces (RKHSs) $\mathcal{K}_\tau $ of functions on the circle induced by the canonical heat kernel at time parameter $\tau$. We recall that the time-$\tau$  heat kernel on the circle, $ \kappa_\tau : S^1 \times S^1 \to \mathbb R_+$, gives the transition probability for the diffusion semigroup $ e^{-\tau \mathcal L} $ on $L^2(\mu)$, where $ \mathcal L $ is the (positive-semidefinite) Laplace-Beltrami operator associated with the standard Riemannian metric on $S^1$ \cite{Rosenberg97}.  

Next, we study the action of the dynamics on a different class of RKHSs, $ \hat{\mathcal K}_\tau$, $ \tau > 0$, which are associated with order-$1/2$ fractional diffusions on the circle generated by $ -\mathcal L^{1/2}$. As shown in recent work \cite{DasGiannakis19c}, these spaces have the distinguished property forming Banach algebras of functions. For this reason, it is natural to refer to them as reproducing kernel Hilbert algebras (RKHAs). As with the canonical heat kernel $ \kappa_\tau$, the reproducing kernel $ \hat \kappa_\tau : S^1 \times S^1 \to \mathbb R $ of $ \hat{\mathcal K}_\tau$ is a smooth, translation-invariant function for any $ \tau > 0 $.  In particular, while a general RKHS of functions on $S^1$ need not be invariant under the circle rotation (so that one cannot speak of a Koopman evolution group on an arbitrary RKHS), by virtue of the translation invariance of $\kappa_\tau $ and  $ \hat \kappa_\tau $, the spaces $\mathcal K_\tau$ and $\hat{\mathcal K}_\tau$ are invariant under the dynamics, and moreover the corresponding Koopman evolution groups are unitary. In the case of $\hat{\mathcal K}_\tau$, the fact that this space is also a Banach algebra means that every element $ f \in \hat{\mathcal K}_\tau$ has a corresponding bounded multiplication operator on $ \hat{\mathcal K}_\tau$ multiplying by that function (which is not the case for $\mathcal K_\tau$). To our knowledge, the properties of Koopman groups on the spaces $\hat{\mathcal K}_\tau$ have not been discussed elsewhere in the literature, so this material (stated as Theorem~\ref{thmRKHS2} in Section~\ref{secRKHAClassical}) should be of independent interest. 

For our purposes, a key property provided by the RKHS structure of $ \hat{\mathcal K}_\tau$ is that pointwise function evaluation can be carried out by bounded, and thus continuous, linear functionals. This leads to a stronger form of classical--quantum correspondence than Theorem~\ref{thmL2}, which is compatible with the natural Banach algebra homomorphism mapping real-valued functions in $ \hat{\mathcal K}_\tau$ to their corresponding self-adjoint multiplication operators. As a result, the evolution of every classical observable in $ \hat{ \mathcal K }_\tau$ under the circle rotation can be consistently mapped into evolution of an  observable of the quantum system on Minkowski space.

\begin{thm}[Classical-quantum correspondence based on $\hat{\mathcal K}_\tau$]
    \label{thmRKHS} Let $\hat \kappa_\tau : S^1 \times S^1 \to \mathbb{R}_+$ be the time $\tau > 0 $ transition probability kernel associated with the fractional diffusion $ e^{-\tau \mathcal L^{1/2}}$, and $ \hat{\mathcal{K}}_\tau$ be the corresponding RKHA. Let also $ \hat{ \mathcal K }_{\mathbb R, \tau} $ be the RKHA formed by the real elements of $ \hat{\mathcal K}_\tau$. Then, with the notation of Theorem~\ref{thmL2}, there exists a unitary map $ \hat{\mathcal  V}_\tau:  \hat{\mathcal K}_\tau \to L^2(\mu)$ such that the following hold for every $ \tau > 0$: 
    \begin{enumerate}[(i), wide]
        \item The Koopman operators $U^t$ induced by the circle rotation act as a strongly-continuous, unitary evolution group on $ \hat{ \mathcal K }_\tau$.
        \item  For every $ \tau > 0 $, there exists a continuous, injective map $ \hat \Psi_\tau : S^1 \to \mathcal Q(\hat{\mathcal K}_\tau)$ from the circle into the space $\mathcal{Q}(\hat{\mathcal K}_\tau)$ of regular quantum states on $\hat{\mathcal K}_\tau$ (defined analogously to $\mathcal Q(L^2(\mu))$), which is compatible with the dynamical flow  and the unitary evolution $ \hat \Phi^t : \rho \mapsto U^{t*} \rho U^t $ on $ \mathcal{Q}(\hat{\mathcal K}_\tau)$ induced by the Koopman group; that is, 
        \begin{displaymath}
            \hat \Psi_\tau \circ \Phi^t = \hat \Phi^t \circ \hat \Psi_\tau, \quad \forall t \in \mathbb R.
        \end{displaymath} 
        As a result, there exists an injective map $ \theta \mapsto \hat{\mathcal W}_\tau^* \hat \Psi_\tau(\theta) \hat{\mathcal W}_\tau$, $\hat{\mathcal W}_\tau = \mathcal U \hat{\mathcal V}_\tau$, from $S^1$ into the space of quantum states $\mathcal{Q}(\mathcal{H})$ associated with the gauge theory on Minkowski space, mapping classical states to pure quantum states evolving periodically under the Heisenberg evolution group associated with $H$. 

    \item There exists a continuous linear map $ \hat \Omega_\tau : \mathcal A(\hat{\mathcal K}_\tau) \to \hat{\mathcal K}_{\mathbb R, \tau}$, where $ \mathcal A(\hat{\mathcal K}_\tau)$ is the Abelian Banach algebra of bounded, self-adjoint operators on $ \hat{\mathcal K}_\tau$ (defined analogously to $ \mathcal A(L^2(\mu))$),  such that $\hat\Omega_\tau$ is compatible with the Koopman group on $\hat{\mathcal K}_\tau$ and the evolution $\tilde U^t: A \mapsto U^t A U^{t*}$ on $\mathcal A(\hat{\mathcal K}_\tau)$ induced by it, viz.
        \begin{displaymath}
            \hat \Omega_\tau \circ \tilde U^t = U^t \circ \hat \Omega_\tau, \quad \forall t \in \mathbb R.
        \end{displaymath}
        As a result, $ A \mapsto \hat \Omega_\tau( \mathcal W_\tau A \mathcal W_\tau) $ is a continuous linear map from the space  $ \mathcal A(\mathcal H)$ of bounded quantum mechanical observables of the quantum system on Minkowski space to classical observables in $ \hat{\mathcal K}_{\tau,\mathbb R} $.     
    \item The map $ \hat \Omega_\tau $ is a left inverse of the natural Banach algebra homomorphism $ T : \hat{\mathcal K}_{\tau,\mathbb R} \to \mathcal A(\hat{\mathcal K}_\tau)$ mapping $f \in \hat{\mathcal K}_{\tau, \mathbb R}$ to the bounded, self-adjoint multiplication operator $ T_f  : g \mapsto f g$. As a result, for every classical observable $ f \in \hat{\mathcal K}_{\tau,\mathbb R}$ and state $ \theta \in S^1$, pointwise evaluation can be expressed as a quantum mechanical expectation value, 
            \begin{displaymath}
                f(\theta) = \mathbb{E}_{\rho_\theta} T_f  := \tr(\rho_\theta T_f ) , \quad \rho_\theta = \hat \Psi_\tau(\theta).
            \end{displaymath}
            Equivalently, $ f(\theta) $ can be expressed as an expectation value of the observable $ A_f = \hat{\mathcal W}_\tau T_f \hat{\mathcal W}_\tau^* $ of the quantum system on Minkowski space (which is not necessarily a multiplication operator), i.e., 
            \begin{displaymath}
                f(\theta) = \mathbb E_{\sigma_\theta}A_f, \quad \sigma_\theta = \hat{\mathcal W}^*_\tau \rho_\theta \hat{\mathcal W}_\tau.
            \end{displaymath}
        \item  The correspondences in~(iii) are compatible with dynamical evolution, i.e., 
            \begin{displaymath}
                f(\Phi^t(\theta)) = \mathbb E_{\tilde \Phi^t(\rho_\theta)} T_f = \mathbb E_{Z^t(\sigma_\theta)} A_f, \quad \forall t \in \mathbb R.
            \end{displaymath}
    \end{enumerate}
\end{thm}

The plan of this paper is as follows. In Section~\ref{secClassical}, we outline aspects of the Koopman operator formalism for the circle rotation on $L^2(\mu)$, as well as the properties of the heat kernel and corresponding RKHSs. In Section~\ref{secGauge}, we describe the construction of our $\SO$ gauge theory on two-dimensional Minkowski space, and in Section~\ref{secQuantum} develop its quantum formulation and correspondence with the dynamics of the circle rotation. Together, these sections constitute a proof of Theorem~\ref{thmL2}. In Section~\ref{secCCO}, we describe the construction of canonically commuting ladder operators for the classical harmonic oscillator, proving Theorem~\ref{thmOps}.  In Section~\ref{secRKHA}, we present the Koopman operator formulation for the circle rotation on the RKHAs associated with fractional diffusions on the circle, as well as the classical--quantum correspondence associated with these spaces, proving Theorem~\ref{thmRKHS}. Concluding remarks are stated in Section~\ref{secConclusions}.  The paper contains three appendices with definitions and  technical results on fiber bundles (Appendix~\ref{appBundle}), RKHS theory (Appendix~\ref{appRKHS}),  and fractional derivatives (Appendix~\ref{appFractional}).          

\section{\label{secClassical}Ergodic dynamics of the classical harmonic oscillator}

In this section, we introduce the Koopman operator formalism for the harmonic oscillator, including recently proposed ladder-like operators acting on Koopman eigenspaces \cite{Mezic19}. In addition, we define and outline the basic properties of the heat kernel and corresponding RKHSs which will be useful in establishing our classical--quantum correspondence. Additional details on Koopman operator theory can be found in one of the many references in the literature, e.g., Refs.~\cite{Baladi00,EisnerEtAl15}. Aspects of RKHS theory can also be found in many references, e.g., \cite{SriperumbudurEtAl11,FerreiraMenegatto13,PaulsenRaghupathi16}. Appendix~\ref{appRKHSBasic} summarizes basic definitions and results from RKHS theory that are pertinent to this work. 

\subsection{\label{secCircle}Koopman operator formalism}

We represent the dynamics of a classical harmonic oscillator of angular frequency $ \alpha \in \mathbb{R} $ by a translation map $ \Phi^t: S^1 \to S^1$, $ t \in \mathbb{R}$, on the circle $ S^1$, viz.
\begin{displaymath}
    \Phi^t( \theta ) = \theta + \alpha t \mod 2 \pi.
\end{displaymath}
As a concrete example, in Hamiltonian mechanics, $S^1$ would be a subset of the cotangent bundle $T^* \mathbb{R} \simeq \mathbb{R}^2$ representing a constant-energy surface in position--momentum coordinates, but in general we can consider $S^1 $ as the state space of a classical dynamical system without reference to an embedding. As is well known, $ \Phi^t$ has a unique Borel, ergodic, invariant probability measure $ \mu $ (i.e., a Haar measure) corresponding to a normalized arclength. 

We use the notations $C^r(S^1)$, $r \in \mathbb{N}_0$, and $L^p(\mu)$, $ p \geq 1 $, to represent the Banach spaces of complex-valued functions on $S^1$ with $r$ continuous derivatives, and (equivalence classes of) of complex-valued functions with $\mu$-integrable $p$-th power, equipped with the standard norms, respectively. In the Hilbert space case, $L^2(\mu)$, we also use the notations $ \langle f, g \rangle_\mu = \int_{S^1} f^* g \, d\mu$ and $ \lVert f \rVert_\mu = \sqrt{\langle f, f \rangle}_\mu$ for the corresponding inner product and norm, respectively. Further, we abbreviate $C^0(S^1)$ by $C(S^1)$. We shall refer to elements of $C^r(S^1)$ and $L^p(\mu)$ as classical observables.

In this setting, the Koopman group of evolution operators is the strongly continuous group of isometries $U^t : \mathcal{E} \to \mathcal{E} $, where $\mathcal{E}$ stands for either $C(S^1)$ or $L^p(\mu)$,  acting on classical observables by composition with the dynamical flow map, i.e.,
\begin{equation}
    \label{eqKoop}
    U^t f = f \circ \Phi^t, \quad t \in \mathbb{R}.
\end{equation}
We will mainly focus on the Koopman group on $L^2(\mu)$, which in addition to being isometric is unitary, $U^{t*} = U^{t^{-1}} = U^{-t}$. By Stone's theorem on strongly continuous unitary evolution groups \cite{Stone32}, $ U^t : L^2(\mu) \to L^2(\mu)$ is completely characterized by its \emph{generator}; the latter, is a skew-adjoint, unbounded operator $V: D(V) \to L^2(\mu) $ with a dense domain $D(V) \subset L^2(\mu)$, acting on elements in its domain according to the formula
\begin{displaymath}
    V f = \lim_{t \to 0} \frac{U^t f - f}{ t},
\end{displaymath}
and generating the Koopman operator at any time $t \in \mathbb{R} $ by operator exponentiation,
\begin{displaymath}
    U^t = e^{tV}.
\end{displaymath}
Ter Elst and Lema\'nczyk \cite{TerElstLemanczyk17} have recently shown that the space $D(V)\cap L^\infty(\mu)$ is a Koopman-invariant algebra on which $V$ acts as a derivation. That is, for any two elements $f,g \in D(V) \cap L^\infty(\mu) $, the pointwise product $fg$ also lies in this space, and the generator satisfies a Leibniz rule,
\begin{equation}
    \label{eqLeibniz}
    V(fg ) = (V f ) g + f ( V g ).
\end{equation}
This is a manifestation of the fact that $V $ behaves as a differential operator on classical observables, which can be viewed as an extension of the vector field $ \vec V$ on $S^1 $ generating the flow $ \Phi^t $, and acting on continuously differentiable functions as a derivative operator, $ \vec V f = \alpha f' $. 

Consider now the continuous dual space $C'(S^1)$ to $C(S^1)$, equipped with the standard (operator) norm. It is a standard result that $C'(S^1)$ can be canonically identified with the Banach space $\mathcal{M}(S^1)$ of complex Radon measures on $S^1$, equipped with the total variation (TV) norm, through a linear isometry mapping $ m \in \mathcal{M}(S^1)$ to $ J_m \in C'(S^1)$ with $ J_m f = \int_{S^1} f \, dm$. The dynamics acts on $\mathcal{M}(S^1)$ through a group of isometries, $\Phi^t_* : \mathcal{M}(S^1) \to \mathcal{M}(S^1)$, $ \Phi^t_* m = m \circ \Phi^{-t}$, known as \emph{Perron-Frobenius} or \emph{transfer} operators. Under the $C'(S^1) \simeq \mathcal{M}(S^1)$ isomorphism, $ \Phi^t_*$ is identified with the transpose of the Koopman operator, $U^{t\prime}: C'(S^1) \to C'(S^1)$, $U^{t\prime} \varphi = \varphi \circ U^t $; specifically, $ U^{t\prime} \varphi_m = \varphi_{\Phi^t_* m} $. It can be readily verified that if $m \in \mathcal{M}(S^1)$ has a density $ \rho = dm/d\mu \in \mathcal{E}$ relative to the invariant measure (where again $\mathcal E $ stands for either $C(S^1)$ or $L^p(\mu)$), then $ \Phi^t_*m$ has density $U^{-t} \rho$. In particular, for $ \rho \in L^2(\mu)$, we have $U^{-t} \rho = U^{t*} \rho$, so we may identify the transfer operator with the adjoint of the Koopman operator. 

In what follows, we shall be concerned with the action of both the Koopman and transfer operators on spaces of functions and measures on the circle, respectively, as well as their induced action on operator algebras and their duals. We will continue to overload notation and employ the same symbol $U^t$ to represent the Koopman operator acting on any of the $C^r(S^1)$ and $L^p(\mu)$ spaces (as well as the RKHSs introduced below), as the particular instance will be clear from the context. In addition, we will sometimes consider spaces of real-valued continuous functions, their duals, and real Radon measures, which we will distinguish using the symbols $C_{\mathbb R}(S^1)$, $C'_{\mathbb{R}}(S^1)$, and $\mathcal M_{\mathbb R}(S^1)$, respectively. We also let $ \mathcal P (S^1) \subset \mathcal M_{\mathbb R}(S^1)$ be the space of Radon probability measures on the circle, equipped with the TV norm topology. Koopman and transfer operators are then defined on these spaces  analogously to the complex case, and we again use the symbols $U^t$, $U^{t\prime}$, and $ \Phi^t_*$, respectively, to represent them.  We will oftentimes refer to any such instance of $U^t $ and $U^{t\prime}$/$\Phi^t_*$ as ``the'' Koopman or transfer operator, respectively.

Before carrying on, we note a mathematical correspondence between the operator-theoretic description of a measure-preserving, classical dynamical system, such as the harmonic oscillator studied here, and quantum mechanics. That is, $-iV $ is a self-adjoint operator analogous to the Hamiltonian $H$ of a quantum system, and the Koopman operator $U^t$ on $L^2(\mu)$ is a unitary operator analogous to the Heisenberg evolution operator $e^{it H}$ generated by $H$. 

\subsection{\label{secKoopEig}Koopman eigenfunctions}

Next, we turn to the spectral characterization of the unitary Koopman group on $L^2(\mu)$ associated with the classical harmonic oscillator. As can be readily verified, there is a smooth orthonormal basis $ \{ \phi_j \}_{j=-\infty}^\infty$ of $L^2(\mu) $ consisting of Koopman eigenfunctions, satisfying the equation
\begin{displaymath}
    V \phi_j = i \alpha_j \phi_j,
\end{displaymath}
where $ \alpha_j = j \alpha $ is an eigenfrequency equal to an integer multiple of the oscillator's frequency, and $\phi_j( \theta ) = e^{i j \theta}$ is equal to a Fourier function. Moreover, each $ \phi_j $ is also an eigenfunction of the Koopman operator $U^t $ corresponding to the eigenvalue $e^{i\alpha t}$. The latter implies that we can compute the dynamical evolution of any classical observable $f  \in L^2(\mu)$ by first expanding it in the Koopman eigenfunction basis, $ f = \sum_{j=-\infty}^{\infty} c_j \phi_j$, and then computing
\begin{displaymath}
    U^t f = \sum_{j=-\infty}^\infty e^{i j \alpha t } c_j \phi_j,
\end{displaymath}
where the infinite sum over $ j $ converges in $L^2(\mu)$ norm.  As with any measure-preserving, continuous-time dynamical system on a manifold, the Koopman eigenfrequencies and eigenfunctions have an important algebraic group structure (which can be verified from the Leibniz rule in~\eqref{eqLeibniz}), namely, they are closed under addition and multiplication, respectively, i.e., 
\begin{displaymath}
    \alpha_j + \alpha_k = \alpha_{j+k}, \quad \phi_j \phi_k = \phi_{j+k}, \quad \forall j,k \in \mathbb{Z}.
\end{displaymath}
For completeness, we note that the existence of a complete basis of Koopman eigenfunctions is not generic to arbitrary measure-preserving, ergodic dynamical systems, and in particular systems exhibiting mixing (chaotic) behavior have a non-empty continuous Koopman spectrum with an associated subspace of $L^2(\mu)$ which does not admit a Koopman eigenfunction basis \cite{Mezic05}. 

\subsection{\label{secKoopLadder}Ladder-like operators}

We now draw another analogy between the operator-theoretic description of the classical harmonic oscillator and quantum mechanics (to our knowledge, first pointed out by Mezi\'c \cite{Mezic19}), pertaining to ladder operators \cite{Sakurai93}. In particular, consider the bounded operator $L : L^2(\mu) \to L^2(\mu) $ which multiplies by the Koopman eigenfunction $ \phi_{-1}$, i.e., $ L f = \phi_{-1} f $.
A direct calculation yields
\begin{equation}
    \label{eqLComm}
    [ V, L] = - i \alpha L
\end{equation}
on $D(V)$, which is analogous to the commutation relation between the Hamiltonian and lowering operator of the quantum harmonic oscillator, modulo the presence of the imaginary number $i$ in the right-hand side due to skew-adjointness of $V$. Similarly, we have
\begin{equation}
    \label{eqLStarComm}
    [ V, L^*] = i \alpha L^*,
\end{equation}
so that $L^*$ behaves analogously to the raising operator in the context of the quantum harmonic oscillator. Stated explicitly, the last two equations imply 
\begin{equation}
    \label{eqLadderL}
    L \phi_j = \phi_{j-1}, \quad L^* \phi_j = \phi_{j+1}, 
\end{equation}
respectively, which shows that $L$ ($L^*$) lowers (raises) the eigenfrequency of $ \phi_j $ by a unit of $\alpha$. 

Despite these similarities with ladder operators of quantum harmonic oscillators, it is important to note that there is no analog of the ground state in this picture, and that $ L\phi_j$ and $ L^* \phi_j $ are equal to $\phi_{j-1}$ and $\phi_{j+1}$, respectively, without the presence of $j$-dependent multiplication factors (cf.~\eqref{eqLadderPsi} ahead). The first of these facts underpins the correspondence put forward in Section~\ref{secIso} between the classical harmonic oscillator and a \emph{relativistic} oscillator, which naturally supports negative-energy states. The latter fact implies that $L$ and $L^*$ are bounded operators (unlike the ladder operators of the quantum harmonic oscillator), and that $L^* L$ is not a number operator. In fact, it follows directly from~\eqref{eqLadderL} that $L$ is a unitary operator, and therefore 
\begin{equation}
    L^*L = LL^* = \Id. 
    \label{eqLadderLId}
\end{equation}

In Section~\ref{secCCO} we will employ the correspondence established in Section~\ref{secIso} to construct operators that, besides satisfying the commutation relations in~\eqref{eqLComm} and~\eqref{eqLStarComm}, they also give rise to a number operator satisfying the appropriate commutation relations. These operators will turn out to be fractional differentiation operators of order $1/2$. 

\subsection{\label{secHeatKernel}Heat kernel and the associated reproducing kernel Hilbert spaces (RKHSs)}

Let $ \tilde{\mathcal{L}} : C^\infty(S^1) \to L^2(\mu)$, be the canonical Laplace-Beltrami operator on the circle, defined here as the operator on $L^2(\mu)$ with dense domain $C^\infty(S^1)$ such that $ \tilde{\mathcal L} f =  -f'' $. It is a standard result from analysis on manifolds \cite{Rosenberg97} that $ \tilde{\mathcal L} $ is a positive-semidefinite, essentially self-adjoint operator, having the Fourier functions $ \phi_j $ as its eigenfunctions,
\begin{displaymath}
    \tilde{\mathcal L} \phi_j = j^2 \phi_j, \quad j \in \mathbb Z.
\end{displaymath}
The (unique) self-adjoint extension $ \mathcal{ L } : D(\mathcal L) \to L^2(\mu)$ of $ \tilde{\mathcal L}$, whose domain $D(\mathcal L)$ is an order-2 Sobolev space, generates a Markov diffusion semigroup $\{ e^{-\tau \mathcal L } \}_{\tau \geq 0}$ on $L^2(\mu)$, called \emph{heat semigroup}, such that, for any $ \tau > 0$,
\begin{displaymath}
    e^{-\tau \mathcal L} f = \int_{S^1} \kappa_{\tau}(\cdot, \theta') f(\theta') \, d\mu(\theta').
\end{displaymath}
Here, $ \kappa_\tau : S^1 \times S^1 \to \mathbb{R}_+$ is the time-$\tau$ \emph{heat kernel} on the circle---the smooth, strictly positive, bivariate function given by 
\begin{equation}
    \begin{aligned}
        \kappa_\tau(\theta,\theta') &= \sum_{j=-\infty}^{\infty} e^{- j^2 \tau} \phi^*_j(\theta) \phi_j(\theta') \\
        & = 1 + \sum_{j=1}^\infty e^{-j^2\tau} \cos(j(\theta-\theta'))\\
        &= \sqrt{4 \pi \tau} \sum_{j=-\infty}^\infty e^{-\tau(\theta-\theta' - 2 j\pi)^2}, 
    \end{aligned}
    \label{eqHeatKernel}
\end{equation}
where the sums over $j \in \mathbb{Z}$ converge in any $C^r(S^1)$ norm, $r \in \mathbb{N}_0$. A continuous kernel admitting a uniformly convergent eigenfunction expansion as in the first line of~\eqref{eqHeatKernel} is known as a \emph{Mercer kernel}. 

The heat kernel on the circle has the property of being \emph{translation invariant}, i.e., $\kappa_{\tau}(\theta,\theta')$ depends only on the arclength distance between $\theta$ and $\theta'$. As a result, since $ \Phi^t$ preserves arclength distances, $\kappa_\tau$ is also shift-invariant under the circle rotation, $\kappa_{\tau}( \Phi^t(\theta),\Phi^t(\theta')) = \kappa_{\tau}(\theta,\theta')$ for all $ t \in \mathbb{R}$ and $\theta,\theta' \in S^1$. In addition, $\kappa_\tau$ is \emph{strictly positive-definite}, meaning that for any collection $\theta_1, \ldots, \theta_N$ of distinct points in $S^1$, the $N\times N$ kernel matrix $\bm K = [ \kappa_\tau(\theta_i,\theta_j) ]_{ij}$ is strictly positive. 

Associated with the heat kernel $\kappa_\tau$ is a \emph{reproducing kernel Hilbert space (RKHS)} of complex-valued functions on $S^1$; that is, a Hilbert space $ \mathcal{K}_\tau  $ with inner product $ \langle \cdot, \cdot \rangle_{\mathcal{K}_\tau}$ and corresponding norm $ \lVert \cdot \rVert_{\mathcal{K}_\tau}$, such that, for every $ \theta \in S^1$, (i) the kernel section $ \kappa_\tau(\theta,\cdot) $ lies in $ \mathcal{K}_\tau$; and (ii) the evaluation functional $ \mathbb V_\theta : \mathcal{K}_\tau \to \mathbb{C}$, $\mathbb V_\theta f = f(\theta)$ is bounded (and thus, continuous), and satisfies $ \mathbb V_\theta f = \langle \kappa_{\tau}(\theta,\cdot), f \rangle_{\mathcal{K}_\tau}$. The latter, is known as the \emph{reproducing property}, and leads to the inner product relationships $ \langle \kappa_\tau(\theta,\cdot),\kappa_{\tau}(\theta',\cdot) \rangle_{\mathcal{K}_\tau}= \kappa_{\tau}(\theta,\theta')$ for the kernel sections at any $\theta,\theta' \in S^1$.  

It can be shown, e.g., using results in Refs.~\cite{SriperumbudurEtAl11,FerreiraMenegatto13}, that the $\mathcal{K}_\tau$ are subspaces of $C^\infty(S^1)$, forming an increasing sequence as $ \tau $ decreases to 0. Moreover, for every $\tau > 0$,  $\mathcal{K}_\tau $ is a dense subspace of $C(S^1)$, and for any $ r \in \mathbb{N}_0$, the inclusion $ \mathcal{K}_\tau \hookrightarrow C^r(S^1) $ (and thus $\mathcal K_\tau \hookrightarrow L^p(\mu)$) is bounded. It is a consequence of the continuity of $ \kappa_\tau$ and density of $\mathcal K_\tau$ in $C(S^1)$ that $\mathbb K_\tau : \mathcal P(S^1) \to \mathcal K_\tau$, with
\begin{displaymath}
    \mathbb K_\tau(m) = \int_{S^1} \kappa_\tau(\theta,\cdot) \, dm(\theta),
\end{displaymath}
is an injective, continuous map from Radon probability measures to RKHS functions on the circle. Such a map is oftentimes referred to as an \emph{RKHS embedding} of probability measures. Now, as can be verified by routine calculations, the map $ \delta : S^1 \to \mathcal P(S^1)$ sending $\theta \in S^1$ to the Dirac measure $\delta_\theta =: \delta(\theta) $ supported at $\theta$ is also injective and continuous. As a result, the  so-called \emph{feature map}, $ F_{\tau} : S^1 \to \mathcal{K}_\tau$, 
\begin{equation}
    F_\tau = \mathbb K_\tau \circ \delta, \quad F_\tau(\theta) = \kappa_\tau(\theta,\cdot),
    \label{eqFeatureMap}
\end{equation}
is an injective, continuous map, mapping points on the circle to the RKHS functions given by the corresponding kernel sections (also known as \emph{feature vectors}). It also follows from the strict positive-definiteness of $\kappa_\tau$ that   $F_\tau(\theta)$ and $F_\tau(\theta')$ are linearly independent whenever $\theta$ and $\theta' $ are distinct, so that the image $\mathcal F(\mathcal K_\tau ) := F_\tau(S^1) \subset \mathcal K_\tau$ contains only nonzero functions. See Lemmas~\ref{lemFeature} and~\ref{lemUniversal} in Appendix~\ref{appRKHSBasic} for further details.  

A useful correspondence between the Hilbert spaces $\mathcal K_\tau$ and $L^2(\mu)$ is provided by the integral operators $K_\tau : L^2(\mu) \to \mathcal K_\tau$, 
\begin{displaymath}
    K_\tau f = \int_{S^1} \kappa_\tau(\cdot,\theta') f(\theta') \, d\mu(\theta'),
\end{displaymath}
which can be shown to be well-defined, compact integral operators with dense range for any $\tau >0$. Moreover, the adjoint $ K_\tau^* $ maps injectively $f \in \mathcal{K}_\tau \subset C^\infty(S^1) $ to its corresponding $L^2(\mu)$ equivalence class, so that $K^*_\tau : \mathcal{K}_\tau \hookrightarrow L^2(\mu)$ is a compact embedding. It then follows that $ e^{-\tau \mathcal L} = K_\tau^* K_\tau$, and that $ \{ \phi_{j,\tau}  \}_{j\in \mathbb{Z}}$ with 
\begin{displaymath}
    \phi_{j,\tau}= e^{j^2 \tau / 2} K_\tau \phi_j = e^{-j^2 \tau / 2} \phi_j   
\end{displaymath}
is an orthonormal basis of $\mathcal{K}_\tau$. As a result, we can characterize $\mathcal K_\tau$ as the space of smooth functions $f: S^1 \to \mathbb C$ admitting (uniformly convergent) Fourier expansions $ f(\theta) = \sum_{j=-\infty}^\infty \hat f_j \phi_j(\theta) $, $ \hat f_j \in \mathbb{C}$, satisfying
\begin{equation}
    \label{eqRKHSDecay}
    \sum_{j=-\infty}^\infty e^{j^2\tau} \lvert \hat f_j \rvert^2<\infty.
\end{equation}

It is worthwhile noting that, besides $K^*_\tau$, another natural way of mapping functions in $\mathcal K_\tau$ to $L^2(\mu)$ elements is through the unitary operator $\mathcal V_\tau : \mathcal K_\tau \to L^2(\mu)$ defined uniquely by $\mathcal V_\tau \phi_{j,\tau} = \phi_j$, for all $ j \in \mathbb Z$. This operator can also be defined in a basis-free manner through the polar decomposition of the integral operator $K_\tau$. Specifically, $\mathcal V_\tau$ is the unique operator from $\mathcal K_\tau $ to $L^2(\mu)$ such that 
\begin{equation}
    \label{eqPolarDecomp}
    K_\tau = \mathcal V_\tau^{*} e^{-\tau \mathcal L / 2}.
\end{equation}

\section{\label{secGauge}Gauge field theory}

In this section, we describe the construction of the  gauge field theory into which we will embed the classical harmonic oscillator in Section~\ref{secQuantum}. We will follow a geometric approach \cite{BerlineEtAl04,Michor08,RudolphSchmidt17}, where the two key objects are a principal bundle of the proper, orthochronous Lorentz group over two-dimensional Minkowski space and an associated $ \mathbb{C}$-line bundle. The gauge field is then understood as a connection 1-form on the principal bundle, inducing a covariant derivative acting on sections of the associated bundle. The space of sections is also endowed with a Hilbert space structure, providing the foundation for defining  quantum mechanical observables as self-adjoint operators. Among these, the Hamiltonian operator generates the quantum dynamics of the theory, and will take the form of a connection Laplacian when expressed in terms of the covariant derivative.   

In what follows, we will present the steps of this construction sequentially, starting from the basic properties of Minkowski space, and adding the necessary structure to arrive at the connection 1-form and its associated covariant derivatives and curvature. Readers familiar with gauge theory may wish to skip to Section~\ref{secConn}, where these objects are explicitly defined. Auxiliary results, as well as an overview of basic concepts from fiber bundle theory, are included in Appendix~\ref{appBundle}. We refer the reader to one of the many references in the literature (e.g., \cite{BerlineEtAl04,Michor08,RudolphSchmidt17}) for detailed expositions on these topics.

\subsection{Notation}
We will use the notation $  {E} \xrightarrow{\pi}  {M} $ to represent a smooth fiber bundle with base space $ {M}$, total space $ {E}$, and projection map $ \pi:  {E} \to {M}$. Moreover, $ \Gamma( {E})$ will denote the space of smooth sections of $ {E}\xrightarrow{\pi}  {M}$. If $  {E} \xrightarrow{\pi}  {M}$ is a vector bundle with real fibers,  $\Omega^k( {M}, {E}) $ will be the space of sections of the vector bundle over $ {M}$ with total space $\bigwedge^k T^*  {M} \otimes_{\mathbb{R}}  {E}$, where $ \bigwedge^k T^* {M} $ is the $k$-th exterior power of the cotangent bundle $T^* {M}$, and $\otimes_{\mathbb{R}}$ the tensor product of real vector bundles. If $ F$ is a vector space over the real numbers, $C^\infty(M, F) $ and $ \Omega^k ( {M},F)$ will denote the space of smooth $F$-valued functions and $k$-forms on $ {M}$, respectively. Note that $C^\infty(M,F) = \Omega^0(M,F)$ and $\Omega^k( {M},F) = \Omega^k( {M}, M \times F)$, where we view $M\times F$ as the total space of the trivial bundle $ M \times F \xrightarrow{\pi} M $ with $\pi $ the canonical projection map onto the first factor. We will abbreviate $ \Omega^k( {M}, \mathbb{R})$ by $ \Omega^k( {M})$. If $  {E} \xrightarrow{\pi}  {M}$ has complex fibers and $ F$ is a vector space over the complex numbers, $ \Omega^k( {M}, {E})$ and $\Omega^k( {M}, F)$ are defined analogously to the real case, replacing $T^* {M}$ by the complexified cotangent bundle, $T_{\mathbb C}^*M$ and $ \otimes_{\mathbb{R}} $ by the tensor product $\otimes_{\mathbb C}$ of vector bundles over the complex numbers. We will use $\Gamma_c( {E}) \subset \Gamma( {E}) $ to denote the space of compactly supported sections of $ {E} \xrightarrow{\pi}  {M}$. When convenient, we will use subscript notation to represent pointwise evaluation of sections; e.g., for a section $ s \in \Gamma(E)$ we set $s_m = s(m)$.  

\subsection{\label{secMinkowski}Minkowski space}

We are interested in constructing a gauge theory over two-dimensional \emph{Minkowski space}, which has the structure of a metric affine space $ (M, \vec M, \eta)$.  Here, $M$ is a two-dimensional manifold, whose points represent \emph{events}, $\vec M$ is a two-dimensional vector space over the real numbers, whose elements represent \emph{translations}, acting on $M$ freely and transitively as an Abelian group, and $\eta : \vec M \times \vec M \to \mathbb{R}$ is a symmetric, non-degenerate bilinear form with signature $ (-,+) $. In addition, $ \vec M $  is equipped with a distinguished vector $ \vec\tau \in \vec M $ with $\eta(\vec \tau, \vec \tau ) < 0 $, providing a notion of \emph{future direction}.

Given any point $ m \in M $, the tangent space $T_m M $ can be canonically identified with $ \vec M $ (through identification of curves), and thus inherits a pseudo-Riemannian metric tensor $\eta_m : T_m M \times T_m M \to \mathbb{R} $ from $ \eta $, called \emph{Minkowski metric}. Moreover, for every point $o \in M$ and basis $ \{ \vec X_0, \vec X_1 \} $ of $ \vec M $, there exists a \emph{Cartesian coordinate chart}, $ x: M \to \mathbb{R}^2$, such that $ x(m) = (x^0,x^1) $, where $  x^0 \vec X_0 + x^1 \vec X_1 = \vec v $, and $ \vec v $ is the unique element of $ \vec M $ such that $ m = o + \vec v $. Here, the notation $ o + \vec v $ represents translation of the point $ o \in M $ by the vector $ \vec v \in \vec M $. 

Given two points $m_1, m_2 \in M$, we will let $\overrightarrow{m_1 m_2}$ be the unique translation in $ \vec M $ such that $ m_2 = m_1 + \overrightarrow{m_1 m_2}$. Moreover, we will denote the ``inverse metric'' to $ \eta $, acting on dual vectors, by $ \eta' $, and use the same symbol to represent the canonical lift of $\eta'$ to 2-forms. We also let ${}^\flat : \vec M \to \vec M'$ and $ {}^\sharp : \vec M' \to \vec M $ be the Riemannian isomorphisms between vectors and dual vectors, where $ \vec M' $ is the dual space to $\vec M$, $ \vec v^\flat = \eta( \vec v, \cdot ) $, and $ w^\sharp = \eta'( w,\cdot)$.  

A map $ a : M \to M $ is said to be \emph{affine} if there exists $ A \in \text{GL}(\vec M)$ such that for all $ m \in M$ and $ \vec v \in \vec M $, $ a( m + \vec v ) = a(m) + A \vec v $. The linear map $ A $ may then be identified with the pushforward map $ a_* $ on tangent vectors.  The set of all affine maps $ a $ on $M$ preserving $\eta $, i.e., $ \eta(\vec v_1, \vec v_2) = \eta(A\vec v_1, A\vec v_2) $, for all $\vec v_1, \vec v_2 \in \vec M$, forms a group under composition of maps, called \emph{Poincar\'e group}. Note that the Poincar\'e group contains the translations $\vec M $ as a subgroup, since for every $ \vec v \in \vec M $ the map $a : M \to M $ with $ a(m) = m + \vec v $ is affine with $ A = \Id $. In what follows, given a linear map $ A $ on $ \vec M $, $ A^* $ will denote its adjoint with respect to $ \eta $, i.e., the unique linear map satisfying $ \eta( \vec v_1, A^* \vec v_2 ) = \eta( A \vec v_1, \vec v_2 ) $.

Next, let $ \nu \in \Omega^2(M)$ be the volume form associated with $ \eta $. Together, $ \eta $ and $ \nu $ induce a \emph{Hodge star operator} on forms, $ \star : \Omega^k(M) \to \Omega^{2-k}(M)$, defined uniquely through the requirement that $ v \wedge \star w = \eta'( v, w ) \nu $ for all $ v, w \in \Omega^k(M)$. The Hodge star operator induces in turn a map $^\perp : \Gamma(TM) \to \Gamma(TM) $ on vector fields, such that 
\begin{equation}
    \label{eqPerp}
    X^\perp f = (\star df) X, \quad  \forall f \in C^\infty(M). 
\end{equation}
Intuitively, $X^\perp$ can be thought of as a directional derivative in a perpendicular direction to $ X $, obtained by a positive (``anticlockwise'') rotation with respect to the orientation induced by $ \nu$. One can readily verify that $^\perp $ is compatible with the Poincar\'e group; that is, for every affine map $ a : M \to M $ and vector field $ X \in \Gamma(TM) $, $(a_* X)^\perp = a_*( X^\perp)$, where $ a_* : \Gamma(TM) \to \Gamma(TM)$ is the pushforward map on vector fields associated with $ a$.  It should be noted that the fact that $M$ is two-dimensional is important in the definition of $^\perp$, and thus will also be important in our definition of a connection 1-form in Section~\ref{secConn} utilizing this map. In what follows, $L^2(\nu)$ will denote the Hilbert space of (equivalence classes of) complex-valued functions on $M$, square-integrable with respect to $ \nu $. This space is equipped with the inner product $ \langle f_1, f_2 \rangle_\nu = \int_M f_1^* f_2 \, d\nu$ and the corresponding norm $\lVert f \rVert_\nu = \sqrt{\langle f, f \rangle_\nu}$.

Due to the existence of global Cartesian coordinate charts, $M $ is diffeomorphic to $ \mathbb{R}^2 $, but note that the diffeomorphism is not canonical as it depends both on the choice of the point $ o $, which can be thought of as an \emph{origin}, and the basis $ \{ \vec X_0, \vec X_1 \} $. Nonzero vectors $ \vec v \in \vec M$ for which $ \eta( \vec v, \vec v ) $ is negative, zero, or positive will be said to be \emph{timelike}, \emph{null}, or \emph{spacelike}, respectively. Two timelike vectors $\vec v_1$ and $\vec v_2$ with $ \eta( \vec v_1, \vec v_2 ) < 0 $ are said to be \emph{co-oriented}.  Timelike, null, spacelike, and co-oriented elements of the tangent spaces $T_m M$ are defined analogously. 

Hereafter, we will let $x$ be a Cartesian chart with origin $ o $ and basis $ \{ \vec X_0, \vec X_1 \} $, and further require that $ \{ \vec X_0, \vec X_1 \} $  be orthonormal, positive-oriented, and has $ \vec X_0 $ future-directed, i.e., $ \eta( \vec X_0, \vec X_0 ) = - 1$, $ \eta(\vec X_1, \vec X_1 ) = 1$, $ \eta( \vec X_0, \vec X_1 ) = 0$, $ \nu( \vec X_0, \vec X_1) = 1 $, and $ \eta( \vec \tau, \vec X_0 ) < 0 $. We will denote the dual basis vectors to $ \vec X_i $ by $ \hat X^i$, where $ \vec X^i \cdot \vec X_j = {\delta^i}_j $.  For convenience, we will abbreviate the coordinate vector fields $ \frac{\partial\ }{\partial x^0} $ and $ \frac{\partial\ }{\partial x^1} $ on $M$ by $ X_0$ and $ X_1 $, respectively. Note that because $\eta$ does not have $(+,+)$ signature, the dual basis vectors are not all equal to their Riemannian duals; in particular, $\hat X^0 = - X_0^\flat$ and $\hat X^1 = X_1^\flat$.  We also note the relationships $ X^\perp_0 = - X_1 $ and $ X^\perp_1 = - X_0 $, which indicate that $ \{ X_{0,m}^\perp, X_{1,m}^\perp \} $ is a negatively-oriented, past-directed, orthonormal basis of $T_mM$ at every $m \in M$. For brevity, we shall refer to any Cartesian chart with the properties listed above as an \emph{inertial chart with origin $o$}. Moreover, when there is no risk of confusion about the choice of origin, we will simply refer to $ x$ as an \emph{inertial chart}.  

\subsection{Lorentz group}

In what follows, we will construct principal and associated bundles over $M$ having as their structure group the proper, orthochronous \emph{Lorentz group}, defined as the subgroup $ G \subset \text{GL}(\vec M )$, whose every element $ \Lambda $ satisfies the conditions (i) $ \eta( \Lambda \vec v_1, \Lambda \vec v_2 ) = \eta( \vec v_1, \vec v_2 )  $, for all $ \vec v_1, \vec v_2 \in \vec M $; (ii) $ \det \Lambda = 1 $; $ \eta( \vec \tau, \Lambda \vec \tau ) < 0 $. The group $G$ is a one-dimensional, connected, Abelian Lie group, and is isomorphic to the matrix group $\SO$, consisting of the real $2 \times 2$ matrices 
\begin{displaymath}
    \bm \Lambda_\theta = 
    \begin{pmatrix}
        \cosh \theta & \sinh \theta \\
        \sinh \theta & \cosh \theta
    \end{pmatrix}, 
    \quad \theta \in \mathbb{R}.
\end{displaymath}
Specifically, given an inertial chart $x : M \to \mathbb R^2$, each element  $ \Lambda \in G$ can be smoothly identified with a matrix $ \bm \Lambda_\theta \in \SO$, with elements $ \hat X^i \cdot \Lambda \vec X_j $. Because $G$ is Abelian, this identification does not depend on the choice of inertial chart $x$, leading to a canonical global coordinate chart $ \vartheta : G \to \mathbb{R}$ such that $ \vartheta(\Lambda) = \theta$ if $ \Lambda$ is identified with $\bm \Lambda_\theta$. We denote the corresponding coordinate basis vector fields on $G$ by $\Theta = \frac{\partial\ }{\partial \vartheta}$. The linear transformations $ \vec x \mapsto \bm \Lambda_\theta \vec x $ on $\mathbb{R}^2$ carried out by $\SO$ are known as \emph{squeeze mappings}, or \emph{hyperbolic rotations}. We shall denote the identity element of $G $ by $ I$. The action of $ G $ on the translation $\vec M $ canonically extends to a linear action on the tangent space $T_m M $ at every $ m \in M $, which we will also denote by $ \Lambda $.    

The Lie algebra of $ G $, denoted by $ \mathfrak{g} $, is isomorphic to the Lie algebra $ \so $ of $ \SO$; the latter, consists of the set of symmetric, real $ 2 \times 2 $ matrices
\begin{displaymath}
    \bm \lambda_\theta = 
    \begin{pmatrix}
        0 & \theta \\
        \theta & 0 
    \end{pmatrix},
    \quad \theta \in \mathbb{R},
\end{displaymath}
where $ \exp(\bm \lambda_\theta) = \bm \Lambda_\theta$. In particular, if $ \lambda \in \mathfrak{g} $ is such that $ \Lambda = \exp(\lambda )$ and $ \vartheta(\Lambda) = \theta$, then $ \lambda $ is identified with $ \bm \lambda_\theta $. Equivalently, $ \lambda $ is equal to $ \theta u$, where $ u $ is the basis vector of $ \mathfrak{g} $ equal to $ \Theta_I $, and we can also write $\theta = u'( \lambda )$, where $u'= d \vartheta_I $. By virtue of these facts, the $\vartheta$ chart has the property
\begin{equation}
    \label{eqThetaChart}\vartheta(\Lambda \Lambda') = \vartheta(\Lambda) + \vartheta(\Lambda'), \quad \forall \Lambda,\Lambda' \in G.
\end{equation}
We equip $ \mathfrak{g} $ with a metric $ b : \mathfrak{g} \times \mathfrak{g} \to \mathbb{R} $ inherited from the coordinate chart $ \vartheta $; specifically, $ b(u,u) = 1$. Note that this metric is canonical, in the sense of being independent of the choice of inertial chart in the construction of  $\vartheta$.

As can be seen directly from~\eqref{eqThetaChart},  $G$ and $\mathfrak g$ are isomorphic as a Lie group and a Lie algebra to $(\mathbb{R},+)$ and $(\mathbb{R},\cdot,+)$, respectively, i.e., the Abelian group and Lie algebra of real numbers equipped with the standard addition and multiplication operations. As a result, we may identify $G$ with the universal covering group $\UU$ of unitary maps on the complex plane.

We denote the left and right multiplication map by $ \Lambda \in G $ by $ L^\Lambda : G \to G $ and $R^\Lambda : G \to G $, respectively; that is, $ L^\Lambda \Lambda' = \Lambda \Lambda' $ and $ R^\Lambda \Lambda' = \Lambda' \Lambda $, for all $ \Lambda' \in G$. Of course, since $ G $ is Abelian, the left and right multiplication maps coincide, $ L^\Lambda \Lambda' = R^\Lambda \Lambda' $, but we prefer to keep these notions distinct so as to better delineate the correspondence between the analysis that follows with analogous constructions in the non-Abelian setting. It should be noted that a fundamental difference between Abelian and non-Abelian groups is that in the former case the adjoint map $ \AD_\Lambda : G \to G$ trivially reduces to the identity for all $ \Lambda \in G$, i.e., $ \AD_\Lambda \Lambda': = \Lambda \Lambda' \Lambda^{-1} = \Lambda'$. 

Besides the left action on itself by multiplication, $G $ has a left, affine action $L^\Lambda_o : M \to M $ on Minkowski space, defined for a fixed origin $ o \in M $  as $ L^\Lambda_o( m) = o + \Lambda \overrightarrow{om}$. The map $ L^\Lambda_o $ is known as a (proper, orthochronous) \emph{Lorentz transformation}. It can be readily verified that for each $ o \in M$,  $ \Lambda \mapsto L^\Lambda_o $ defines a homomorphism of $ G $ into the Poincar\'e group of $M$. Moreover,  for every $\Lambda \in G$, the map $x': M \to \mathbb{R}^2$ with $ x' = x \circ L^\Lambda_o$ is an inertial chart, with coordinate basis vectors $ X'_j = \sum_{i=0}^1 X_i {\Lambda^{-1,i}}_j $, where  ${\Lambda^{-1,i}}_j = \hat X^i \cdot \vec X'_j =  dx^i \cdot \Lambda^{-1} X_j $ are the matrix elements of the inverse transformation $\Lambda^{-1}$ in the $x$ chart. See Lemma~\ref{lemChartX} in Appendix~\ref{appBundleMinkowski} for further details.   

\subsection{\label{secPrincipal}Inertial frame bundle}
Given any $ m \in M$, we let $ P_m $ be the set of future-directed, positively-oriented, orthonormal frames (ordered bases) of $T_m M $. That is, every element $ p \in P_m M $ consists of an ordered pair $ ( p_0, p_1 ) $ of tangent vectors in $ T_m M $, satisfying $ \eta( \tau, p_0 ) < 0 $ (future direction),  $ \nu(p_0,p_1) > 0 $ (positive orientation), and $ [ \eta( p_i, p_j ) ]_{ij} = \diag( -1, 1 ) $ (orthonormality). We will refer to every such $ p $ as an \emph{inertial frame}. On $P_m M $, $ G $ has a free, transitive right action, denoted $ p \cdot \Lambda $, where $ \Lambda \in G $ and 
\begin{displaymath}
    p \cdot \Lambda  = ( \Lambda^{-1} p_0, \Lambda^{-1} p_1 ), 
\end{displaymath}
With some abuse of notation, when convenient, we will use $R^\Lambda : P \to P $ to denote the diffeomorphism induced by right action by group element $ \Lambda\in G $, i.e., $ R^\Lambda( p ) = p \cdot \Lambda $.  

The disjoint union $ P = \bigsqcup_{m\in M} P_m$, endowed with an appropriate smooth manifold structure, then becomes the total space of a principal $G$-bundle $ P \xrightarrow{\pi} M $, where the projection map $ \pi$ maps $ p \in P $ to the underlying base point $ m \in  M$. In differential geometry, such a bundle is known as an \emph{oriented orthonormal frame bundle}; here, we will refer to $P \xrightarrow{\pi}M$ as the principal bundle, or inertial frame bundle, for brevity. 

The kernel of the pushforward map $\pi_{*p} : T_p P \to T_{\pi(p)} M $ on tangent vectors is called the \emph{vertical subspace} at $p \in P$, denoted $V_p P \subset T_p P $. A vector field $W \in \Gamma(TP)$ is said to be \emph{vertical} if $ W_p \in V_p P $ for all $ p \in P$. Given any Lie algebra element $ \lambda \in \mathfrak{g} $, we can construct a vertical vector field $W^\lambda \in \Gamma(TP)$, such that
\begin{displaymath}
    W^{\lambda} f = \lim_{\epsilon\to 0} \frac{f \circ R^{\exp(\epsilon \lambda)} -f}{\epsilon}, \quad \forall f \in C^\infty(P).
\end{displaymath}
Such a vector field is called \emph{fundamental}. In fact, the set of fundamental vector fields is in one-to-one correspondence with $ \mathfrak{g} $; that is, the map $ \lambda \mapsto W^{\lambda} $ has a smooth inverse. Moreover, at any point $ p \in P$, the vertical subspace $ V_p P $ is naturally isomorphic to $ \mathfrak{g} $ under the map $ \lambda \mapsto W^\lambda_p $.  We also note the relationship
\begin{equation}
    \label{eqFundamental}
    W^{\theta u} = \theta W^{u}, \quad \forall \theta \in \mathbb{R},
\end{equation}
where $u$ is the canonical unit basis vector of $\mathfrak g$ from Section~\ref{secMinkowski}.

Due to the existence of a global inertial coordinate chart $ x $ for $M$, $ P \xrightarrow{\pi} M$ admits a global section $ \sigma_x : M \to P$ induced by the coordinate basis vectors, i.e.,  $ \sigma_x(m ) = ( X_0, X_1 )_m $, $ \pi \circ \sigma_x = \Id_M $. This section induces in turn a global trivialization, i.e., a diffeomorphism $ \iota_{\sigma_x} : M \times G \to P $, defined as $ \iota_{\sigma_x}( m, \Lambda ) = \sigma_x( m ) \cdot\Lambda $.  Note that the inverse map $ \iota_{\sigma_x}^{-1}$ satisfies $ \iota_{\sigma_x}^{-1}(p) = (m, \Lambda )$, where $ m = \pi(p ) $, and $ \Lambda $ is the unique element of $ G $ such that $ \sigma_x(m) \cdot \Lambda = p $. The uniqueness of $ \Lambda  $ is a consequence of the fact that the action of $G$ on $ P $ is free. The map $\iota_{\sigma_x}^{-1}$ induces in turn a map $\gamma_{\sigma_x} : P \to G$ from the principal bundle into the structure group, such that $  \Lambda =  \gamma_{\sigma_x}(p) $ is the unique element of $G$ satisfying $\iota_{\sigma_x}^{-1}(p) = (m, \Lambda )$.  It is then straightforward to verify that $\gamma_{\sigma_x}$ is $G$-equivariant, i.e.,  
\begin{equation}
    \label{eqGammaEquiv}
    \gamma_{\sigma_x} \circ R^\Lambda = L^\Lambda \circ \gamma_{\sigma_x}, \quad \forall \Lambda \in G. 
\end{equation}
Together, these results lead to the following commutative diagram for any $ \Lambda \in G $,
\begin{displaymath}
    \begin{tikzcd}
        G & P \arrow{l}{\gamma_{\sigma_x}}  & M \times G \arrow{l}{\iota_{\sigma_x}} \\
        G \arrow{u}{L^\Lambda} & P \arrow{l}{\gamma_{\sigma_x}} \arrow{d}{\pi}\arrow{u}{R^\Lambda}  & M \times G \arrow{l}{\iota_{\sigma_x}} \arrow{ld}{\pi_1} \arrow{u}{\Id_M \times R^\Lambda}  \\
        & M 
    \end{tikzcd},
\end{displaymath}
where $\pi_1$ denotes projection onto the first factor. It is a direct consequence of Lemma~\ref{lemChartX} that for any $ \Lambda \in G $, the trivializing section $ \sigma_{x'} : M \to P $ associated with the Lorentz-transformed chart $ x' = x \circ L^\Lambda_o $ satisfies $ \sigma_{x'} = R^\Lambda \circ \sigma_x$. 

More generally, any local section $ \sigma : U \to P$ (not necessarily associated with an inertial chart as is $ \sigma_x$) defined on a smooth submanifold $U \subseteq M$  induces a local trivialization $ \iota_\sigma : U \times G \to P$ and a map $\gamma_{\sigma}:\pi^{-1}(U) \to G$,  with analogous properties to those of $ \iota_{\sigma_x} $ and $ \gamma_{\sigma_x}$, respectively. By definition, every global section $\sigma$ corresponds to an assignment of an \emph{inertial frame} of $TM$; that is, a positively-oriented pair $ ( e_0, e_1 ) $ of smooth, orthonormal vector fields in $\Gamma(TM)$, of which $ e_0 $ is timelike and future-directed,  such that $ ( e_{0,m}, e_{1,m} ) = \sigma(m) $ at every $ m \in M $. Note, however, that $ e_0 $ and $ e_1  $ may not be coordinate vector fields; in particular, they may fail to commute. 

A coordinate-induced section $\sigma_x$ also induces a global coordinate chart $ y : P \to \mathbb{R}^3 $ on the principal bundle, given by $ y(p) = (y^0, y^1, y^2 ) = ( x^0( m ), x^1( m ), \vartheta( \Lambda ) ) $, where $m = \pi(p)$ and $\Lambda = \gamma_{\sigma_x}(p)$. The corresponding coordinate basis vector fields, $ Y_j := \frac{\partial\ }{\partial y^j}$, $ j \in \{ 0, 1, 2 \} $,  are then lifts of  $ X_0 $, $ X_1 $, and $ \Theta $, respectively, i.e., $ \pi_* Y_0 = X_0$, $\pi_* Y_1 = X_1$, and $ \gamma_{\sigma_x * } Y_2 = \Theta$. These vector fields are all invariant under the $G$ action on $ P$, i.e., $R^\Lambda_* Y_j = Y_j$, and $Y_2 $, in particular, is a fundamental field generated by Lie algebra element $ u$. See Lemma~\ref{lemChartY} in Appendix~\ref{appBundleMinkowski} for further details. The section $\sigma_x$ also induces a metric tensor $ \tilde \eta_p : T_p P \times T_p P \mapsto{R}$, given as a pullback of the direct sum metric $ \eta \oplus b $ on $ M \times G $ under $ \iota^{-1 *}_{\sigma_x} $. This metric has the matrix representation $ [ \tilde \eta ( Y_i, Y_j ) ]_{ij} = \diag(-1,1,1 ) $. 

In gauge theory, a local section of the principal bundle is known as a \emph{gauge}, and  oftentimes, the task is to reconstruct global objects (e.g., connection 1-forms) from their behavior on local sections, as well as to study the behavior of these objects under gauge transformations (to be discussed in Section~\ref{secTransf}). As noted above, in the context of the present work, a gauge corresponds to a smooth assignment of an orthonormal basis to tangent spaces of Minkowski space, which can intuitively be thought of as a local choice of inertial frame. Our task is then to construct a \emph{gauge-covariant} quantum mechanical system, i.e., a system transforming naturally under changes of local inertial frame.  

\subsection{\label{secAssoc}Associated $\mathbb{C}$-line bundle}

Let $ \GLC $ be the Lie group of invertible linear maps on the complex plane, and  $ \glc $ its associated Lie algebra. The group $ \GLC $ is canonically isomorphic to the multiplicative group of nonzero complex numbers, so that $ w \in \GLC $ acts linearly on $ z \in \mathbb{C} $ by multiplication, $ z \mapsto w z $. Meanwhile, the Lie algebra $ \glc $ is canonically isomorphic to the vector space of complex numbers, and we have the exponential map $ \exp : \glc \to \GLC $, where $ \exp( \theta ) = e^{\theta}$ for any $ \theta \in \glc$. We equip $\mathbb{C}$ with the standard inner product, $ \langle z_1, z_2 \rangle_{\mathbb{C}} = z_1^* z_2$. 

To construct the $\mathbb{C}$-line bundle associated with the principal bundle from Section~\ref{secPrincipal}, we start from the Lie algebra  representation  $\varrho : \mathfrak{g} \to \mathfrak{gl}(1,\mathbb{C})$, defined as 
\begin{displaymath}
    \varrho(\lambda) = i \alpha \vartheta(\lambda) / \sqrt{2}. 
\end{displaymath}
Here, $ \alpha $ is a real parameter (which will be set in Section~\ref{secIso} below equal to the frequency of the classical oscillator from Section~\ref{secClassical}), so that the representation $ \varrho$ is skew-adjoint,
\begin{displaymath}
    \varrho(\lambda)^* = - \varrho(\lambda) = \varrho(-\lambda), \quad \forall \lambda \in \mathfrak{g}.
\end{displaymath}
Expressed in terms of matrices, $ \varrho$ acts by extracting from the hollow, symmetric matrix $ \bm\lambda_\theta$ representing $\lambda$ the value $ \theta = d\vartheta_I\lambda $ in its off-diagonal elements, and multiplying that value by the imaginary number $i \alpha / \sqrt{2}$. The factor of $1/\sqrt{2}$ is introduced here for later convenience. 

Using $\varrho $, we then construct a group representation  $\rho : G \to \GLC$, making use of the fact that the exponential map $ \exp: \mathfrak{g} \to G $ for $ G $ is a actually a diffeomorphism, whose inverse, $ \log : G \to \mathfrak{g} $, allows one to recover the unique Lie algebra element underlying a given group element. Based on these facts, we define the unitary representation
\begin{displaymath}
    \rho = \exp \circ \varrho \circ \log, \quad \ran \rho = \UU, 
\end{displaymath}
where $ \rho(\Lambda) = e^{i \alpha \vartheta(\Lambda)/\sqrt{2}} $, and the differential at the identity recovers $\varrho$, 
\begin{displaymath}
    \rho_{*,I} = \varrho.
\end{displaymath}
Note that, unlike $ \exp: \mathfrak{g} \to G $, the exponential map on $ \glc $ is not injective, and therefore $ \rho $ is not a faithful (injective) representation. In fact, viewed as a map from $G\simeq\SO$ to $\UU$, $ \rho $ becomes the universal covering map for $\UU$.   

The representation $\rho$ induces a left action on $\mathbb{C}$, denoted 
\begin{equation}
    \label{eqLeftAction}
    L^\Lambda z \equiv \Lambda\cdot z = \rho( \Lambda ) z = e^{i\alpha\vartheta(\Lambda)/ \sqrt{2}} z 
\end{equation}
for any $ \Lambda \in G $ and $ z \in \mathbb{C} $. Together with the $G$ action on the total space $ P $ of the principal bundle, this action induces an equivalence relation $ \sim $ on $ P \times \mathbb{C} $, whereby $( p, z ) \sim (p', z')$ if there exists $\Lambda \in G$ such that $(p', z') = ( p \cdot \Lambda, \Lambda^{-1} \cdot z) $. We then define $E $ as the set of equivalence classes in $ P \times \mathbb{C}$ under this equivalence relation. It can be readily verified that the projection map $ \pi_E : E \to M $, sending equivalence class $ [ p, z ] \in E $ to $ \pi(p) $, is well-defined. Moreover, letting $E_m = \pi^{-1}_E(\{m\})$ denote the fiber in $E$ over $ m \in M $, and $p $ an arbitrary point in $E_m $, it can be shown that the map $\varepsilon_p : \mathbb{C} \to E_m $ with $ \varepsilon_p( z ) = [ p, z] $ is a bijection, and the property $ \epsilon^{-1}_{R^\Lambda(p)} = L^{\Lambda^{-1}} \circ \varepsilon^{-1}_p  $ holds (see Lemma~\ref{lemUnique} in Appendix~\ref{appBundle}). The tuple $ E \xrightarrow{\pi_E} M$ is therefore an \emph{associated vector bundle} to $P \xrightarrow{\pi}M$ over the complex numbers, with typical fiber $  \mathbb{C}$. In particular, for any $ p \in P$ and $ z, c \in \mathbb{C}$, the scalar multiplication  $ c [ p, z ] = [ p, c z ] $ and complex conjugation $ [ p, z ]^*  = [ p, z^* ] $ are well defined operations. 

Next, consider the space of smooth sections of $E$, $\Gamma(E)$. Using any global section $ \sigma : M \to P $ from Section~\ref{secPrincipal},  we can construct a trivialization of $ E \xrightarrow{\pi_E} M$ analogously to that of the principal bundle, viz. 
\begin{displaymath}
    \begin{tikzcd}
        E\arrow{d}{\pi} & M \times \mathbb{C} \arrow{l}{\iota_{\sigma,E}} \arrow{ld}{\pi_1}  \\
        M
    \end{tikzcd}.
\end{displaymath}
Here, $ \iota_{\sigma,E} : M \times \mathbb{C} \to E$ is the diffeomorphism with $ \iota_{\sigma,E}(m,z) = [ \sigma( m ), \varepsilon_{\sigma(m)}(z ) ] $ and inverse $\iota_{\sigma,E}^{-1}([p,z]) = ( \pi(p), \varepsilon_{\sigma(\pi(p))}^{-1}([p,z]) )$. Moreover, $\sigma $ induces a Hermitian metric on $ E$; that is, a smooth assignment $ m \mapsto g_m $, where $ m \in M $ and $ g_m $ is a positive-definite sesquilinear form on $E_m $. Explicitly, given any $ m \in M$ and  $ e_1, e_2 \in E_m $, we have 
\begin{equation}
    \label{eqG}
    g_m(e_1,e_2) = \langle \varepsilon_{\sigma(m)}(e_1), \varepsilon_{\sigma(m)} ( e_2 ) \rangle_{\mathbb{C}}.
\end{equation}
The complex structure of $E$ induces a complex structure $J: \Gamma(E) \to \Gamma(E)$ on sections, such that $ (J s)(m) = (s(m))^*$.

It can be verified that, due to the unitarity of $ \rho $, $ g_m $ is independent of the choice of trivializing section $ \sigma $ (see Lemma~\ref{lemG} in Appendix~\ref{appBundle}), i.e., it is \emph{gauge-invariant}. We can therefore define a Hilbert space $\mathcal{H}$ of sections of $E$ (modulo sets of zero Riemannian measure), equipped with the inner product $ \langle s_1, s_2 \rangle_\mathcal{H} = \int_M g_m(s_1(m),s_2(m)) \, d\nu(m)$ and norm $ \lVert s \rVert_\mathcal{H} = \sqrt{\langle s, s \rangle_\mathcal{H}}$, in a manner that does not depend on the choice of gauge.  We will employ this Hilbert space in Section~\ref{secHamiltonian} to define quantum mechanical observables as self-adjoint operators acting on it.  

The trivializing section $\sigma$ also induces an isomorphism $ \zeta_\sigma :  C^\infty(M) \to \Gamma(E) $ between smooth, complex-valued functions on $M$ and sections of $\Gamma(E)$, such that  
\begin{displaymath}
    \zeta_\sigma f(m) = \varepsilon_{\sigma(m)}( f(m)), \quad \zeta_{\sigma}^{-1} s(m) = \varepsilon_{\sigma(m)}^{-1}(s(m)).
\end{displaymath}
This extends in turn to a Hilbert space isomorphism between equivalence classes of functions in $L^2(\nu)$ and sections in $\mathcal{H}$, which is very convenient for calculational purposes, and will also facilitate establishing the correspondence with Koopman operator theory for the classical harmonic oscillator. In gauge theory, pullbacks of sections of associated bundles, such as $\zeta_{\sigma}^{-1} s$ are known as \emph{matter fields}. As will be discussed in more detail in Section~\ref{secTransf}, matter fields have the distinguished property of being \emph{gauge-covariant}, i.e., they transform naturally under changes of section $\sigma$, and thus under changes of inertial frame in the present context. In particular, this is contrast to arbitrary $C^\infty(M)$ functions which have no natural transformation properties with respect to $\sigma$.

Before closing this section, we note another key property of sections in $ \Gamma(E) $, namely, that they are in one-to-one correspondence with $G$-equivariant, $\mathbb{C}$-valued functions on the principal bundle.  In particular, a function $ f \in C^\infty(P)$ is said to be \emph{$G$-equivariant} if $ f \circ R^\Lambda = L^\Lambda \circ f $ for any $ \Lambda \in G$. We denote the vector space of all such functions by $C^\infty_{G}(P)$. It can be verified that the map $ \beta: C^\infty_{G}(P) \to \Gamma(E) $ with $ \beta f(m) = [ p, f(p) ] $, where $ p $ is an arbitrary element of $ \pi^{-1}(m)$, is well-defined, and possesses a smooth inverse given by $ \beta^{-1} s(p) = \varepsilon_{p}^{-1}(s(\pi(p)))$. 

\subsection{\label{secConn}Connection 1-form and covariant derivative}

We now have the necessary ingredients to construct the connection 1-form on the principal bundle for our gauge theory, as well as the induced gauge field on the base space  and covariant derivative on sections of the associated bundle.   

In the Abelian setting under study, a \emph{connection 1-form} on the principal bundle $P \xrightarrow{\pi} M$ is a Lie-algebra-valued 1-form $\omega \in \Omega^1(P, \mathfrak{g})$, such that for every point $p \in P$, group element $ \Lambda \in G$, and fundamental vector field $W^\lambda \in \Gamma(TP)$, the conditions 
\begin{equation}
    \label{eqConnCond} ( \omega W^\lambda )_p = \lambda, \quad R^{\Lambda*} \omega = \omega 
\end{equation}
hold. That is, $\omega$ recovers the Lie algebra elements generating the fundamental vector fields, and is also $G$-invariant. Note that the invariance condition (the second equation in~\eqref{eqConnCond}) is specific to Abelian groups, and is replaced by a more general equivariance condition in the non-Abelian setting; see~\eqref{eqAdj} in Appendix~\ref{appBundle}. Given a section $ \sigma : M \to P $, the pullback $\omega^\sigma := \sigma^* \omega $ of the connection 1-form onto the base space $M$ is known as a \emph{gauge field}. A connection 1-form endows the principal bundle, as well as its associated bundles, with important geometrical structure, including the assignment of a horizontal distribution and a corresponding notion of parallel transport of curves from the base space to the total space. See Appendix~\ref{appBundleGeneral} and \cite{BerlineEtAl04,Michor08,RudolphSchmidt17} for further details on these topics. 

For our purposes, a key implication of the connection 1-form is that it induces  covariant derivative operators on the principal bundle, as well as the Minkowski base space. Specifically, associated with a given connection 1-form $\omega $, a vector field $X \in \Gamma(TM)$, a vector field $ Y \in \Gamma(TP) $, and a section $\sigma : M \to P $ are:
\begin{enumerate}
    \item An \emph{exterior covariant derivative} $ \mathcal{D}_Y: C^\infty(P) \to C^\infty(P)$, acting on  $\mathbb{C}$-valued functions on the principal bundle and preserving the space of $G$-equivariant functions $C^\infty_G(P)$; 
    \item A \emph{covariant derivative} $ \nabla_X : \Gamma(E) \to \Gamma(E)$, acting on sections of the associated bundle;
    \item A \emph{covariant derivative} $ \nabla^\sigma_X : C^\infty(M) \to C^\infty(M)$, acting on $\mathbb{C}$-valued functions on the base space.
\end{enumerate}
Among these, the exterior covariant derivative of $ f \in C^\infty(P) $ with respect to $Y $  is given by
\begin{equation}
    \label{eqExtCov}
    \mathcal{D}_Y f = df \cdot  \hor Y,
\end{equation}
where $ d: C^\infty(P) \to \Omega^1(P) $ is the canonical exterior derivative, and $ \hor : \Gamma(TP) \to \Gamma(TP)$ denotes the \emph{horizontal projection map} associated with $ \omega $:
\begin{equation}
    \label{eqHor}
    \hor Y_p = Y_p - W_p^{(\omega Y)_p}.  
\end{equation}
Setting $Y $ to the \emph{horizontal lift} of $X$ on $P$ (see Appendix~\ref{appBundleGeneral}), the covariant derivative operators $ \nabla_X$ and $ \nabla_X^\sigma$ on base space are then obtained by means of the following commutative diagram:   
\begin{displaymath}
    \begin{tikzcd}
        C^\infty_{G}(P) \arrow{r}{\mathcal{D}_{Y}} \arrow[shift left]{d}{\beta}  & C^\infty_{G}(P) \arrow[shift left]{d}{\beta} \\
        \Gamma(E)\arrow{r}{\nabla_X}\arrow[shift left]{u}{\beta^{-1}} \arrow[shift left]{d}{\zeta_{\sigma}^{-1}} & \Gamma(E) \arrow[shift left]{u}{\beta^{-1}} \arrow[shift left]{d}{\zeta_\sigma^{-1}} \\
        C^\infty(M) \arrow[shift left]{u}{\zeta_\sigma}\arrow[shift left]{r}{\nabla^{\sigma}_X} & C^\infty(M)\arrow[shift left]{u}{\zeta_\sigma}
    \end{tikzcd}.
\end{displaymath}
That is, we have 
\begin{equation}
    \label{eqCov0}
    \nabla_X = \beta \circ \mathcal{D}_Y \circ \beta^{-1}, \quad \nabla_X^\sigma = \zeta_\sigma^{-1} \circ \nabla_X \circ \zeta_\sigma.
\end{equation}
It can be shown that the covariant derivative $ \nabla^\sigma_X $ takes the form
\begin{equation}
    \label{eqCov}
    \begin{aligned}
        \nabla_X^\sigma f &= df \cdot X + \varrho(\omega^\sigma X) f \\
        &= X f+ \varrho(\omega^\sigma X) f, 
    \end{aligned}
\end{equation}
where $ \varrho(\omega^\sigma X)$ is a shorthand notation for the $\glc$-valued function $m \mapsto \varrho( (\omega^\sigma X)_m)$ on $M$. This expression is particularly useful for calculational purposes as it involves ordinary complex-valued functions on the base space. 
 
It is evident from~\eqref{eqExtCov} and~\eqref{eqCov} that a covariant derivative deviates from the standard exterior derivative through a ``correction'' that depends on the connection 1-form, and in the case of $ \nabla^\sigma_X $, that correction presents itself through the gauge field $ \omega^\sigma $ and Lie algebra representation $\varrho$.  It should be kept in mind that it is the triviality of $P \xrightarrow{\pi} M $ that allows us to work with  globally defined gauge fields. More generally, $ \omega^\sigma $ would only be defined locally on the domain of definition of $ \sigma $, and~\eqref{eqCov} would be valid locally on the same domain with appropriate compatibility conditions fulfilled in the overlap regions between domains of different local sections.   

To construct our connection 1-form $ \omega$, we begin by fixing an inertial chart $x : M \to \mathbb{R}^2$ with origin $o \in M$, and introducing a linear map $^\odot : T_pP \to T_pP $, $ p \in P $, on the tangent spaces of the principal bundle, characterized uniquely through the relationships
\begin{equation}
    \label{eqCirc}
    Y_{0,p}^\odot = - Y_{1,p}, \quad Y_{1,p}^\odot = - Y_{0,p}, \quad  Y_{2,p}^\odot = Y_{2,p},
\end{equation}
where $Y_j $ are the basis vector fields of the associated chart $ y : P \to \mathbb{R}^3$ to $x$. That is, for any $ Y \in T_pP $ we have
\begin{displaymath}
    Y^\odot = \sum_{j=0}^2 ( dy^j_p \cdot Y ) Y_{j,p}^\odot,
\end{displaymath}
and it can be verified that this definition is independent of the choice of inertial chart $x$ (see Lemma~\ref{lemCirc}). Note that $ \pi_{*,p}(Y^\odot) = ( \pi_{*,p} Y )^\perp $, so that $ ^\odot $ can be interpreted as a lift of the $ ^\perp $ operator from Section~\ref{secMinkowski} to the principal bundle. In particular, the notation $^\odot$ is suggestive of the fact that $Y^\odot$ has perpendicular components to $ Y$ in the $Y_0$ and $ Y_1 $ directions, but the same component in the $Y_2 $ direction, so that $Y_2$ can be thought of as an axis of symmetry remaining unchanged under $^\odot$.  

Next, given the origin $ o \in M $ of the inertial chart $x $, we consider the quadratic form $ h : M \mapsto \mathbb{R} $, defined as 
\begin{equation}
    \label{eqPot}
    h(m) =  \eta(\overrightarrow{om}, \overrightarrow{om}) / 2.
\end{equation}
Explicitly, in terms of the coordinates $x(m) = (x^0(m),x^1(m))$, we have 
\begin{displaymath}
    h( m ) = [ - (x^{0}(m))^2 + (x^1(m))^2 ] / 2,
\end{displaymath}
from which it follows that $ h(m)$ vanishes for points $m$ with null translations from $o$, and is negative (positive) for points with timelike (spacelike) translations.  With the help of the map $ \gamma_{\sigma_x} : P \to G $ associated with the chart, we also introduce the function $ \tilde h_x : P \to \mathbb{R} $ on the principal bundle given by
\begin{displaymath}
    \tilde h_x(p) =h( \pi( p ) ) + \vartheta( \gamma_{\sigma_x}(p) ). 
\end{displaymath}

We then define the Lie-algebra-valued 1-form $ \omega \in \Omega^1( P, \mathfrak{g}) $ with
\begin{equation}
    \label{eqConn}
    ( \omega Y )_p = ( Y^\odot \tilde h_x)_p  u,  \quad \forall Y \in \Gamma(TP), \quad \forall p\in P.
\end{equation}
One can then verify that $\omega $ satisfies~\eqref{eqConn}, and is thus a connection 1-form on the principal bundle  (see Proposition~\ref{propConn} in Appendix~\ref{appBundleMinkowski}). Moreover, $\omega$ does not depend on the choice of inertial chart $x$ with origin $o$. As will be discussed in more detail in Section~\ref{secTransf}, by virtue of these facts, the corresponding operators $\nabla_X$ and $ \nabla_X^\sigma$ in~\eqref{eqCov} are $G$-covariant, Lorentz-invariant derivatives.

It is also straightforward to check that the action of $\omega $ on the coordinate vector fields $Y_j $ of the chart $y : P \to \mathbb{R}^3$ induced by $ x$ on the principal bundle takes the form
\begin{equation}
    \label{eqConnComponents}
    (\omega Y_0)_p = - y^1(p) u, \quad (\omega Y_1)_p = y^0(p) u, \quad (\omega Y_2)_p = u, 
\end{equation}
leading to the expression
\begin{align}
    \label{eqConnHor}\hor Y &= ( dy^0 \cdot Y ) Y_0 + ( dy^1 \cdot Y) Y_1 \\
    & \quad + ( (dy^0 \cdot Y) y^1 - (dy^1 \cdot Y) y^0) Y_2
\end{align}
for the associated horizontal projection map on vector fields. Note that, in general, $ \hor Y$ has a nonzero component along the $ y^2 $ coordinate, despite the fact that $Y_2$ is a vertical vector field.  

Next, pulling back $\omega$ to the base space along the section $\sigma_x$ yields the gauge field $ \omega^{\sigma_x} $, where
\begin{displaymath}
    (\omega^{\sigma_x} X_0)_m = - x^1(m) u, \quad (\omega^{\sigma_x} X_1)_m = x^0(m) u. 
\end{displaymath}
Correspondingly, the covariant derivative $ \nabla^{\sigma_x}$ on complex-valued functions on Minkowski space is found to satisfy
\begin{displaymath}
    \nabla^{\sigma_x}_{X_0} f  = X_0 f -i\frac{ \alpha}{\sqrt{2}} x^1 f , \quad \nabla^{\sigma_x}_{X_1} f = X_1 f + i \frac{\alpha}{\sqrt{2}} x^0 f.
\end{displaymath}
These relationships can be expressed in a coordinate-free manner as 
\begin{displaymath}
    \omega^{\sigma_x} X = X^\perp h u, \quad \nabla^{\sigma_x} = d + i \frac{\alpha}{\sqrt{2}} X^\perp h,   
\end{displaymath}
where it is evident that $\omega^{\sigma_x}$ (and thus $\nabla^{\sigma_x}$) is independent of $\sigma_x$. Henceforth, we will use the notations $\omega^M \equiv \omega^{\sigma_x}$ and $ \nabla^M \equiv \nabla^{\sigma_x} $ to highlight that independence. It can also be readily verified that for all $ X \in \Gamma(TM)$ and $s_1, s_2 \in \Gamma(E)$ the property
\begin{displaymath}
    X g(s_1, s_2) = g(\nabla_X s_1, s_2) + g(s_1, \nabla_X s_2)
\end{displaymath}
holds, which implies that $\nabla_X$ is a \emph{metric covariant derivative}. Moreover, $ \nabla_X $ is a formally skew-symmetric operator with respect to the Hilbert space inner product of $\mathcal{H}$; that is, $ \langle \nabla_X s_1, s_2 \rangle_{\mathcal{H}} = - \langle s_1, \nabla_X s_2 \rangle_{\mathcal{H}}$ for all $ X \in \Gamma(TM)$ and compactly supported sections $ s_1, s_2 \in \Gamma_c(E)$. Using the notation $ \nabla_X^+$ to represent the formal adjoint of $ \nabla_X $ with respect to the $\langle \cdot, \cdot \rangle_{\mathcal H}$ inner product, we have $ \nabla_X^+ = - \nabla_X$.   

Before closing this section, we note that, by $C^\infty(M) $-linearity of the dependence $X \mapsto \nabla_X s$, we can lift $\nabla_X$ to an operator $ \nabla : \Gamma(E) \to \Omega^1(M,E) = \Gamma( T^*_\mathbb{C}M \otimes_\mathbb{C} E)$, such that 
\begin{equation}
    \label{eqCovSX}
    \nabla s \cdot X = \nabla_X s.
\end{equation}
Note, in particular, that $ T^*_\mathbb{C} M \otimes_{\mathbb{C}} E \to M $ is isomorphic as a bundle to the vector bundle $ \Hom( T_{\mathbb C}M, E ) \to M $ of bundle homomorphisms from $T_{\mathbb C}M $ to $E$. Thus,  $ \nabla s $ can act on a vector field $X \in \Gamma(T_{\mathbb C}M) $, giving \eqref{eqCovSX}. Similarly, the exterior covariant derivative $\mathcal D_Y$ on the principal bundle lifts to an operator $\mathcal D : C^\infty(P) \to \Omega^1(P)$, and for any vector space  $F$, we can define an exterior covariant derivative $ \mathcal D : C^\infty(P,F) \to \Omega^1(P,F)$ analogously to~\eqref{eqExtCov}. In what follows, for simplicity of notation we will suppress $\mathbb C$ subscripts from $T_{\mathbb C}M$, $T^*_{\mathbb C} M$, and $\mathbb \otimes_{\mathbb C}$.    
    
\subsection{\label{secLapl}Laplacians on the associated bundle}

We now describe the construction of the Laplace-type operator which will play the role of a quantum mechanical Hamiltonian operator acting on sections in $\Gamma(E)$, and discuss its relationship to the connection Laplacian induced by the connection in Section~\ref{secConn}. 

First, given a vector field $ X \in \Gamma(TM)$,  we introduce the complex-conjugate adjoint covariant derivative $ \bar \nabla_X : \Gamma(E) \to \Gamma(E)$, 
\begin{displaymath}
    \bar \nabla_X = J \circ \nabla_X^+ \circ J = - J \circ \nabla_X \circ J,
\end{displaymath}
where $J : \Gamma(E) \to \Gamma(E) $ is the complex structure on sections. Given a trivializing section $ \sigma \in \Gamma(P)$ of the principal bundle, this operator has a representation $ \bar \nabla_X^M : C^\infty(M) \to C^\infty(M)$ as a differential operator on complex-valued functions on $M$, where $ \bar \nabla^M_X = \zeta_\sigma^{-1} \circ \bar \nabla_X \circ \zeta_\sigma$, and
\begin{displaymath}
    \bar \nabla^M_X f = df \cdot X + \varrho(\omega^M X)^* f = df \cdot X - \varrho(\omega^M X) f. 
\end{displaymath}
If, in particular, $X_0 $ and $X_1$ are the coordinate vector fields of an inertial chart $ x $ centered at $o$, we have
\begin{displaymath}
    \bar \nabla^{M}_{X_0} f = X_0 f + i \frac{ \alpha}{\sqrt{2}} x^1 f, \quad \bar \nabla^{M}_{X_1} f = X_1 f - i \frac{\alpha}{\sqrt{2}} x^0 f. 
\end{displaymath}
As with $ \nabla_X$, $\bar \nabla_X $ can be extended  to an operator $ \bar \nabla : \Gamma(E)\to \Omega^1(M, E)$ mapping sections to 1-forms. 

Next, let $\nabla^{\text{LC}} : \Gamma(TM) \to \Omega^1(M)$ be the Levi-Civita connection associated with the Minkowski metric $ \eta $. Together, $ \bar \nabla$ and $\nabla^\text{LC}$ induce a \emph{tensor covariant derivative} on the tensor product bundle $ T^*M \otimes E$, that is, a covariant derivative operator $ \tilde \nabla : \Gamma(T^*M \otimes E) \to \Omega^1( M, T^*M \otimes E ) = \Gamma( T^*M \otimes T^*M \otimes E ) $, such that for any two vector fields, $ X, Y \in \Gamma(TM) $ and any section $ T \in \Gamma(T^*M \otimes E)$ 
\begin{equation}
    \label{eqCovLC}
    ( \tilde \nabla_X T ) Y = \bar \nabla_X( T \cdot Y ) - T( \nabla_X^\text{LC} Y).
\end{equation}
Note that this definition is consistent with the Leibniz rule, $ \bar \nabla_X(T \cdot Y ) = ( \tilde \nabla_X T )  Y + T( \nabla_X^\text{LC} Y ) $. The composition of $ \tilde \nabla $ and $ \nabla $ then leads to the \emph{second covariant derivative} $ \tilde{\mathsf H} : \Gamma(E) \to \Gamma(T^*M \otimes T^*M \otimes E)$, where $ \tilde{\mathsf H} = \tilde \nabla \circ \nabla$. It follows from~\eqref{eqCovLC} that for any $ X, Y \in \Gamma(TM)$ and $ s \in \Gamma(E)$, we have
\begin{align*}
    \tilde{\mathsf H}_{X,Y} s &\equiv ( \tilde{\mathsf H} s )(X,Y) \\
    &= ( \tilde \nabla_X \nabla s ) \cdot Y = \bar \nabla_X \nabla_Y s - \nabla_{\nabla^{\text{LC}}_X Y} s.
\end{align*}
Using Riemannian isomorphisms, we also define a second covariant derivative $\tilde{\mathsf H}' : \Gamma(E) \to \Gamma( TM \otimes TM^* \otimes E) $, such that for any vector field $X \in \Gamma(TM)$, 1-form $w \in \Gamma(T^*M)$, and section $ s \in \Gamma(E)$, 
\begin{displaymath}
    \tilde{\mathsf H}'_{X,w} s = \tilde{\mathsf H}_{X,w^\sharp} s. 
\end{displaymath}

Observe now that $ TM \otimes TM^* \rightarrow M $ is isomorphic as a bundle to the vector bundle $ \End(TM) \rightarrow M $ of endomorphisms of the tangent bundle. In particular, given  $ s \in \Gamma(E) $ and $ m \in M$,  $ ( \tilde{\mathsf H}' s )_m $ can be viewed as a linear map $ L : T_m M \to T_m M $ such that for any vector field $ X \in \Gamma( TM ) $ and 1-form $ w \in \Gamma( T^* M ) $, $ w_m  ( L X_m ) = ( \tilde{\mathsf H}'_{X,w} s)_m $.    We can therefore reduce $ \tilde{\mathsf H}'$ to an operator $ \tilde \Delta : \Gamma(E) \to \Gamma(E) $ on sections by taking the trace of $  ( \tilde{\mathsf H}' \cdot )_m $ at every point $ m \in M $, i.e.,  
\begin{displaymath}
    ( \tilde \Delta s )_m = \tr ( \tilde{\mathsf H}' s )_m.
\end{displaymath}
That is, in any orthonormal frame $ \{ X_0, X_1 \} $ of $ TM $ with $ \eta(X_0, X_0 ) = -1 $ and $ \eta(X_1,X_1) = 1$, we have
\begin{displaymath}
    \tilde \Delta  =  - \left( \bar \nabla_{X_0} \nabla_{X_0} - \nabla_{\nabla^{\text{LC}}_{X_0} X_0} \right) + \left( \bar \nabla_{X_1} \nabla_{X_1} - \nabla_{\nabla^{\text{LC}}_{X_1} X_1} \right),   
\end{displaymath}
and if the $X_i $ are coordinate vector fields associated with an affine coordinate chart (as is nominally the case), the covariant derivatives $ \nabla^{\text{LC}}X_i $ vanish, leading to the simpler expression
\begin{equation}
    \label{eqLaplSect}
    \tilde \Delta  =  \bar \nabla_{X_0} \nabla_{X_0} - \bar \nabla_{X_1} \nabla_{X_1}.  
\end{equation}

To characterize this operator more explicitly, it is useful to consider its representation $ \tilde \Delta^{M} : C^\infty(M) \to C^\infty(M)$ on complex-valued functions on $M$ induced by the trivializing section $ \sigma_x $, i.e., $ \tilde \Delta^{M} = \zeta_{\sigma_x}^{-1} \circ \tilde \Delta \circ \zeta_{\sigma_x}$. A direct calculation then yields 
\begin{equation}
    \label{eqLaplSigma}
    \tilde \Delta^{M}  = - \left( - X_0^2 + \frac{\alpha^2}{2} ( x^0 )^2 \right) + \left( - X_1^2 + \frac{\alpha^2}{2} (x^1)^2  \right).  
\end{equation}
Noticing that the coordinate vector fields $X_0$ and $X_1$  are formally skew-symmetric as operators on functions, i.e., $ \langle f_1, X_j f_2 \rangle_{\nu} = - \langle X_j f_1, f_2 \rangle_{\nu}$ for all $f_1, f_2 \in C^\infty_c(M)$, it follows from~\eqref{eqLaplSigma} that $ \tilde \Delta^{M}$ (and thus $ \tilde \Delta$) is formally symmetric, $ \langle f_1, \tilde \Delta^{M} f_2 \rangle_{\nu} = \langle \tilde \Delta^{M} f_1, f_2 \rangle_{\nu} $. Moreover, it is evident that, up to a proportionality factor of $1/2$,  $ \tilde \Delta^{M} $ has the structure of the difference between the Hamiltonians of two one-dimensional quantum harmonic oscillators of frequency $ \alpha $, operating along the $ x^0 $ and $x^1 $ coordinates, respectively. We will pursue this correspondence further in Section~\ref{secIso}, where we will use a (self-adjoint extension of) $\tilde \Delta / 2$ as the Hamiltonian of a quantum mechanical theory with $\mathcal{H}$ as its Hilbert space.      

For now, we express $ \tilde \Delta $ in an alternative form, which also exemplifies its relationship with a harmonic oscillator potential. For that, we introduce the \emph{connection Laplacian} $ \Delta : \Gamma(E) \to \Gamma(E)$ on sections of the associated bundle. This operator is defined analogously to $ \tilde \Delta $ as
\begin{displaymath}
    (\Delta s )_m = \tr(\mathsf H' s)_m, \quad \forall m \in M,
\end{displaymath}
where $ \mathsf H' : \Gamma(E) \to \Gamma(TM \otimes TM^* \otimes E) $ is obtained via Riemannian isomorphisms from the operator $ \mathsf H : \Gamma(E) \to \Gamma(TM^* \otimes TM^* \otimes E)$, $ \mathsf H = \hat \nabla \circ \nabla$. Here, $ \hat \nabla : \Gamma(T^*M \otimes E ) \to \Gamma(T^*M \otimes T^*M \otimes E)$ is a tensor covariant derivative, defined analogously to~\eqref{eqCovLC} as 
\begin{displaymath}
    (\hat \nabla_X T) Y = \nabla_X(T \cdot Y) - T(\nabla^{\text{LC}}_X Y), \quad \forall X, Y \in \Gamma(TM).  
\end{displaymath}
In other words, the difference between $ \Delta $ and $ \tilde \Delta $ is that the former is constructed without using complex conjugation and adjoints. 

A calculation analogous to that used to obtain~\eqref{eqLaplSigma} then shows that the function  Laplacian $ \Delta^{M} := \zeta_{\sigma_x}^{-1} \circ \Delta \circ \zeta_{\sigma_x}$ induced by $\sigma_x $ is given by
\begin{equation}
    \label{eqLaplConn}
    \Delta^{M} = - \left( - X_0^2 - \frac{\alpha^2}{2} ( x^0 )^2 \right) + \left( - X_1^2 - \frac{\alpha^2}{2} (x^1)^2 \right).  
\end{equation}
Comparing~\eqref{eqLaplSigma} and~\eqref{eqLaplConn}, it then follows that
\begin{displaymath}
    \tilde \Delta^{M} = \Delta^{M} + \alpha^2( - (x^0)^2 + (x^1)^2 ), 
\end{displaymath}
or, equivalently,
\begin{equation}
    \label{eqDeltaV}
    \tilde \Delta  = \Delta + 2 v, \quad v =  \alpha^2 h,   
\end{equation}
where $ h$ is the quadratic form from~\eqref{eqPot}. Thus, if $ \Delta/2 $ is interpreted as  a ``free-particle'' Hamiltonian, then $ \tilde \Delta/2 $ corresponds to the Hamiltonian for a particle in a quadratic potential $v$, with frequency parameter $\alpha$. Note the distinguished role that the origin $o$ of the coordinate chart $ x $ plays in this definition, as it is the unique saddle point of $v$.

\subsection{\label{secYM}Curvature and Yang-Mills equations}

Thus far, we have identified a connection 1-form $\omega$ on the inertial frame bundle $P \xrightarrow{\pi} M $, which naturally leads to geometrical Laplace-type operators having the structure of quantum mechanical Hamiltonians with quadratic potentials. Yet, the specific choice of $\omega$ in~\eqref{eqConn} may appear somewhat ad hoc, especially given the fact that the space of connections on a principal bundle is an infinite-dimensional (i.e., ``large'') space. More specifically, it can be shown that the space of connections, $\mathcal C$, of a general principal bundle $ P \xrightarrow{\pi} M$ is an infinite-dimensional affine space whose translation group is isomorphic to a vector bundle $\Ad P \xrightarrow{\pi_{\Ad P}} M $ over $M$, with typical fiber isomorphic to the Lie algebra $\mathfrak g $, known as the \emph{adjoint bundle} (see Appendix~\ref{appBundleGeneral}). That is, the difference $\omega - \omega' \in \Omega^1( P, \mathfrak g)$ between any two connections in $\mathcal C $ can be identified with a unique $\Ad P$-valued 1-form (equivalently, a section of $ TM^* \otimes \Ad P $), and conversely, given any connection $\omega \in \mathcal C$ any other connection can be reached by translating from $ \omega $ by sections in  $\Omega^1(M,\Ad P)$. 

For an Abelian structure group such as $ G \simeq \SO$ studied here, $ \Ad P \xrightarrow{\pi_{\Ad P}} M  $ is canonically isomorphic to the trivial bundle $ M \times \mathfrak g \xrightarrow{\pi_1} M $, so that every section in $\Gamma(\Ad P)$ can be identified with a unique $ \mathfrak g $-valued 1-form on $M$. Correspondingly, we can express every element of $\mathcal C $ as $\omega_o + \pi^* \varpi  $, where $\omega_o \in \mathcal C $ is fixed connection 1-form acting as the ``origin'', and $\varpi \in \Omega^1(M,\mathfrak g) $ is a Lie-algebra-valued 1-form on the base space. For instance, choosing $\omega_o$ as the trivial connection with coordinate basis values (cf.~\eqref{eqConnComponents}) 
\begin{displaymath}
    \omega_o Y_0 = \omega_o Y_1 = 0, \quad \omega_o Y_2 = u,
\end{displaymath}
and horizontal projection map (cf.~\eqref{eqConnHor})
\begin{displaymath}
    \hor_o Y = (dy^0 \cdot Y ) Y_0 + (dy^1 \cdot Y) Y_1, 
\end{displaymath}
we can express the connection 1-from from~\eqref{eqConn} as
\begin{displaymath}
    \omega = \omega_o + \pi^* \varpi, \quad \varpi X = X^\perp h u.
\end{displaymath}

In gauge theories for fundamental physics, the problem of identifying gauge field configurations realized in nature is approached through the \emph{Yang-Mills equations} \cite{YangMills54} involving a \emph{field strength} associated with the connection 1-form. Given a connection 1-form $\omega \in \mathcal C$, the field strength employed in Yang-Mills theory is a 2-form $ F^\omega \in \Omega^2(M,\Ad P) $, taking values in the adjoint bundle, which can be constructed by pulling back the \emph{curvature 2-form} associated with $ \omega $ along trivializing sections of the principal bundle. The latter, is defined as the Lie-algebra-valued 2-form $\Omega \in \Omega^2(P,\mathfrak g )$, given by the exterior covariant derivative of the connection 1-form; i.e., $\Omega = \mathcal D \omega $, where 
\begin{displaymath}
    \Omega(Y,Z) = d\omega(\hor Y, \hor Z), \quad Y, Z \in \Gamma(TP).
\end{displaymath}
Using $\nabla : \Omega^k(M,\Ad P ) \to \Omega^{k+1}(M, \Ad P)$ to denote the covariant derivative induced on the adjoint bundle by $\omega $ (constructed analogously to the covariant derivative on $ E \xrightarrow{\pi_E} M $ from Sections~\ref{secConn} and~\ref{secLapl}), the field strength can be shown to satisfy the \emph{Bianchi identity},
\begin{equation}
    \label{eqBianchi}
    \nabla F^\omega = 0.
\end{equation}
In Yang-Mills theory, $F^\omega$ is further required to satisfy
\begin{equation}
    \label{eqYM}
    \nabla \star F^\omega = 0,
\end{equation}
where $\star : \Omega^k(M,\Ad P ) \to \Omega^{\dim M - k}(M, \Ad P)$ is the Hodge star operator. Equation~\eqref{eqYM} is derived by extremizing a quadratic functional of $F^\omega$ known as the \emph{Yang-Mills action}. See Appendix~\ref{appYM} for further details, and one of the many references in the literature \citep[e.g.,][]{RudolphSchmidt17} for a complete exposition of Yang-Mills theory. 

In the general, non-Abelian, setting, \eqref{eqYM} represents a set of nonlinear partial differential equations, known as Yang-Mills equations, whose rigorous characterization, including existence of solutions, is a challenging open problem (a Millennium Prize problem \cite{JaffeWitten00}). However, in the present setting involving an Abelian gauge theory on two-dimensional Minkowski space, \eqref{eqBianchi} and~\eqref{eqYM} greatly simplify to a linear, first-order differential equation for a single Lie-algebra-valued function, 
\begin{equation}
    d f^\omega = 0, \quad f^\omega \in C^\infty(M,\mathfrak g ),  
    \label{eqSimpleYM}
\end{equation}
such that $ F^\omega = \star f^\omega $. The reduction to~\eqref{eqSimpleYM} is a consequence of the following three simplifications:
\begin{enumerate}
    \item For an Abelian structure group such as $ G \simeq \SO$, the adjoint bundle $ \Ad P \xrightarrow{\pi_{\Ad P}} M $ is trivial and canonically isomorphic to $ M \times \mathfrak g \xrightarrow{\pi_1} M $ (as already stated). As a result, we can canonically identify $ F^\omega \in \Omega^2(M,\Ad P ) $ with a unique Lie-algebra-valued 2-form $ \tilde F^\omega \in \Omega^2(M,\mathfrak g)$. Moreover, because the principal bundle $P \xrightarrow{\pi} M$ over Minkowski space is trivial, $\tilde F^\omega$ is given by the pullback $ \tilde F^\omega = \sigma^* \Omega $ of the curvature 2-form along any global section $\sigma \in \Gamma(P) $.  
    \item For an Abelian structure group, the covariant derivative on $\Ad P $ reduces to a standard exterior derivative (due to the triviality of the adjoint representation of $\mathfrak g$; see Appendix~\ref{appCurv}). Thus,~\eqref{eqYM} reduces to the linear equation, $ d \star \tilde F^\omega = 0 $. 
    \item For a two-dimensional manifold $M$ such as two-dimensional Minkowski space, $F^\omega$ is a top form, and thus the Bianchi identity is trivially satisfied by any element of $ \Omega^2(M,\mathfrak g)$ (even if it is not the field strength induced by a connection). Moreover, $\star \tilde F^\omega $ is equal to a Lie-algebra-valued function, $ f^\omega \in C^\infty(M,\mathfrak g)$, giving $\tilde F^\omega = \star f^\omega $, since $\star \star = \Id $ for 2-forms in two dimensions. Equation~\eqref{eqSimpleYM} then follows.  
\end{enumerate}

The simplified Yang-Mills equation in~\eqref{eqSimpleYM} can be thought of as a two-dimensional analog of Maxwell's equations for electromagnetism. Since Minkowski space is a connected manifold, the only solutions of this equation are constant functions, $ f^\omega( m ) = \lambda_0 $, for some $ \lambda_0 \in \mathfrak g $. To verify that the connection 1-form from~\eqref{eqConn} indeed induces a constant $f^\omega$, let $y : P \to \mathbb R^3$ be a coordinate chart on the principal bundle induced by an inertial chart $ x : M \to \mathbb R^2 $, and consider the matrix representation of the curvature 2-form with respect to the corresponding coordinate vector fields $Y_j$, i.e., $ \bm \Omega = [ \Omega(Y_i, Y_j) ]$. A direct calculation yields
\begin{displaymath}
    \bm \Omega = 
    \begin{pmatrix}
        0 & f & 0 \\
        - f & 0 & 0 \\
        0 & 0 & 0
    \end{pmatrix},
\end{displaymath}
where $ f \in C^\infty(M,\mathfrak g) $ is a constant Lie-algebra-valued function, equal to $ \lambda_0 = 2 u $. As a result, $ \tilde F^\omega = \sigma_x^* \Omega $ is represented by the matrix
\begin{displaymath}
    \bm F := [ \tilde F^\omega(X_i, X_j)] = 
    \begin{pmatrix}
        0 & f  \\
        -f & 0
    \end{pmatrix},
\end{displaymath}
from which we deduce that $ f^\omega = \star \tilde F^\omega = f $ is a constant function. We therefore conclude that the connection 1-form employed in this work satisfies the Yang-Mills equations. 

\subsection{\label{secTransf}Gauge and Lorentz transformations}

In this section, we examine the transformation properties of the gauge theory constructed above under two types of geometrical transformations, namely gauge and Lorentz transformations.

Starting from the former, a \emph{gauge transformation} is a principal-bundle isomorphism of $ P \xrightarrow{\pi} M $; that is, a diffeomorphism $ \varphi : P \to P $ satisfying $ \pi \circ \varphi = \pi $ and  $ \Psi( p \cdot \Lambda ) = \Psi(p) \cdot \Lambda $ for every $ p \in P $ and $ \Lambda \in G$. These properties are represented by the following commutative diagram: 
\begin{equation}
    \label{eqGaugeCommute}
    \begin{tikzcd}
        P \arrow{r}{\varphi} & P \\
        P \arrow{u}{R^\Lambda} \arrow{r}{\varphi} \arrow{d}{\pi} & P\arrow{u}{R^\Lambda} \arrow{dl}{\pi} \\
        M
    \end{tikzcd}.
\end{equation}
The \emph{gauge group} $ \mathcal{G} $ of $ P \xrightarrow{\pi} M$ is the group formed by all such maps $ \varphi $ with composition of maps as the group multiplication. In the present context, where $P$ represents the collection of all inertial frames over $M$ (see Section~\ref{secPrincipal}), a gauge transformation can be thought of as a change of inertial reference frame. 

Every gauge transformation $\varphi \in \mathcal{G} $ induces a bundle isomorphism of the associated bundle $E \xrightarrow{\pi_E} M $. Specifically, there is a diffeomorphism $ \varphi_* : E \to E $, defined as $ \varphi_*( [ p, z]) = [ \varphi(p), z] $, making the following diagram commute:
\begin{displaymath}
    \begin{tikzcd}
        E \arrow{r}{\varphi_*} \arrow{d}{\pi_E} & E \arrow{dl}{\pi_E} \\
        M
    \end{tikzcd}.
\end{displaymath}
Note that the well-definition of $\varphi_*$ as a map on $G$-equivalence-classes in $ E $ depends on the $G$-equivariance of $\varphi$. 

Next, the map $ \varphi_*$ induces a map $ \tilde \varphi_* : \Gamma(E) \to \Gamma(E)$, acting on sections from the left, 
\begin{equation}
    \label{eqGauge}
    \tilde \varphi_* s =  \varphi_* \circ s. 
\end{equation}
Given a section $ \sigma: M \to P$ of the principal bundle, $ \tilde \varphi_* $ induces in turn a map $ \tilde \varphi_*^{\sigma} : C^\infty(M) \to C^\infty(M)$ on $\mathbb{C}$-valued functions on $M $ such that 
\begin{equation}
    \label{eqGaugeSigma}
    \tilde \varphi_*^\sigma  = \zeta_{\sigma}^{-1} \circ \tilde \varphi_* \circ \zeta_\sigma.  
\end{equation}
Henceforth, we will abbreviate $ \tilde \varphi_* $ and $ \tilde \varphi_*^\sigma $ by $ \varphi_* $ and $ \varphi_*^\sigma$ for simplicity of notation.  It can then be shown that for any $ X \in \Gamma(TM)$, the covariant derivative $ \nabla_X$ is \emph{gauge-covariant}, i.e., $ \nabla_X \circ \varphi_* = \varphi_* \circ \nabla_X $. The latter implies that the Laplace-type operators $ \Delta$ and $\bar \Delta$ are also gauge-covariant,   
\begin{displaymath}
    \Delta \circ \varphi_* = \varphi_* \circ \Delta, \quad \bar \Delta \circ \varphi_* = \varphi_* \circ \bar \Delta. 
\end{displaymath}
Similarly, the corresponding operators induced by the section $\sigma$ on $\mathbb{C}$-valued functions on $M$ satisfy $ \nabla_X^M \circ \varphi_*^\sigma = \varphi_*^\sigma \circ \nabla_X^M $, $ \Delta^M \circ \varphi_*^\sigma = \varphi_*^\sigma \circ \Delta^M$, and $ \bar \Delta^M \circ \varphi_*^\sigma = \varphi_*^\sigma \circ \bar \Delta^M$, respectively. 

The above highlight an important difference between arbitrary complex-valued functions on $M$ on one hand, and sections $ s \in \Gamma(E) $ and their associated functions $ s^\sigma = \zeta_{\sigma}^{-1} s $ on the other hand, namely, that the former have no canonical behavior under gauge transformations (as indicated by the commutative diagram in~\eqref{eqGaugeCommute}, where there is no action on $M$), whereas the latter transform non-trivially according to~\eqref{eqGauge} and~\eqref{eqGaugeSigma}. In fact, as we now discuss, gauge transformations characterize the behavior of gauge fields, matter fields, and other geometrical objects defined on the base space $M$ under changes of section $ \sigma :M \to P $ of the principal bundle, which in the present context corresponds to a choice of inertial frame. Intuitively, all matter fields $ s^\sigma $ corresponding to the same underlying section $s$ represent the same intrinsic physical configuration, and it is precisely the transformation properties in~\eqref{eqGaugeSigma} that ensure the mutual consistency of these representations. 

To examine this in more detail, we introduce the \emph{nonlinear adjoint bundle} $ \AD P \xrightarrow{\pi_{ \AD P}} M $, which is a fiber bundle over $M$ associated to the principal bundle, with typical fiber $ G$ (see Appendix~\ref{appBundleGeneral}). The gauge group $\mathcal{G}$ can then be identified with sections in $ \Gamma(\AD P )$. In the present setting involving the Abelian group $G \simeq \SO $, $ \AD P $ is a trivial bundle canonically isomorphic to $ M \times G$, and every gauge transformation $ \varphi \in \mathcal{G} \simeq \Gamma(\AD P) $ can be identified with a unique $G$-valued function $\xi $ on $M$, sometimes referred to as a \emph{gauge map}, such that $ \varphi( p ) = p \cdot \xi( \pi( p ) ) $. Conversely, every smooth function $ \xi : M \to G $ induces a gauge transformation $ \varphi$ defined by the same formula. It then follows from the definition of the left action in~\eqref{eqLeftAction} that the induced transformation $ \varphi_* $ on the associated bundle $ E \xrightarrow{\pi_E} M$ acts by multiplication by a spatially dependent phase factor, viz.    
\begin{equation}
    \label{eqGaugeTransfE}
    \varphi_* e = e^{i\alpha \vartheta(\xi(m)) / \sqrt{2}} e, \quad m \in M, \quad e \in E_m.
\end{equation}
See Lemma~\ref{lemPhase} in Appendix~\ref{appBundleMinkowski} for additional details.

Let now $\sigma, \sigma' : M \to P $ be sections of the principal bundle, and $\xi : M \to G $ the unique gauge map such that $ \sigma'(m) = \sigma(m) \cdot \xi(m)$ for all $m \in M$ (note that the uniqueness of $\xi$ is a consequence of the fact that the action of $G$ on $P$ is free). Associated with $\xi$ is a gauge transformation $ \varphi : P \to P$, constructed as described above, such that $\sigma' = \varphi \circ \sigma$. This transformation represents a change of inertial frame from $\sigma$ to $\sigma'$. Given a section $s \in \Gamma(E)$, it is a direct consequence of~\eqref{eqGaugeTransfE} (see also Lemma~\ref{lemPhase}) that the corresponding matter fields $s^\sigma, s^{ \sigma'} \in C^\infty(M)$ transform according to 
\begin{equation}
    \label{eqTransfS}
    s^{\sigma'}(m) = \varphi_{*}^{-1} s^\sigma(m) = e^{-i \alpha \vartheta(\xi(m))/\sqrt{2}} s^\sigma(m).
\end{equation}
Equation~\ref{eqTransfS} shows that matter fields transform in a contravariant (``inverse'') manner under the gauge transformation describing the change of section from $ \sigma $ to $\sigma'$. This behavior is analogous to the transformation rule for the components of vectors in Euclidean space under a change of basis, which acts by the inverse of the transformation of the basis vectors (the analogs of the inertial frames studied here). As stated above, we view $s^\sigma$ and $s^{\sigma'}$ as representations of the same intrinsic physical configuration, i.e., the section $s $, which should transform according to~\eqref{eqTransfS} for consistency. 

It should be noted that the transformation in~\eqref{eqTransfS} is analogous to the transformation of the quantum mechanical wavefunction of a nonrelativistic charged particle moving in an electromagnetic field under gauge transformations of the scalar and vector potential \citep[e.g.,][Section~2.6]{Sakurai93}. This should be of no surprise, since electromagnetism is a $\UU $ gauge theory, and the representation $\rho : G \to \UU$ used to construct $E \xrightarrow{\pi_E} M$ acts as the universal cover of $\UU$.  

Next, we turn to Lorentz transformations of the base space, represented by the affine maps $L^\Lambda_o : M \to M $ from Section~\ref{secMinkowski}. Here, it is important that the origin $o \in M $ is the same as in the definition of the connection 1-form $\omega$ in Section~\ref{secConn}. This type of transformation is distinct from gauge transformations, in the sense that the latter describe how matter or gauge fields (e.g.,  sections in $\mathcal H$) present themselves under different choices of inertial frame, whereas the former describe how the structure of the theory itself (``physical laws'') depends on the choice of inertial coordinate chart. Here, the key structural object that could be potentially affected by the choice of inertial chart is the connection  1-form. However, as shown in Proposition~\ref{propConn}, and already stated in Section~\ref{secConn}, $\omega$ is in fact independent of the choice of inertial chart, so long as that chart is centered at $o$. Thus, the gauge theory employed in this work is invariant under Lorentz transformations with that point as the origin. Clearly, due to the distinguished role that $ o$ plays, the theory is not invariant under translations in the Poincar\'e group.

\section{\label{secQuantum}Quantum dynamics}

In this section, we construct a quantum mechanical system associated with the gauge theory on Minkowski space from Section~\ref{secGauge}. Then, we describe an embedding of the Koopman operator formulation of the classical harmonic oscillator from Section~\ref{secClassical} into the quantum system, thereby proving Theorem~\ref{thmL2}(i). Our construction of the quantum system on Minkowski space, will be based on the standard density operator formulation of quantum mechanics \cite[e.g.,][]{Sakurai93}. This choice is in part motivated from the fact that the density operator formalism provides a useful model for statistical inference in measure-preserving deterministic systems \cite{Giannakis19b}. In particular, here we will not employ the canonical quantization approach employed in the context of quantum field theory, whereby one would treat solutions $ \tilde \Delta s = 0$ associated with the Laplacian in~\eqref{eqLaplSect} as being ``classical'', and then  promoting each of the normal mode expansion coefficients of these solutions to ladder operators for a quantum field. Instead, we will treat a (self-adjoint extension of) $\tilde \Delta$ directly as a quantum Hamiltonian without reference to an underlying classical system. When there is no risk of confusion with the usage of that term in other contexts, we will refer to the quantum system on Minkowski space as a quantum harmonic oscillator.

\subsection{\label{secHamiltonian}Quantum Dynamics on Minkowski space}

Let $\mathcal{S}(M,E)$ be the Schwartz space of rapidly decreasing sections in $\Gamma(E)$, defined using any inertial chart $ x : M \to \mathbb{R}^2$ and corresponding trivializing section $\sigma_x \in \Gamma(P)$ as
\begin{multline*}
    \mathcal{S}(M,E) = \{ s \in \Gamma(E) : \\  \sup_{m \in M} \lvert x(m)^\alpha \partial^\beta_{x(m)} ( s \circ \zeta_{\sigma_x}^{-1} \circ x ) \rvert < \infty,  \forall \alpha,\beta \in \mathbb{N}^n,\;  n \in \mathbb{N}_0 \}. 
\end{multline*}
The operator $ \tilde H : \mathcal{S}(M) \to \mathcal{H}$, 
\begin{displaymath}
    \tilde H = \frac{1}{2} \tilde \Delta  = \frac{1}{2} \Delta + v, 
\end{displaymath}
is then a symmetric operator with a dense domain $ \mathcal{S}(M) \subset \mathcal{H}$. Similarly, we let $ \mathcal{S}(M) \subset L^2(\nu)$ be the Schwartz space of complex-valued functions on $M$, and $\tilde H^M : \mathcal{S}(M) \to L^2(\nu)$  be the densely defined, symmetric operator equal to $ \tilde \Delta^{M}/2$. As in Section~\ref{secGauge}, we always consider that $x$ is centered at $o \in M$, so that the various coordinate-based formulas for covariant derivatives and Laplace operators from Section~\ref{secConn} apply.

It can then be readily verified from~\eqref{eqLaplSect} and~\eqref{eqLaplSigma} that $\tilde H$ is diagonalizable in an orthonormal basis of eigensections associated with tensor products of Hermite functions in the $x^0 $ and $x^1 $ coordinates. Specifically, let $ p_j \in C^\infty(\mathbb{R}) $ be the Hermite polynomial of degree $ j \in \mathbb{N}_0$, and $ \chi_j \in C^\infty(\mathbb{R}) $ the Hermite function with 
\begin{displaymath}
    \chi_j(z) = \frac{1}{\sqrt{2^j j!}} \left( \frac{\alpha}{\pi} \right)^{1/4} e^{-\alpha^2 z^2/2} p_j(\alpha^{1/2} z).
\end{displaymath}
As is well known, the set $ \{ \chi_j \}_{j=0}^\infty $ is an orthonormal basis of the Hilbert space $L^2(\mathbb{R})$ associated with the Lebesgue measure on $\mathbb{R}$, consisting of eigenfunctions of the quantum harmonic oscillator Hamiltonian \cite{Sakurai93,Hall13}. That is,
\begin{displaymath}
    H_0 \chi_j = E_j \chi_j, \quad E_j = ( 2 j + 1 ) \alpha /2, \quad j \in \mathbb N_0,
\end{displaymath}
where $H_0 : \mathcal{S}(\mathbb{R}) \to L^2(\mathbb{R}) $ is defined on the Schwartz space $\mathcal{S}(\mathbb{R})$ as 
\begin{displaymath}
    H_0 f(z) = \frac{1}{2} f''(z) + \frac{\alpha^2}{2} z^2 f(z).
\end{displaymath}
It then follows from~\eqref{eqLaplSect} and the fact that $L^2(\nu) = L^2(\mathbb{R}) \otimes L^2(\mathbb{R}) $ that the functions $ \chi_{jk} = \chi_j \otimes \chi_k $ form an orthonormal basis $ \{ \chi_{jk} \}_{j,k=0}^\infty$ of $L^2(\nu)$ consisting of eigenfunctions of $\tilde H^M $, corresponding to the eigenvalues
\begin{displaymath}
    E_{jk} = (k -j) \alpha.
\end{displaymath}
Correspondingly, $ \{ \psi_{jk} \}_{j,k=0}^\infty$, with $ \psi_{jk} = \zeta_\sigma \chi_{jk}$, is an orthonormal basis of $\mathcal{H}$ consisting of eigensections of $\tilde H$ at the same corresponding eigenvalues $E_{jk}$. 

Based on the above, we can conclude that $\tilde H$ is a densely-defined, symmetric, diagonalizable operator, and therefore has a unique self-adjoint extension $H : D(H) \to \mathcal{H}$ defined on the domain $D(H) \supset \mathcal{S}(M,E)$ with
\begin{displaymath}
    D(H) = \left \{ \sum_{j,k=0}^\infty c_{jk} \psi_{jk} \in \mathcal{H} : \sum_{j,k=0}^\infty E_{jk}^2 \lvert c_{jk} \rvert^2 < \infty \right \}.
\end{displaymath}
This operator acts as the Hamiltonian of a quantum system, generating a strongly-continuous, unitary group of Heisenberg evolution operators $W^t : \mathcal{H} \to \mathcal{H}$, given by
\begin{equation}
    \label{eqHeis}
    W^t = e^{i t H }.
\end{equation}
Following the standard formulation of quantum mechanics, we will consider the states of this quantum system to be non-negative, trace-class operators on $\mathcal{H}$ with unit trace. Specifically, the set of \emph{regular quantum states} on $\mathcal{H}$ is the closed, convex subset of the Banach space $B_1(\mathcal{H})$ of trace-class operators on $\mathcal H$, equipped with the trace norm, $\lVert A \rVert_{1} = \tr \lvert A \rvert$, given by 
\begin{displaymath}
    \mathcal{Q}(\mathcal{H}) = \{ \rho \in B_1(\mathcal{H}) : \rho \geq 0,\; \tr \rho = 1 \}. 
\end{displaymath}
We recall that a state $\rho \in \mathcal{Q}(\mathcal H)$ is said to be \emph{pure} if there exists $ s \in \mathcal H$ such that $ \rho = \langle \cdot, s \rangle_{\mathcal H} s$. That is, every pure state is a rank-1, orthogonal projection on $\mathcal H$, $ \rho^2 = \rho$, and we have $ \lVert \rho \rVert_1 = \lVert \rho \rVert = 1$, where $ \lVert \cdot \rVert$ denotes the $\mathcal H$-operator norm. The states in $\mathcal Q( \mathcal H)$ which are not pure are said to be \emph{mixed}.

The unitary evolution group $ W^t $ acts on states in $\mathcal{Q}(\mathcal{H})$ through the conjugation map $ Z^t : \mathcal{Q}(\mathcal{H}) \to \mathcal{Q}(\mathcal{H})$, defined as
\begin{equation}
    \label{eqZt}
    Z^t(\rho) = W^{t*} \rho W^t.
\end{equation}
As usual, the observables of the quantum system will be the (bounded and unbounded) self-adjoint operators on $\mathcal{H}$. In what follows, we will typically restrict attention to bounded observables in the set
\begin{displaymath}
    \mathcal{A}(\mathcal{H}) = \{ A \in B(\mathcal{H}) : A^* = A \},
\end{displaymath}
where $B(\mathcal H)$ is the Banach space of bounded operators on $\mathcal H$, equipped with the operator norm, $\lVert \cdot \rVert$. It can be readily verified that $\mathcal{A}(\mathcal{H})$ is closed under addition of operators, scalar multiplication by real numbers, and the multiplication operation 
\begin{displaymath}
    A \cdot B = B \cdot A = ( AB + BA ) / 2.   
\end{displaymath}
Furthermore, we have $ \lVert A \cdot B \rVert \leq \lVert A \rVert \lVert B \rVert$, so that $\mathcal{A}(\mathcal{H})$ has the structure of an Abelian Banach algebra of quantum observables over the real numbers. On $\mathcal{A}( \mathcal H) $, the quantum system has an induced linear action $ \tilde W^t : \mathcal{A}(\mathcal H) \to  \mathcal A( \mathcal H)$, 
\begin{equation}
    \label{eqWtStar}
    \tilde W^t A = W^t A W^{t*},
\end{equation}
which preserves the operator norm. We also let $ \mathcal{A}'(\mathcal H) $ be the  the Banach space of continuous linear functionals on $ \mathcal{A}(\mathcal{H})$, equipped with the standard (operator) norm. 

The space of quantum states $\mathcal Q(\mathcal H)$ embeds naturally  into $\mathcal A'(\mathcal H)$ through the continuous map $\iota : \mathcal Q(\mathcal H) \to \mathcal A'(\mathcal H)$, such that $\iota(\rho) := \mathbb E_\rho $ is equal to the quantum mechanical expectation functional given by
\begin{equation}
    \label{eqQExp}
    \mathbb{E}_\rho A = \tr(\rho A).
\end{equation}
Note that the continuity of $\mathbb E_\rho$ follows from the fact that 
\begin{displaymath}
    \lvert \mathbb E_{\rho} A \rvert \leq  \lVert \rho \rVert_1 \lVert A \rVert = \lVert A \rVert.
\end{displaymath}
In fact, it can be readily verified that $ \iota $ is a contraction, i.e., $ \lVert \iota(\rho) \rVert \leq \lVert \rho \rVert_1 = 1 $, where the equality is saturated by pure states. We shall refer to the elements of $\mathcal A'(\mathcal H)$ as \emph{generalized quantum states}.

Intuitively, $ A \mapsto \mathbb{E}_\rho A $ can be thought of as ``evaluation'' of the quantum mechanical observable $A$ at the state $\rho$.  It follows directly from~\eqref{eqZt} and~\eqref{eqWtStar} that the property $ \mathbb{E}_{Z^t(\rho)} = \mathbb{E}_\rho \circ \tilde W^t$ holds for all $ t\in \mathbb{R}$. In other words, the following diagram commutes, 
\begin{displaymath}
    \begin{tikzcd}
        \mathcal{Q}(\mathcal H) \arrow{r}{Z^t} \arrow{d}{\iota} & \mathcal Q( \mathcal H) \arrow{d}{\iota}\\
        \mathcal{A}'(\mathcal H) \arrow{r}{\tilde W^{t\prime}} & \mathcal A'(\mathcal H),
    \end{tikzcd}
\end{displaymath}
where $ \tilde W^{t\prime} : \mathcal A'(\mathcal H) \to \mathcal A'( \mathcal H)$ is the composition map by $\tilde W^t$, acting isometrically on $ \mathcal A'( \mathcal H)$ according to the formula
\begin{equation}
    \label{eqTildeWt}
    \tilde W^{t\prime} J = J \circ \tilde W^t. 
\end{equation}

This completes the construction of our quantum mechanical system on Minkowski space. It should be noted that the dynamics of this system, generated by $ H $, differ from a spherical quantum harmonic oscillator of frequency $\alpha$ in two-dimensional Euclidean space in two key ways, namely: 
\begin{enumerate}
    \item The energy levels $E_{jk}$ range from $-\infty $ to $ \infty $, as opposed to the Euclidean case, where the energy levels are strictly positive, and have the ground-state eigenvalue $\alpha $. 
    \item The eigenspaces corresponding to any energy level $E_{jk} $ are infinite-dimensional, whereas eigenspaces of the harmonic oscillator in Euclidean space are all finite-dimensional. 
\end{enumerate}
The evolution group generated by $H$ also has a fundamental difference from the Koopman operator group $U^t$ for the classical harmonic oscillator generated by $V $. That is, unlike $U^t$, which is realized through the deterministic flow $ \Phi^t $ on the circle via~\eqref{eqKoop}, the Heisenberg operators $W^t$ in~\eqref{eqHeis} are \emph{not} the realization of a flow on Minkowski space, i.e., there exists no measurable flow $ \hat \Phi^t : M \to M $ such that $W^t f = f \circ \hat \Phi^t$. Indeed, a necessary condition for that to happen is that $ H $ acts as a derivation on an algebra of bounded sections contained in $ D(H) $ \cite[][Theorem~1.1]{TerElstLemanczyk17}, and this is clearly not the case since $H $ is a second-order operator that does not obey a Leibniz rule (cf.~\eqref{eqLeibniz}). 

Yet, as alluded to in Section~\ref{secIntro}, the Koopman and Heisenberg evolution groups studied here do have an important structural similarity, namely that, up to multiplication by $i $, the spectra $ \{ i \alpha_j =  ij \alpha \}_{j=-\infty}^\infty $ and $ \{ E_{jk} = (k-j) \alpha \}_{j,k=0}^\infty$ of their generators, respectively, are identical. In Section~\ref{secIso} ahead, we will take advantage of this spectral equivalence to construct an isometric embedding of the Hilbert space $L^2(\mu)$ associated with the classical harmonic oscillator into the Hilbert space of sections  $\mathcal{H}$ that maps the generator $ V$ to the Hamiltonian $ H $. This construction will be based on a quantum mechanical representation of the periodic dynamics of the circle rotation, which we describe next.         

\subsection{\label{secQuantumCircle}Quantum dynamics of the circle rotation}

Following \cite{Giannakis19b}, we assign ``quantum states'' to the classical harmonic oscillator as positive, unit-trace operators on $L^2(\mu)$; that is, operators lying in the set
\begin{align*}
    \mathcal{Q}(L^2(\mu)) & = \{ \rho \in B_1(L^2(\mu)): \rho \geq 0, \; \tr \rho = 1 \},  
\end{align*}
on which the Koopman group acts by the conjugation map $ \tilde\Phi^t : \mathcal Q(L^2(\mu)) \to \mathcal Q(L^2(\mu))$, 
\begin{displaymath}
    \tilde\Phi^t (\rho) = U^{t*} \rho U^t.
\end{displaymath}
Note that we use the symbol $ \tilde\Phi^t$ to indicate that the dynamics on the quantum states $ \mathcal{Q}(L^2(\mu))$ is induced by the classical flow $\Phi^t$ on the circle (cf.\ $Z^t$ from~\eqref{eqZt}, where no such correspondence exists). 

Next, we assign quantum observables as  self-adjoint operators on $L^2(\mu) $. Proceeding as in Section~\ref{secHamiltonian}, we will focus on the space 
\begin{displaymath}
    \mathcal{A}(L^2(\mu)) =\{ A \in B(L^2(\mu)) : A^* = A \}
\end{displaymath}
of bounded self-adjoint operators, which becomes an Abelian Banach algebra over the real numbers analogously to $\mathcal A(\mathcal H)$. We define the continuous dual  $\mathcal A'(L^2(\mu))$ of $\mathcal A(L^2(\mu))$, as well as the expectation functionals $\mathbb E_\rho \in \mathcal A'( L^2(\mu)) $, evolution maps $ \tilde U^t : \mathcal A(L^2(\mu)) \to \mathcal A(L^2(\mu))$, $\tilde U^{t\prime} : \mathcal A'(L^2(\mu)) \to \mathcal A'(L^2(\mu)) $, and inclusion map $ \iota : \mathcal Q(L^2(\mu)) \to \mathcal A'(L^2(\mu))$ analogously to their counterparts for $ \mathcal A(\mathcal H )$ and $\mathcal A'(\mathcal H)$ in Section~\ref{secQuantum}.   

We now describe an embedding of the classical rotation on the circle to the quantum system introduced above, at the level of both states and observables. For that, consider first the Abelian Banach algebra $C_{\mathbb{R}}(S^1)$ consisting of continuous, real-valued functions on the circle, equipped with the maximum norm. This algebra embeds homomorphically into $\mathcal A(L^2(\mu))$ through the mapping $T : C_{\mathbb{R}}(S^1) \to \mathcal A(L^2(\mu))$ that sends the continuous, real-valued function $f \in C_{\mathbb R}(S^1)$ to the self-adjoint multiplication operator $ T_f := Tf \in \mathcal A(L^2(\mu))$ that multiplies by that function, i.e., $ T_f g = fg$ for all $ g \in L^2(\mu)$. It is then straightforward to verify that 
\begin{equation}
    \label{eqCommuteTS1}
    \tilde U^t \circ T = T \circ U^t, \quad \forall t \in \mathbb{R},
\end{equation}
The above shows that continuous, real-valued classical observables  evolve consistently under the action of the Koopman operator with their corresponding quantum observables (multiplication operators in $\mathcal A(L^2(\mu))$). Similarly, the transpose map $T': \mathcal A'(L^2(\mu)) \to C'_{\mathbb R}(S^1)$, given by $ T' J = J \circ T $, satisfies
\begin{equation}
     T' \circ \tilde U^{t\prime} = U^{t\prime} \circ T', \quad \forall t \in \mathbb{R}
    \label{eqCommuteTPrimeS1}.
\end{equation}

Clearly, one could carry out a similar construction for the quantum system on Minkowski space $M$, replacing $C_\mathbb{R}(S^1)$ with the Banach space of real-valued, compactly supported functions on $M$. In the case of the circle rotation, however, additional correspondences can be derived with the help of the heat kernel (see Section~\ref{secHeatKernel}), namely:  

\begin{prop}
    \label{propS1Quantum}
    Let $ F_\tau : S^1 \to \mathcal F(\mathcal K_\tau) \subset \mathcal K_\tau$ be the feature map associated with the heat kernel on the circle at time parameter $ \tau > 0$. Let also $\Pi : C(S^1) \setminus \{ 0 \} \to \mathcal Q(L^2(\mu))$ map the nonzero classical observable $f $ to the pure quantum state $ \Pi(f) = \rho_f  $ given by $ \rho_f = \langle f, \cdot \rangle_\mu f / \lVert f \rVert^2_\mu $. Then, the following hold for every $ \tau > 0$ and $ t \in \mathbb R$: 
    \begin{enumerate}[(i), wide]
        \item $\Psi_\tau : S^1 \to \mathcal{Q}(L^2(\mu))$, $\Psi_\tau = \Pi \circ F_\tau$, is an injective, continuous map with respect to the trace-norm topology of $\mathcal Q(L^2(\mu))$, and the following dynamical compatibility relationships hold: 
        \begin{gather*}
            F_\tau \circ \Phi^t = U^{-t} \circ F_\tau, \quad \Pi \circ U^{-t} = \Phi^t_* \circ \Pi, \\
            \Psi_\tau \circ \Phi^t = \tilde\Phi^t \circ \Psi_\tau.
        \end{gather*}
    \item The induced linear map $ \Omega_\tau : \mathcal A(L^2(\mu)) \to C_{\mathbb R}(S^1)$, mapping quantum observable $A$ to the classical observable $ f_A = \Omega_\tau A $, with $ f_A(\theta) = \mathbb{E}_{\Psi_\tau(\theta)} A $ is well-defined (i.e., $f_A$ is a continuous, real-valued function), satisfying 
        \begin{displaymath}
            \Omega_\tau \circ \tilde U^t = U^t \circ \Omega_\tau.
        \end{displaymath}
    \item The transpose map $\Omega_\tau': C_{\mathbb R}'(S^1) \to \mathcal A'(L^2(\mu))$, defined as $ \Omega_\tau' J = J \circ \Omega_\tau$ satisfies
        \begin{displaymath}
            \tilde U^{t\prime} \circ \Omega'_\tau = \Omega_\tau' \circ U^{t\prime}.
        \end{displaymath}
    \end{enumerate}
\end{prop}

A proof of Proposition~\ref{propS1Quantum} is included in Appendix~\ref{appS1Quantum}. Figure~\ref{figCommutL2} depicts the quantum-classical relationships from the proposition and \eqref{eqCommuteTS1}, \eqref{eqCommuteTPrimeS1} in the form of commutative diagrams. These diagrams indicate that the time-$\tau$ heat kernel on the circle induces, through its corresponding RKHS feature map (which underpins the maps $\Psi_\tau$, $\Omega_\tau$, and $\Omega'_\tau$) two types of classical--quantum correspondence associated with the circle rotation, namely:
\begin{enumerate}
    \item A continuous, one-to-one mapping of classical states (points) in the circle and regular quantum states on $L^2(\mu)$, carried out by $\Psi_\tau$, as well as a continuous one-to-one mapping with the larger space of functionals in $\mathcal A'(L^2(\mu))$, carried out by $ \Omega_\tau'$. 
    \item A continuous mapping of bounded quantum mechanical observables on $L^2(\mu)$ to classical observables (continuous, real-valued functions) on the circle, carried out by $\Omega_\tau$. 
\end{enumerate}
Moreover, both of these mappings are consistent (equivariant) with the classical and quantum dynamics, in the sense of the commutative diagrams in Fig.~\ref{figCommutL2}. Together, these facts prove the claims about $\Psi_\tau$ and $\Omega_\tau$ in Theorem~\ref{thmL2}(ii, iii), respectively, and will be sufficient to complete the proof of the theorem in Section~\ref{secIso} below. 

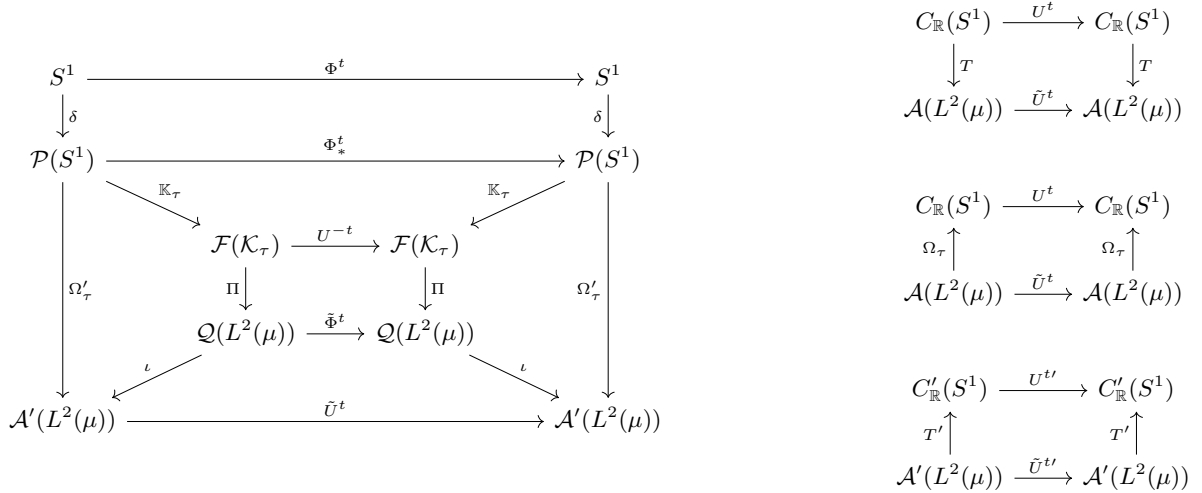
\begin{figure*}
    \begin{minipage}{.65\linewidth}
        \begin{displaymath}
            \begin{tikzcd}
                S^1 \arrow{d}{\delta} \arrow{rrr}{\Phi^t} & & & S^1 \arrow[swap]{d}{\delta} \\
                \mathcal{P}(S^1)  \arrow{rrr}{\Phi^t_*} \arrow{ddd}{\Omega'_{\tau}} \arrow{dr}{\mathbb K_\tau} & & & \mathcal P(S^1) \arrow[swap]{dl}{\mathbb K_\tau} \arrow[swap]{ddd}{\Omega'_{\tau}} \\
                & \mathcal F(\mathcal K_\tau) \arrow[swap]{d}{\Pi} \arrow{r}{U^{-t}} & \mathcal F(\mathcal K_\tau) \arrow{d}{\Pi} \\
                & \mathcal{Q}(L^2(\mu)) \arrow[swap]{dl}{\iota} \arrow{r}{\tilde \Phi^t} & \mathcal{Q}(L^2(\mu)) \arrow{dr}{\iota} \\
                \mathcal A'(L^2(\mu)) \arrow{rrr}{\tilde U^t} & & & \mathcal A'(L^2(\mu)) 
            \end{tikzcd}
        \end{displaymath}
    \end{minipage} \hfill
    \begin{minipage}{.3\linewidth}
        \begin{displaymath}
            \begin{tikzcd}
                C_\mathbb{R}(S^1) \arrow{r}{U^t} \arrow{d}{T} & C_\mathbb{R}(S^1) \arrow{d}{T} \\
                \mathcal A(L^2(\mu)) \arrow{r}{\tilde U^t} & \mathcal A(L^2(\mu))
            \end{tikzcd}
        \end{displaymath} 

        \begin{displaymath}
            \begin{tikzcd}
                C_{\mathbb R}(S^1) \arrow{r}{U^t} & C_{\mathbb R}(S^1) \\
                \mathcal A(L^2(\mu)) \arrow{r}{ \tilde U^t} \arrow{u}{\Omega_\tau} & \mathcal A(L^2(\mu)) \arrow{u}{\Omega_\tau}
            \end{tikzcd}
        \end{displaymath}

        \begin{displaymath}
            \begin{tikzcd}
                C'_{\mathbb R}(S^1) \arrow{r}{U^{t\prime}} & C'_{\mathbb R}(S^1) \\
                \mathcal A'(L^2(\mu)) \arrow{u}{T'} \arrow{r}{\tilde U^{t\prime}} & \mathcal A'(L^2(\mu)) \arrow{u}{T'}
            \end{tikzcd}
        \end{displaymath}
    \end{minipage}
    \caption{\label{figCommutL2}Commutative diagrams illustrating the correspondence between classical states/observables of the harmonic oscillator and quantum states/observables on the $L^2(\mu)$ space associated with the invariant measure. The diagram on the left shows how classical states (points in the circle $S^1$) embed injectively into Radon probability measures $\mathcal P(S^1)$ via the map $\delta$ mapping into Dirac measures, which in turn map injectively into nonzero RKHS functions under the embedding $\mathbb K_\tau$, and then into quantum states and generalized quantum states on $L^2(\mu)$ via the maps $\Pi$  and $\iota$, respectively. The composition $ \mathbb K_\tau \circ \delta$ corresponds to the RKHS feature map $F_\tau$ (not shown). These embeddings are all equivariant with the dynamical evolution maps at each level, as depicted in the diagram. The top diagram in the right shows the dynamical equivariance properties of the Banach algebra homomorphism $ T $, mapping real-valued continuous functions on $S^1$ to bounded, self-adjoint multiplication operators in $\mathcal A(L^2(\mu))$. The middle diagram shows the equivariance properties of $\Omega_\tau : \mathcal A(L^2(\mu)) \to C_\mathbb{R}(S^1)$, which is not an algebra homomorphism. The bottom diagram shows the dynamical equivariance properties of the transpose map $T':\mathcal A'(L^2(\mu)) \to \mathbb C_\mathbb{R}'(S^1)$.}
\end{figure*}

We close this section by noting that, despite their equivariance properties, the classical--quantum correspondences identified thus far fail to preserve an important structure of both the classical and quantum representations of the circle rotation, namely the Abelian algebraic structure of classical observables in $C_{\mathbb R}(S^1)$ and continuous multiplication operators in $ \mathcal A(L^2(\mu))$. That is, apart from special cases, starting from a classical  observable, mapping it to a multiplication operator in $ \mathcal A(L^2(\mu))$ via $T$, and mapping the result back to $C_{\mathbb R}(S^1)$ through $ \Omega_\tau$ does not recover the original observable. In other words, $ \Omega_\tau $ is not a left inverse of $ T$, and similarly, $\Omega'_\tau$ is not a right inverse of $T'$. Had these properties been true, $ \Omega_\tau \circ \tilde U^t \circ T $ would be equal to $ U^t $ and $ T' \circ \tilde U^{t\prime} \circ \Omega_\tau' $ would be equal to $U^{t\prime} $ which would then imply that the evolution of any classical observable in $C_\mathbb{R}(S^1)$ and Radon measure in $C_{\mathbb R}'(S^1)$ can be respectively understood as the evolution of a quantum observable in $\mathcal A(L^2(\mu))$ and a generalized quantum state in $ \mathcal A'(L^2(\mu))$. We will improve upon this inconsistency in Section~\ref{secRKHA} by considering quantum states and algebras of quantum observables associated with the RKHAs $ \hat{\mathcal K}_\tau$, as opposed to the $L^2(\mu)$ space. 

\subsection{\label{secIso}Isometric embedding}

We now have the necessary ingredients to construct an isometric embedding (Hilbert space homomorphism) of the $L^2(\mu)$ space associated with the invariant measure of the circle rotation into the Hilbert space $\mathcal H$ of sections over Minkowski space, thereby inducing an embedding of classical and quantum states of the circle rotation into the states of the quantum harmonic oscillator on Minkowski space, as well as a corresponding pullback map from quantum observables of the latter system into quantum and classical observables of the circle rotation. 

We begin by introducing the following three Hilbert subspaces of $\mathcal{H}$,
\begin{equation}
    \label{eqHSub}
    \begin{aligned}
        \mathcal{H}_0 &= \spn \{ \psi_{00} \}, \\
        \mathcal{H}_- &= \overline{ \spn\{\psi_{j0} : j > 0 \}}, \\
        \mathcal{H}_+ &= \overline{ \spn\{ \psi_{0j}: j > 0 \}}, 
    \end{aligned}
\end{equation}
where overlines denote closure with respect to $ \mathcal{H} $ norm. Because they are spanned by distinct eigenfunctions of $ H $, these subspaces are mutually orthogonal, and each of them is invariant under the Heisenberg operator $W^t $ for all $ t \in \mathbb{R}$. Moreover, if a section $ s $ lies in $\mathcal{H}_0 $, $ \mathcal{H}_- \cap D(H) $, or $\mathcal{H}_+ \cap D(H)$, then $ \langle s, H s \rangle_{\mathcal{H}} $ is zero, strictly negative, or strictly positive, respectively. Due to that, we intuitively interpret pure states $ \rho= \langle s, \cdot \rangle_{\mathcal{H}} s$ with $s $  in $\mathcal{H}_0$, $\mathcal{H}_-$, or $\mathcal{H}_+$, as ``vacuum'', ``antimatter'', or ``matter'' states, respectively. Defining  $ \tilde{\mathcal{H}} = \mathcal{H}_0 \oplus \mathcal{H}_{-} \oplus \mathcal{H}_+ $, we note that for pure states with $s $ in the orthogonal complement of $\tilde{\mathcal{H}}$ in $ \mathcal H $,  $ \langle s, Hs \rangle_{\mathcal{H}}$ is not sign-definite.

With these considerations in mind, and in analogy with~\eqref{eqHSub}, we define the Hilbert subspaces of $L^2(\mu)$ given by
\begin{equation}
    \label{eqL2Sub}
    \begin{aligned}
        L^2_0(\mu) &= \spn\{ \phi_0 \}, \\
        L^2_-(\mu) &= \overline{\spn\{ \phi_j : j < 0 \} }, \\
        L^2_+(\mu) &= \overline{\spn\{ \phi_j : j > 0 \} }, 
    \end{aligned}
\end{equation}
where $\phi_j$ are the Koopman eigenfunctions from Section~\ref{secKoopEig}. These subspaces are all invariant under the Koopman operator $U^t$ for any $t \in \mathbb{R}$. Moreover, $ \langle f, V f \rangle_{\mu}$ is zero, strictly negative, or strictly positive whenever  $f$ lies in $L^2_0(\mu) $, $L^2_-(\mu) \cap D(V) $, or $L^2_+(\mu) \cap D(V)$, so we interpret the corresponding pure states $ \rho = \langle f, \cdot \rangle_\mu f$ as being zero-, negative-, or positive-frequency states, respectively.  

Motivated by the similarity between~\eqref{eqHSub} and~\eqref{eqL2Sub}, we define the linear operator $\mathcal{U} : L^2(\mu) \to \mathcal{H}$, characterized completely through the relationships
\begin{equation}
    \label{eqUTransf}
    \mathcal{U} \phi_j = 
    \begin{cases}
        \psi_{00}, & j = 0, \\
        \psi_{-j,0}, & j < 0, \\
        \psi_{0j}, & j > 0.
    \end{cases}
\end{equation}
It then follows directly from its definition that $\mathcal{U}$ is an isometric embedding of $L^2(\mu)$ into $\mathcal{H}$, i.e., it is an injective operator satisfying  
\begin{displaymath}
    \langle \mathcal{U} f_1, \mathcal{U} f_2 \rangle_{\mathcal{H}} = \langle f_1, f_2 \rangle_{\mu}, \quad \forall f_1,f_2 \in L^2(\mu),
\end{displaymath}
and $\mathcal{U}^* \mathcal{U}$ is equal to the identity on $L^2(\mu)$.  Moreover, $\mathcal{U}$ clearly maps $L^2_0(\mu)$ to $\mathcal{H}_0$, $L^2_-(\mu)$ to $\mathcal{H}_-$, and $L^2_+(\mu)$ to $\mathcal{H}_+$.  In other words, $\mathcal{U}$ maps  the zero-, negative-, and positive-frequency states of the classical harmonic oscillator map to the vacuum, matter, and antimatter states of the quantum harmonic oscillator on Minkowski space. The range of $\mathcal{U}$ is clearly equal to $\tilde{\mathcal{H}}$, and as a result $\mathcal{U} \mathcal{U}^*$ is equal to the identity operator on that space. 

The operator $\mathcal{U}$ provides the isometric embedding of the Koopman operator formulation of the classical harmonic oscillator into the quantum harmonic oscillator on Minkowski space stated in Theorem~\ref{thmL2}. One of its key properties is that for every $j\in \mathbb{Z}$, $\mathcal{U} \phi_j$ is an eigenfunction of $ H $ at the same eigenvalue $ j \alpha $ as the eigenfrequency corresponding to $ \phi_j$. As a result, $\mathcal{U}$ pulls back the Hamiltonian $H $ to the generator $V$, 
\begin{displaymath}
    \mathcal{U}^* H \mathcal{U} = V / i, 
\end{displaymath}
and we have the dynamical correspondence
\begin{equation}
    \label{eqDynCorr}
    \mathcal{U}^* e^{it H} \mathcal{U} = e^{tV}, \quad \forall t \in \mathbb{R}.
\end{equation}
It should also be noted that the dual operation, i.e., $ V \mapsto \mathcal{U} V \mathcal{U}^* $, maps the generator (up to multiplication by $i$) to a projected Hamiltonian, 
\begin{displaymath}
    \mathcal{U} V \mathcal{U}^* = i \Pi_{\tilde{\mathcal{H}}} H \Pi_{\tilde{\mathcal{H}}},
\end{displaymath}
where $\Pi_{\tilde{\mathcal{H}}} : \mathcal{H} \to \mathcal{H}$ is the orthogonal projection operator mapping into $ \tilde{\mathcal{H}}$. 

Turning now to the spaces $ \mathcal A(L^2(\mu))$ and $ \mathcal A(\mathcal H)$, $\mathcal U$ induces a surjective linear map $ \tilde{ \mathcal U} : \mathcal A( \mathcal H) \to \mathcal A(L^2(\mu))$, an isometry $ \tilde{\mathcal U}^+ : \mathcal Q(L^2(\mu)) \to \mathcal Q(\mathcal H)$, and a linear isometry $ \tilde{\mathcal U}' : \mathcal A'(L^2(\mu)) \to \mathcal A'(\mathcal H) $, where
\begin{displaymath}
    \tilde{\mathcal U} A = \mathcal U^* A \mathcal U, \quad U\tilde{\mathcal U}^+(\rho) = \mathcal U \rho \mathcal U^*, \quad  \tilde{ \mathcal U }' J = J \circ \tilde{\mathcal U}.
\end{displaymath}
It is then straightforward to verify using Proposition~\ref{propS1Quantum} and \eqref{eqDynCorr} that the equivariance properties
\begin{displaymath}
    \tilde{\mathcal U} \circ \tilde W^t = \tilde U^t \circ \tilde{\mathcal U}, \quad \tilde{\mathcal U}^+ \circ Z^t = \tilde \Phi^t \circ \tilde \Phi^t, \quad \tilde{\mathcal U}' \circ \tilde U^{t\prime} = \tilde W^{t\prime} \circ \tilde{\mathcal U}',
\end{displaymath}
hold for every $ t \in \mathbb{R}$. 

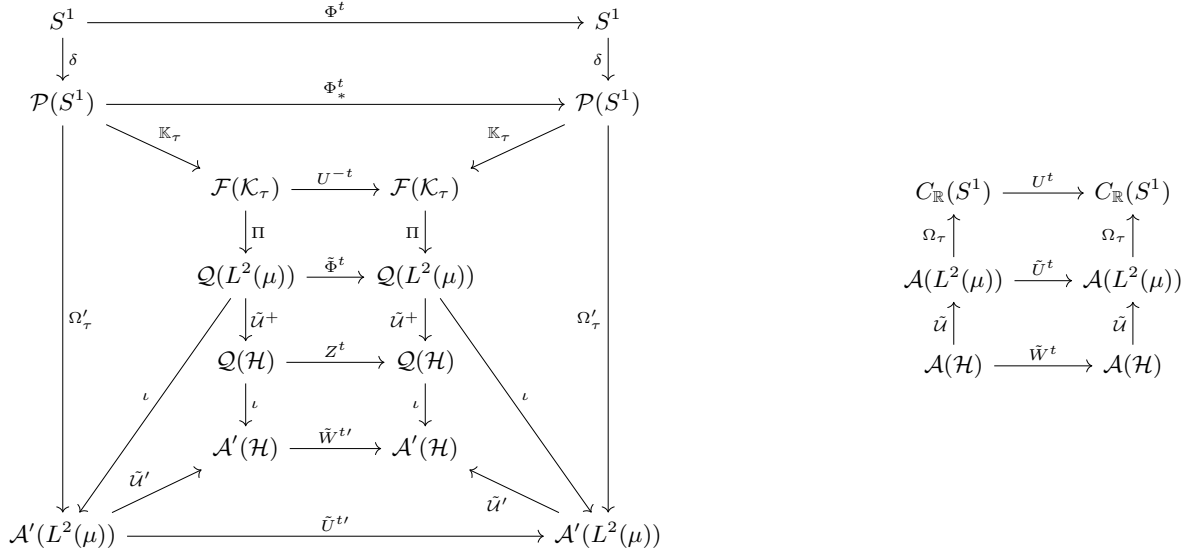
\begin{figure*}
    \begin{minipage}{.65\linewidth}
        \begin{displaymath}
            \begin{tikzcd}
                S^1 \arrow{d}{\delta} \arrow{rrr}{\Phi^t} & & & S^1 \arrow[swap]{d}{\delta} \\
                \mathcal{P}(S^1) \arrow{rrr}{\Phi^t_*} \arrow{ddddd}{\Omega'_{\tau}} \arrow{dr}{\mathbb K_\tau} & & &   \mathcal{P}(S^1) \arrow[swap]{dl}{\mathbb K_\tau} \arrow[swap]{ddddd}{\Omega'_{\tau}} \\
                & \mathcal F(\mathcal K_\tau) \arrow{d}{\Pi} \arrow{r}{U^{-t}} & \mathcal F(\mathcal K_\tau) \arrow[swap]{d}{\Pi} \\
                & \mathcal{Q}(L^2(\mu)) \arrow{d}{\tilde{\mathcal U}^+} \arrow[swap]{dddl}{\iota} \arrow{r}{\tilde \Phi^t} & \mathcal{Q}(L^2(\mu)) \arrow{dddr}{\iota} \arrow[swap]{d}{\tilde{\mathcal U}^+} \\
                & \mathcal Q(\mathcal H) \arrow{d}{\iota} \arrow{r}{Z^t} & \mathcal Q(\mathcal H) \arrow[swap]{d}{\iota}\\
                & \mathcal A'(\mathcal H) \arrow{r}{\tilde W^{t\prime}} & \mathcal A'(\mathcal H) \\
                \mathcal A'(L^2(\mu)) \arrow{ur}{\tilde{\mathcal U}'} \arrow{rrr}{\tilde U^{t\prime}}  & & & \mathcal A'(L^2(\mu)) \arrow{ul}{\tilde{\mathcal U}'} 
            \end{tikzcd}
        \end{displaymath}
    \end{minipage} \hfill
    \begin{minipage}{.3\linewidth}
        \begin{displaymath}
            \begin{tikzcd}
                C_{\mathbb R}(S^1) \arrow{r}{U^t} & C_{\mathbb R}(S^1) \\
                \mathcal A(L^2(\mu)) \arrow{r}{ \tilde U^t} \arrow{u}{\Omega_\tau} & \mathcal A(L^2(\mu)) \arrow{u}{\Omega_\tau} \\
                \mathcal A(\mathcal H) \arrow{u}{\tilde{\mathcal U}} \arrow{r}{\tilde W^t} & \mathcal A(\mathcal H) \arrow{u}{\tilde{\mathcal U}}
            \end{tikzcd}
        \end{displaymath}
    \end{minipage}
\caption{\label{figCommutL2H}As in Fig.~\ref{figCommutL2}, but for commutative diagrams showing the relationships between the quantum formulations of the circle rotation and the quantum harmonic oscillator on Minkowski space. Note that in the left-hand diagram we abuse notation, using $\iota$ to denote the inclusion maps of both $\mathcal Q(L^2(\mu))$ into $\mathcal A'(L^2\mu))$ and $ \mathcal{Q}(\mathcal H)$ into $ \mathcal A'(\mathcal H)$.}
\end{figure*}

The above, allow us to augment the commutative diagrams in Figure~\ref{figCommutL2} to obtain the diagrams in Figure~\ref{figCommutL2H}, illustrating the combined relationships between states and observables of the circle rotation (both classical and quantum) and those of the quantum harmonic oscillator on Minkowski space. The following proposition, which can be deduced directly from these diagrams, shows that there exist quantum states of the system on Minkowski space for which the expectation value of any observable evolves consistently as evaluation of a classical observable of the circle rotation.

\begin{prop}
    For any $\theta \in S^1$ and $ \tau >0 $, let $ \rho_{\theta,\tau} \in \mathcal Q(\mathcal H) $ be the quantum state given by $ \sigma_{\theta,\tau} = \tilde{\mathcal U}^+(\Psi_\tau(\theta))$. Then, for any quantum mechanical observable $ A \in \mathcal A(\mathcal H) $ and $ t \in \mathbb R$, the relationship
    \begin{displaymath}
        \mathbb E_{Z^t(\sigma_{\theta,\tau})} A = f_{A,\tau}(\Phi^t(\theta))
    \end{displaymath}
    holds, where $ f_{A,\tau} \in C_{\mathbb R}(S^1)$ is the classical observable of the circle rotation given by $ f_{A,\tau} = \Omega_\tau(\tilde{\mathcal{U}}(A))$. In particular, $ t \mapsto \mathbb E_{Z^t(\rho_\theta)}$ is periodic with period $ 2 \pi / \alpha $.
\end{prop}

It can be readily verified that $ \rho_{\theta,\tau}$ takes the form $ \rho_{\theta,\tau} = \langle \psi_{\theta,\tau},\cdot \rangle_{\mathcal H} \psi_{\theta,\tau} / \lVert \kappa_\tau(\theta,\cdot) \rVert^2_\mu$, where $ \psi_{\theta,\tau} \in \mathcal H $ is the section (wavefunction) given by
\begin{equation}
    \label{eqWF}
    \psi_{\theta,\tau} = \sum_{j=0}^\infty e^{-j^2\tau}( \phi_j(\theta) \psi_{j0} + \phi_j^*(\theta) \psi_{0j}).
\end{equation}
Figure~\ref{figPsi} shows plots of the corresponding $\mathbb C$-valued function $ \psi^\sigma_{\theta, \tau} = \zeta_{\sigma_x}^{-1} \psi_{\theta,\tau}$ associated with an inertial coordinate chart $x = (x^0,x^1)$ centered at $o$. There, as $\theta $ increases from 0 to $\pi$, $\psi^\sigma_{\theta,\tau} $ is seen to undergo an evolution from a real-valued, diffuse configuration supported mainly in the $x^0,x^1\geq 0$ coordinate quadrant, to a focused configuration near the origin at $\theta= \pi/2$, and finally to a real-valued, diffuse configuration supported in the $x^0,x^1\leq 0$ quadrant.

\begin{figure*}
    \includegraphics[width=\linewidth]{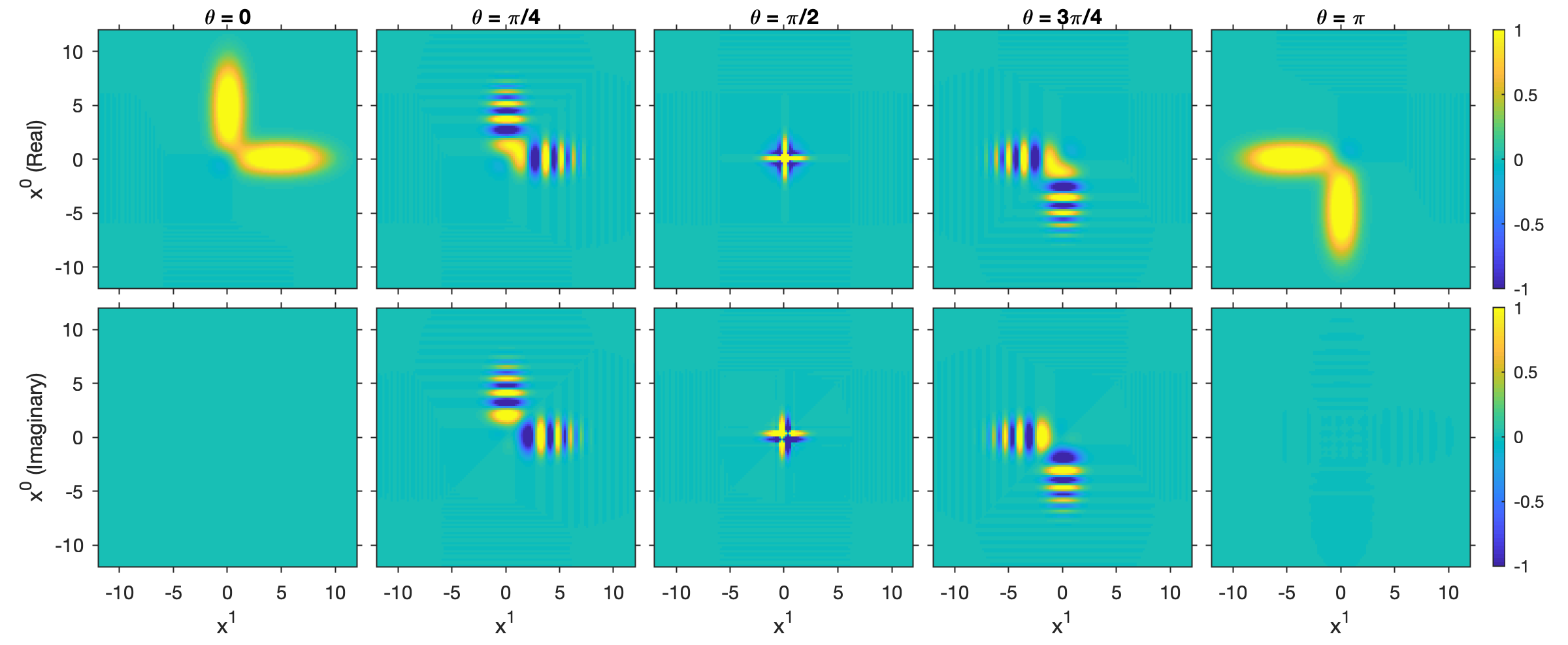}
    \caption{\label{figPsi}Wavefunction $\psi^{\sigma_x}_{\theta,\tau}$ associated with the embedding of the classical states (points) on the circle to states of the quantum harmonic oscillator on two-dimensional Minkowski space. Here, the wavefunction is shown as a function of the inertial coordinates $(x^0,x^1)$ of Minkowski space for $\tau = 10^{-3}$ and representative values of $\theta \in S^1$ in the range $[0, \pi]$, using an arbitrary normalization with respect to $L^@(\nu)$ norm.}
\end{figure*}

\subsection{Gauge covariance}
At this point, we have completed the proof of Theorem~\ref{thmL2} aside from the claimed gauge covariance of the isometry $\mathcal U : L^2(\mu) \to \mathcal H$. Specifically, the construction of $ \mathcal U$ in Section~\ref{secIso} employed a particular basis of $\mathcal H$ consisting of the eigensections $ \psi_{jk}$, whose construction (in Section~\ref{secHamiltonian}) depended on a choice of section $ \sigma: M \to P$ of the principal bundle to map Hermite functions into sections. Since every such choice $\sigma$ corresponds to a choice of inertial frame, and Minkowski space does not have a preferred inertial frame, it is important that the transformation imparted to $ \mathcal U$ by passing from $\sigma$ to a different section, $ \sigma' : M \to P$, be structure-preserving. In particular, the operator $ \mathcal U_{\sigma'} : L^2(\mu) \to \mathcal H $ defined as in~\eqref{eqUTransf}, but using the $H$ eigensections $ \psi_{\sigma',jk} = \zeta_{\sigma'} \chi_{jk}$ instead of $ \psi_{jk}$, should be relatable to $\mathcal U$ by a unitary map $ \Xi : \mathcal H \to \mathcal H$ such that $ \mathcal U_{\sigma'} = \Xi \mathcal U$ (otherwise, Hilbert space structure would not be preserved), and moreover the dynamical correspondence in~\eqref{eqDynCorr} should still hold for $ \mathcal U_{\sigma'}$. 

To verify that this is indeed the case, let $ \varphi : P \to P $ be the unique gauge transformation such that $ \sigma' = \varphi \circ \sigma$. Then, by Lemma~\ref{lemUnique}(i), $ \psi_{\sigma',jk}(m) = e^{i\alpha \vartheta(\xi(m)) / \sqrt{2}} \psi_{jk}(m) $ for a $G$-valued function $ \xi $ on $M$, so that $ \mathcal U_{\sigma'} = \Xi \circ \mathcal U $, where $ \Xi$ is the unitary multiplication operator on $\mathcal H$ by the function $ e^{i \alpha \vartheta(\xi(\cdot))}$. Moreover, again by Lemma~\ref{lemUnique}(i), we have $ \psi_{\sigma',jk} = \varphi_* \circ \varphi_{jk}$, and because the operator $ \varphi_* \circ $ commutes with the Laplacian $ \bar \Delta $ (see Section~\ref{secTransf}), it follows that $ \mathcal U^*_{\sigma'} H \mathcal U_{\sigma'} = \mathcal U^* H \mathcal U$ so that $ \mathcal U^*_{\sigma'} e^{itH } \mathcal U_{\sigma' } = \mathcal U^*_{\sigma} e^{itH } \mathcal U_{\sigma}$. We therefore conclude that the dynamical correspondence in~\eqref{eqDynCorr} holds for $\mathcal U_{\sigma'}$, completing the proof of Theorem~\ref{thmL2}. 

\section{\label{secCCO}Canonically commuting operators for the classical harmonic oscillator}

In this section, we employ the correspondence between the circle rotation and quantum harmonic oscillator on Minkowski space established in Section~\ref{secQuantum} to construct operators acting on classical observables (functions) on the circle exhibiting canonical position--momentum commutation relationships, as well as their corresponding ladder operators. This construction will be carried out using the isometry $ \mathcal U : L^2(\mu) \to \mathcal H$ from~\eqref{eqUTransf} to pull back position, momentum, and ladder operators of the quantum harmonic oscillator to densely defined operators on $L^2(\mu)$. These operators will automatically satisfy the appropriate commutation relationships since $ \mathcal U$ is a Hilbert space homomorphism. 

\subsection{\label{secLadder}Ladder operators}

As with the non-relativistic quantum harmonic oscillator, the quantum system on Minkowski space from Section~\ref{secHamiltonian} admits canonically commuting \emph{position} and \emph{momentum operators} with respect to any inertial chart $ x = ( x^0, x^1 ) : M \to \mathbb{R}^2$ (taken to be centered at $ o  \in M $ as per our convention). These operators can be constructed as usual by considering the one-parameter unitary groups of operators $ \mathcal{X}^a_j : \mathcal{H} \to \mathcal{H}$ and $\mathcal{P}_j^a : \mathcal{H} \to \mathcal{H}$, with $ j \in \{ 0, 1 \}$ and $a \in \mathbb{R}$, acting on sections $ s \in \mathcal{H}$ by phase multiplication and translation, respectively, i.e., 
\begin{displaymath}
    \mathcal{X}_j^a s( m )  = e^{i a x^j(m)} s(m), \quad \mathcal{P}_j^a s( m )  = s(m + a \vec X_j).
\end{displaymath}
Note that $ \mathcal X_j^a$ coincides with the action $ \varphi_* $ on sections associated with the gauge transformation $ \varphi : P \to P $ with $ \varphi( p ) = p \cdot \xi(\pi(p))$, $ \xi(m) = \exp( \sqrt{2} a x^j(m) / \alpha )$.  We then define the position operators $ \hat X_j : D( \hat X_j) \to \mathcal{H}$, $ D(\hat X_j) \subset \mathcal{H} $, and momentum operators $ \hat P_j : D(\hat P_j) \to \mathcal{H}$, $D( \hat P_j) \subset \mathcal{H}$, as the generators of these groups, respectively, times a factor of $ 1/i $ introduced by convention to render the generators self-adjoint. That is, we have
\begin{equation}
    \label{eqQP}
    \begin{aligned}
        \hat X_j s &= i^{-1} \lim_{a\to 0} (\mathcal{X}_j^a s - s) / a, \\
        \hat P_j s &= i^{-1} \lim_{a \to 0} (\mathcal{P}_j^a s - s) / a,
    \end{aligned}
\end{equation}
and the domains $D(\hat X_j) $ and $D(\hat P_j) $ are defined as the dense subspaces of $ \mathcal{H} $ where the respective limits in~\eqref{eqQP} exist with respect to $\mathcal{H}$ norm. It then follows from these definitions that $\hat X_j $ is a multiplication operator by the coordinate function $x^j $, i.e.,       
\begin{displaymath}
    \hat X_j s = x^j s,  
\end{displaymath}
while $ \hat P_j $ is a differentiation operator that behaves as an extension of the coordinate vector field $X_j $ to smooth sections in $\mathcal{H} $, 
\begin{displaymath}
    \hat P_j s = -i X_j s = -i\, ds \cdot X_j, \quad \forall s \in D(\hat P_j) \cap \Gamma(E). 
\end{displaymath}

Note now that the Schwartz space $\mathcal{S}(M,E)$ is invariant under all of $\hat X_j$ and $ \hat P_j $, so we can consider the restricted operators $ \tilde X_j : \mathcal{S}(M, E) \to \mathcal{H}$ and $ \tilde P_j : \mathcal{S}(M,E) \to \mathcal{H}$, where $ \tilde X_j = \hat X_j \rvert_{\mathcal{S}(M,E)}$, $ \tilde P_j = \hat P_j \rvert_{\mathcal{S}(\Gamma,E)}$, and $ \ran \tilde X_j $, and $ \ran \tilde P_j $ are both subspaces of $ \mathcal{S}(M,E)$. It then follows directly from their definition that these operators obey the canonical position--momentum commutation relations,
\begin{displaymath}
    [\tilde X_j, \tilde X_k ] = 0, \quad [ \tilde P_j,\tilde P_k] = 0, \quad [ \tilde X_j, \tilde P_k] = i \delta_{jk}. 
\end{displaymath}
Moreover, associated with $ \tilde X_j $ and $\tilde P_j $ are the \emph{ladder operators}
\begin{displaymath}
    A_j = \sqrt{\frac{\alpha}{2}} \tilde X_j + \frac{i}{\sqrt{2\alpha}} \tilde P_j, \quad A_j^+ = \sqrt{\frac{\alpha}{2}} \tilde X_j - \frac{i}{\sqrt{2\alpha}} \tilde P_j, 
\end{displaymath}
and the \emph{number operators}
\begin{displaymath}
    N_j = A_j^+ A_j. 
\end{displaymath}
These operators satisfy the standard commutation relations 
\begin{equation}
    \label{eqCommMink}
    \begin{gathered}
        \,[N_j, A_k] = - A_j \delta_{jk}, \quad [N_j, A_k^+] = A_j \delta_{jk}, \\
        [A_j, A_k] = [A_j^+,A_k^+] = [N_j,N_k] = 0,
    \end{gathered}
\end{equation}
and moreover we have
\begin{equation}
    \label{eqNumMink}
    \tilde H = \alpha (-N_0 + N_1). 
\end{equation}
The following are formulas for the action of these operators on the $ \psi_{jk} $ basis elements of $\mathcal{H}$, which follow from standard results on ladder operators \cite{Sakurai93}:
\begin{equation}
    \label{eqLadderPsi}
    \begin{gathered}
        A_0 \psi_{jk} = \sqrt{j} \psi_{j-1,k}, \quad A_0^+ \psi_{jk} = \sqrt{j+1} \psi_{j+1,k},  \\
        A_1 \psi_{jk} = \sqrt{k} \psi_{j,k-1}, \quad A_1^+ \psi_{jk} = \sqrt{k+1} \psi_{j,k+1},  \\ 
        N_0 \psi_{jk} = j \psi_{jk}, \quad N_1 \psi_{jk} = k \psi_{jk}.
    \end{gathered}
\end{equation}

We will now pull back the $ A_j $, $A_j^+$, and $ N_j $ to operators on classical observables of the circle rotation.  For that, let $\mathcal{S}(S^1) = \mathcal{U}^* \mathcal{S}(M,E)$ be the image of the Schwartz space of sections on Minkowski space under $\mathcal{U}^*$. For completeness, we note that $\mathcal{S}(S^1)$ can be identified with a Schwartz space of functions on the circle (see Ref.~\cite{AizenbudGourevitch11} for definitions), which in this case coincides with $C^\infty(S^1)$ by compactness of $S^1$. However, here we will not be needing that structure. 

By construction, $\mathcal{U} $ maps $\mathcal{S}(S^1) $ to $\mathcal{S}(M,E)$, and therefore the following are well-defined as densely-defined operators from $\mathcal{S}(S^1) \subset L^2(\mu)$ to $L^2(\mu)$:
\begin{equation}
    \label{eqLadder}
    \begin{gathered}
        A_- = \mathcal{U}^* A_0 \mathcal{U}, \quad A_-^+ = \mathcal{U}^* A_0^+ \mathcal{U}, \quad N_- = \mathcal{U}^* N_0 \mathcal{U}, \\
        A_+ = \mathcal{U}^* A_1 \mathcal{U}, \quad A_+^+ = \mathcal{U}^* A_1^+ \mathcal{U}, \quad N_+ = \mathcal{U}^* N_1 \mathcal{U}.
    \end{gathered}
\end{equation}
It then follows from the fact that $\mathcal{U} \mathcal{U}^*$ is the identity on $\tilde{\mathcal{H}}$ that these operators satisfy the canonical commutation relationships for ladder operators analogously to~\eqref{eqCommMink}, viz.
\begin{gather*}
    [ N_-,A_-] = - A_-, \quad [ N_-, A_-^+ ] = A_-^+, \\
    [ N_+,A_+] = - A_+, \quad [ N_+, A_+^+ ] = A_+, \\
    [ A_-, A_+] = [ A_-^+, A_+^+ ] = [ A_-, A_+^+ ] = [ A_+, A_-^+ ] = 0, \\
    [ N_-, N_+] = 0.
\end{gather*}
Moreover, the generating vector field $ \vec V $ of the classical harmonic oscillator, acting on $C^\infty(S^1)$ functions, can be expressed in terms of the number operators as (cf.~\eqref{eqNumMink})
\begin{equation}
    \label{eqVNum}
    \vec V = i \alpha( -N_- + N_+) = i \mathcal{U}^* \tilde H \mathcal{U}. 
\end{equation}

To characterize these operators more explicitly, consider  the order-$r$ \emph{fractional derivative operator} $\partial^r : D(\partial^r) \to L^2(\mu)$, $ r \geq 0 $, associated with standard angle coordinates on the circle (see Appendix~\ref{appFractional} for an explicit definition and additional details). Let also $\Pi_- : L^2(\mu) \to L^2(\mu)$ and $\Pi_+ : L^2(\mu) \to L^2(\mu)$ be the orthogonal projections mapping into $L^2_-(\mu)$ and $L^2_+(\mu)$, respectively, and define the spectrally truncated derivative operators $ \partial_-^r = \partial^r \Pi_-$ and $ \partial^r_+ = \partial^r \Pi_+$, both of which are defined on the same dense domain $D(\partial^r) \supset \mathcal{S}(S^1)$ as $\partial^r$. It can be shown (see Proposition~\ref{propLadder} in Appendix~\ref{appFractional}) that the ladder operators $A_{\pm} $ and $A_{\pm}^+$ take the form  
\begin{equation}
    \label{eqLadderFrac}
    \begin{gathered}
        A_- = i^{1/2} L^* \partial^{1/2}_-, \quad A_-^+ = i^{1/2} \partial_-^{1/2} L, \\
        A_+ = i^{-1/2} L \partial^{1/2}_+, \quad A_+^+ = i^{-1/2} \partial^{1/2}_+ L^*,
    \end{gathered}
\end{equation}
where $L$ and $L^*$ are the ladder-like operators from Section~\ref{secKoopLadder}. Moreover, it can be readily verified that the relationships
\begin{displaymath}
    \partial_-^ q \partial_-^r = \partial_-^{q+r}, \quad \partial_+^q \partial_+^r = \partial^{q+r}_+, \quad \partial_-^r+\partial_+^r = \partial^r     
\end{displaymath}
hold for every $q,r \geq 0$, leading, in conjunction with~\eqref{eqLadderLId}, to the formulas
\begin{displaymath}
    N_- = i \partial_-, \quad N_+ = - i \partial_+
\end{displaymath}
for the number operators. Note that the expressions above are consistent with~\eqref{eqVNum}, i.e., 
\begin{equation}
    \label{eqVNumDecomp}
    \vec V = \alpha \partial = \alpha( \partial_- + \partial_+) = i\alpha(-N_- + N_+). 
\end{equation}

In summary, thus far we have seen that pulling back the ladder operators for the quantum harmonic oscillator on Minkowski space under the Hilbert space homomorphism $\mathcal U$ leads to operators on classical observables ($\mathbb C$-valued functions) of the circle rotation which have the structure of order-$1/2$ fractional derivatives, composed with the ladder-like operators from Section~\ref{secKoopLadder}. By virtue of the structure-preserving properties of $ \mathcal U$, the ladder operators for the circle rotation exhibit canonical commutation relationships, and lead to the decomposition in~\eqref{eqVNumDecomp} of the generator of the Koopman group on $L^2(\mu)$ through their corresponding number operators. This completes the proof of Theorem~\ref{thmOps}(i).

\subsection{\label{secPosMom}Position and momentum operators}

Next, to construct ``position'' and ``momentum'' operators for the circle rotation, we can simply pull back the corresponding operators, $ \hat X_j$ and $ \hat P_j$, respectively, for the quantum harmonic oscillator on Minkowski space using $\mathcal U $. That is, we define $ \hat X_\pm : D( \hat X_\pm) \to L^2(\mu)$, $ \hat P_{\pm} : D(\hat P_\pm) \to L^2(\mu)$, where $D(\hat X_\pm)$ and $D(\hat P_\pm) $ are the dense subspaces of $L^2(\mu)$ given by $D(\hat X_-) = \mathcal U^* D( \hat X_0) $, $D(\hat X_+) = \mathcal U^* D(\hat X_1)$, $ D(\hat P_-) = \mathcal U^* D( \hat P_0)$, $ D(\hat P_+) = \mathcal U^* D(\hat P_1)$, and 
\begin{gather*}
    \hat X_- = \mathcal U^* \hat X_0 \mathcal U, \quad \hat X_+ = \mathcal U^* \hat X_1 \mathcal U, \\
    \hat P_- = \mathcal U^* \hat P_0 \mathcal U, \quad \hat P_+ = \mathcal U^* \hat P_1 \mathcal U.
\end{gather*}
Restricted to the Schwartz space $\mathcal S(S^1)$, these operators can be expressed in terms of the ladder operators from Section~\ref{secLadder}, viz., 
\begin{equation}
    \label{eqPosMom}
    \begin{aligned}
        \tilde X_-  &= \frac{1}{\sqrt{2\alpha}}(A_- +  A_-^+) = \frac{i}{\sqrt{2\alpha}}( L^* \partial_-^{1/2} + \partial_-^{1/2} L ), \\
        \hat P_-  &= -i \sqrt{\frac{\alpha}{2}} (A_- - A_-^+) = - \sqrt{\frac{\alpha}{2}} ( L^* \partial_-^{1/2} - \partial_-^{1/2}L), \\
        \tilde X_+  &= \frac{1}{\sqrt{2\alpha}}(A_+ +  A_+^+) = \frac{i}{\sqrt{2\alpha}}( L \partial_+^{1/2} + \partial_+^{1/2} L^* ), \\
        \tilde P_+  &= -i \sqrt{\frac{\alpha}{2}} (A_+ - A_+^+) = - \sqrt{\frac{\alpha}{2}} ( L \partial_+^{1/2} - \partial_+^{1/2}L^*), 
    \end{aligned}
\end{equation}
where $ \tilde X_- = \hat X_- \rvert_{\mathcal S(S^1)}$ and similarly for the other operators. In particular, it follows immediately from the commutation relationships for $A_\pm$ and $A_{\pm}^+$ established in Section~\ref{secLadder} that the operators introduced above exhibit canonical position--momentum commutation relationships, i.e., 
\begin{gather*}
    [ \tilde X_\pm, \tilde X_\pm ] = 0, \quad [ \tilde P_\pm, \tilde P_\pm ] = 0, \\
    [ \tilde X_-, \tilde P_- ] = 1, \quad [\tilde X_+, \tilde P_+] = 1, \\
    [ \tilde X_-, \tilde P_+ ] = 0, \quad [ \tilde X_+, \tilde P_-] = 0.
\end{gather*}
It is also worthwhile noting that $ \tilde X_\pm$ and $ \tilde P_\pm$ can be employed to define ``kinetic energy'' and ``potential energy'' operators for the circle rotation, respectively, 
\begin{displaymath}
    \mathfrak K = \frac{1}{2}( - \tilde P_-^2 + \tilde P_+^2 ), \quad \mathfrak V = \frac{\alpha^2}{2}( -\tilde X_-^2 + \tilde X_+^2  ).
\end{displaymath}
One can then verify that, up to a multiplication factor of $1/2$, $\mathfrak K $ is equal to the pullback under $ \mathcal U $ of the connection Laplacian $\Delta$ from Section~\ref{secConn}, $ \mathfrak K = \mathcal U \Delta \mathcal U^* / 2$, while $ \mathfrak V $ is equal to the pullback of the multiplication operator on $\mathcal H$ that multiplies by the potential function $v$, i.e., $ \mathfrak V = \mathcal U^* T_{v} \mathcal U$. 

Inspecting the formulas for $\tilde X_\pm$ and $ \tilde P_\pm$ in~\eqref{eqPosMom}, it is evident that the position and momentum operators for the circle rotation differ fundamentally from their counterparts $\hat X_j$ and $ \hat P_j $ for the quantum harmonic oscillator since (i) unlike $\hat X_j $ and $ \hat P_j $, which are all local operators, $\hat X_\pm$ and $ \hat P_\pm $ are all non-local; and (ii) unlike $\hat X_j $, which are multiplication operators by the corresponding coordinate functions $x^j $, $ \hat X_\pm $ are not multiplication operators. Here, by an operator $A$ on functions being \emph{local}, we mean that for a given function $ f \in D(A)$, the evaluation of the function $ A f $ at each point on its domain of definition depends only on the values of $ f $ in an arbitrarily small neighborhood of that point. Examples of operators meeting this condition are multiplication operators and derivative operators of integral order, such as $\hat X_j$ and $\hat P_j$. On the other hand, given a function $ f : S^1 \to \mathbb C $ of appropriate smoothness, the result of the fractional derivative $ \partial^{1/2} f(\theta) $ depends on the behavior of $ f $ at distant points from $\theta$ (see Appendix~\ref{appFractional}), and as a result $ \hat X_{\pm} $ and $ \hat P_\pm $ are non-local. In fact, the operators $ \hat X_\pm$ and $ \hat P_\pm$ generate fractional diffusion semigroups on the circle.  

Still, despite these differences, the fact that $\tilde X_\pm$ and $ \tilde P_\pm$ obey canonical commutation relationships means that many results for the quantum harmonic oscillator which are a consequence of these relationships carry over to the circle rotation. As an example, we mention here the position-momentum uncertainty relationships, which hold for the $ \tilde X_\pm$ and $\tilde P_\pm $ operators analogously to $ \tilde X_j $ and $ \tilde P_j $. That is, letting $ \sigma_{jk} = \langle \psi_{jk}, \cdot \rangle_{\mathcal H} \psi_{jk} $ be the pure quantum state in $ Q( \mathcal H)$ associated with eigensection $ \psi_{jk}$, it follows from standard results on quantum harmonic oscillators (e.g., Refs.~\cite{Sakurai93,Hall13}) that 
\begin{align*}
    \var_{\sigma_{j0}} \hat X_0^2 &:= \mathbb E_{\sigma_{j0}} \hat X_0^2 - ( \mathbb E_{\sigma_{j0}} \hat X_0)^2 = \frac{1}{\alpha}\left( j + \frac{1}{2}\right),\\
    \var_{\sigma_{j0}} \hat P_0^2 &:= \mathbb E_{\sigma_{j0}} \hat P_0^2 - ( \mathbb E_{\sigma_{j0}} \hat P_0)^2 = \alpha\left( j + \frac{1}{2}\right),
\end{align*}
and similarly that
\begin{displaymath}
    \var_{\sigma_{0j}} \hat X_1^2 = \frac{1}{\alpha}\left( j + \frac{1}{2}\right), \quad \var_{\sigma_{0j}} \hat P_1^2 = \alpha\left( j + \frac{1}{2}\right),
\end{displaymath}
leading to the position-momentum uncertainty relationships
\begin{align}
    \nonumber
    (\var_{\sigma_{j0}} \hat X_0^2)(\var_{\sigma_{j0}} \hat P_0^2) & = (\var_{\sigma_{0j}} \hat X_1^2)(\var_{\sigma_{0j}} \hat P_1^2)\\ 
    \label{eqUncertaintyH} &= \left( j + \frac{1}{2} \right)^2. 
\end{align}
Correspondingly, for the quantum state $ \rho_j = \langle \phi_j, \cdot \rangle_\mu \phi_j \in \mathcal Q(L^2(\mu)) $ of the circle rotation with $ j \leq 0$, we have 
\begin{displaymath}
    \var_{\rho_j} \hat X_- = \var_{\rho_j}( \mathcal U^* \hat X_0 \mathcal U) = \var_{\mathcal U \rho_j \mathcal U^*} \hat X_0 = \var_{\sigma_{-j,0}} \hat X_0,
\end{displaymath}
and $ \var_{\rho_j} \hat P_- = \var_{\sigma_{-j,0}} \hat P_0 $, while, for $ j > 0 $,
\begin{displaymath}
    \var_{\rho_j} \hat X_+ = \var_{\sigma_{0j}} \hat X_1, \quad \var_{\rho_j} \hat P_+ = \var_{\sigma_{0j}} \hat P_1,
\end{displaymath}
leading to analogous uncertainty relationships to~\eqref{eqUncertaintyH}, i.e.,
\begin{equation}
    \label{eqUncertaintyL2}
    \begin{gathered}
        (\var_{\rho_{j}} \hat X_-^2)(\var_{\rho_{j}} \hat P_-^2) = \left( j + \frac{1}{2} \right)^2, \quad j \leq 0,\\
        (\var_{\rho_{j}} \hat X_+^2)(\var_{\rho_{j}} \hat P_+^2) = \left( j + \frac{1}{2} \right)^2, \quad j > 0.
    \end{gathered}
\end{equation}
Note that in both~\eqref{eqUncertaintyH} and~\eqref{eqUncertaintyL2} the ground states, $ \psi_{00}$ and $\phi_0$, respectively, saturate the corresponding Sch\"rodinger uncertainty inequality \cite{Hall13}. With these results, we have completed the proof of Theorem~\ref{thmOps}.

\section{\label{secRKHA}Classical and quantum dynamics on reproducing kernel Hilbert algebras (RKHAs)}

What could be considered a shortcoming of the classical--quantum correspondence results established thus far is that the map $ \Omega_\tau $, mapping quantum observables in $\mathcal A(L^2(\mu))$ to classical observables in $C_\mathbb{R}(S^1)$, is not compatible with the natural Banach algebra homomorphism $ T : C_{\mathbb R}(S^1) \to \mathcal A(L^2(\mu)) $ mapping continuous functions to bounded multiplication operators on $L^2(\mu)$. To address this issue, in this section we shift attention from the dynamics (either classical or quantum) induced by the circle rotation on $L^2(\mu)$, and consider instead the dynamics on the RKHAs $ \hat{ \mathcal K}_\tau$ associated with order-$1/2$ fractional diffusions on the circle. After introducing the basic properties of these spaces (Section~\ref{secRKHAIntro}),  we consider aspects of classical and quantum dynamics (Sections~\ref{secRKHAClassical} and~\ref{secRKHAQuantum}), including the classical--quantum correspondence stated in Theorem~\ref{thmRKHS}. In particular, we will see that the reproducing property of $\hat{\mathcal K}_\tau$, which has no counterpart in the $L^2(\mu)$ setting, provides additional structure ensuring that the analogous operator to $ \Omega_\tau$ \emph{is} compatible with $T $. 

\subsection{\label{secRKHAIntro}RKHAs induced by fractional diffusions}

Following \cite{DasGiannakis19c}, for any $\tau > 0$, we consider the RKHS $\hat{\mathcal K}_\tau  $ on the circle associated with the reproducing kernel $\hat\kappa_\tau : S^1 \times S^1 \to \mathbb R_+$ (cf.~\ref{eqHeatKernel}),
\begin{equation}
    \label{eqFracHeatKernel}
    \begin{aligned}
        \hat \kappa_\tau( \theta, \theta') &= \sum_{j=-\infty}^\infty e^{- \lvert j \rvert \tau} \phi_j^*(\theta)\phi_j(\theta'), \\
        &= \sum_{j=-\infty}^\infty e^{-\lvert j \rvert \tau} e^{ij(\theta'-\theta)} \\
        &= \frac{\sinh \tau}{ \cosh\tau - \cos(\theta'-\theta)},
    \end{aligned}
\end{equation}
where we recognize the $\lvert j \rvert$ terms in $e^{-\lvert j \rvert \tau}$ as the eigenvalues of the order-$1/2$ fractional Laplacian,
\begin{displaymath}
    \mathcal L^{1/2} \phi_j = \lvert j \rvert \phi_j.
\end{displaymath}
One can verify that the sum over $j$ in~\eqref{eqFracHeatKernel} converges in any $C^r$ norm, $ r \in \mathbb N_0$, to a smooth positive-definite kernel on $S^1 \times S^1$, so that $ \hat{\mathcal K}_\tau$ is an RKHS of smooth functions. Moreover, $\hat \kappa_\tau$ is translation-invariant, and exhibits the analogous properties stated for the canonical heat kernel in Lemmas~\ref{lemUniversal} and~\ref{lemKernelNorm}. In particular, for every $ \tau > 0 $, the feature map $ \hat F_\tau : S^1 \to \hat{ \mathcal K}_\tau$ with $ \hat F_\tau(\theta) = \hat \kappa_\tau(\theta, \cdot )$ is continuous and injective, the image $\mathcal F(\hat{\mathcal K}_\tau) := \hat F_\tau(S^1) \subset \hat{\mathcal K}_\tau$ contains only nonzero functions, and $\hat{\mathcal K}_\tau $ lies dense in $C(S^1)$. In addition, we have:
\begin{prop}
    \label{propRKHA} For every $ \tau > 0$, the RKHS $ \hat{\mathcal K}_\tau$ is a unital Banach algebra of functions, i.e., there exists a constant $C_\tau $ such that 
    \begin{displaymath}
       \lVert f g \rVert_{\hat{\mathcal K}_\tau} \leq C_\tau \lVert f\rVert_{\hat{\mathcal K}_\tau}  \lVert g\rVert_{\hat{\mathcal K}_\tau}, \quad \forall f,g \in \hat{\mathcal K}_\tau.  
    \end{displaymath}
    Moreover, the space $\hat{\mathcal K}^\infty = \bigcap_{j=1}^\infty \hat{\mathcal K}_j$ is a dense, unital subalgebra of $\hat{\mathcal K}_\tau$, i.e., $ fg \in \hat{\mathcal K}^\infty $ for all $ f, g \in \hat{\mathcal K}^\infty$. 
\end{prop}
Proposition~\ref{propRKHA} follows from results in Ref.~\cite{DasGiannakis19c}. While we do not reproduce a proof here, it is worthwhile noting that \cite{DasGiannakis19c} relies heavily on the fact that the set of Fourier functions $ \phi_j$ is closed under multiplication, $ \phi_j \phi_k = \phi_{j+k}$, and is bounded in $C(S^1)$ norm uniformly with respect to $j$, $\lVert \phi_j \rVert_{C(S^1)} = 1 $.

\begin{rk*}
    In the case of the RKHSs $ \mathcal K_\tau$ associated with the standard heat kernel, Ref.~\cite{DasGiannakis19c} has only established a weaker result than Proposition~\ref{propRKHA}; namely, a H\"older-like inequality,
    \begin{displaymath}
        \lVert f g \rVert_{\mathcal K_\tau} \leq \lVert f \rVert_{\mathcal K_{\tau_1}} \lVert g  \rVert_{\mathcal K_{\tau_2}}, \quad \frac{1}{\tau} \geq \frac{1}{\tau_1} + \frac{1}{\tau_2}. 
    \end{displaymath}
    The failure of $\mathcal K_\tau$ to obey a result analogous to Proposition~\ref{propRKHA} can be traced to the quadratic increase of the Laplacian eigenvalues $ j^2$, which imposes a stronger constraint on the decay of the expansion coefficients $c_j $ for $ f = \sum_{j=-\infty}^\infty c_j \phi_j \in \mathcal K_\tau $ than $ f \in \hat{\mathcal K}_\tau$, where the eigenvalues $\lvert j \rvert$ of $\mathcal L^{1/2}$ grow linearly with $ j $. In effect, it appears that the decay of the exponential terms $e^{-\lvert j \rvert \tau}$ in~\eqref{eqFracHeatKernel} is fast-enough for $\hat{\mathcal K}_\tau$ to contain only smooth functions, yet slow-enough for it to be a Banach algebra. It should be noted that while $\mathcal K_\tau$ has not been shown to have the structure of a Banach algebra, the space $ \mathcal K^\infty = \bigcap_{j=1}^\infty \mathcal K_\tau$ is a unital algebra of functions. 
\end{rk*}

\subsection{\label{secRKHAClassical}Classical dynamics}

We now turn attention to the properties of the Koopman operators $U^t$ acting  on functions in $ \hat{\mathcal K}_\tau$. For that, it is worthwhile to begin by noting that a general RKHS need not be invariant under a dynamical flow $ \Phi^t$, even if that flow is smooth. Intuitively, this is because membership of a function $f$ in an RKHS imposes a stringent condition on the expansion coefficients of $f$ in a natural orthonormal basis for the kernel (e.g., \eqref{eqRKHSDecay}), and as the dynamics can deform the level sets of functions, $ f \circ \Phi^t$ need not satisfy those conditions. Nevertheless, the kernels  $ \hat \kappa_\tau$ possess an important special property, namely that they admit the Mercer representation in~\eqref{eqFracHeatKernel} in terms of Koopman eigenfunctions.  This turns out to be sufficient for $ f \circ \Phi^t$, $f \in \hat{\mathcal K}_\tau$, to lie in $ \hat{\mathcal K}_\tau$, so that one can define groups of Koopman operators $ \hat{\mathcal K}_\tau \to \hat{\mathcal K}_\tau$, which turn out to be strongly continuous and unitary.  Below, we state some of the properties of these groups in the form of a theorem, as we have not seen them stated elsewhere in the literature. 

\begin{thm}
    \label{thmRKHS2}Let $ \Phi^t : S^1 \to S^1$ be the circle rotation with frequency $\alpha$, $ \hat \kappa_\tau : S^1 \times S^1 \to \mathbb R $ the fractional heat kernel from~\eqref{eqFracHeatKernel} at time parameter $ \tau >0 $, and $ \hat{\mathcal K}_\tau$ the corresponding RKHA of $ \mathbb C$-valued functions. Define the Sobolev-like space
    \begin{displaymath}
        \hat{\mathcal K}_\tau^1 = \left\{ \sum_{j=-\infty}^\infty c_j \phi_j \in \hat{\mathcal K}_\tau : \sum_{j=-\infty}^\infty \lvert j \rvert^2 e^{\lvert j \rvert \tau }\lvert c_j \rvert^2 < \infty \right\}.   
    \end{displaymath}
    Then, the following hold for every $ \tau >0 $:
    \begin{enumerate}[(i), wide]
        \item For all $ t \in \mathbb R$, $ \hat{ \mathcal K}_\tau $, and thus $ \hat{\mathcal K}^\infty$, are invariant under $ \Phi^t $.
        \item The group of Koopman operators $ U^t : \hat{ \mathcal K}_\tau \to \hat{\mathcal K}_\tau$, with $ t \in \mathbb R$ and $ U^t f = f \circ \Phi^t$, is a strongly continuous, unitary group. 
        \item The skew-adjoint generator $V : D(V) \to \hat{\mathcal K}_\tau$ of the Koopman group on $\hat{\mathcal K}_\tau$  has domain $D(V) = \hat{\mathcal K}_\tau^1$, and acts as a derivation on $\hat{\mathcal K}^\infty$.
    \end{enumerate}
\end{thm}

A proof of Theorem~\ref{thmRKHS2} can be found in Appendix~\ref{appRKHS2}. Note that the result in Claim~(iii) that $V$ acts as a derivation on $\hat{\mathcal K}^\infty$ is reminiscent of the result of ter Elst and Lema\'nczyk \cite{TerElstLemanczyk17} in the $L^2$ setting, which was stated in Section~\ref{secCircle}. It is also worthwhile noting that analogous results to Theorem~\ref{thmRKHS2}(i, ii) hold for the Koopman groups on $\mathcal K_\tau$, and it can also be shown that the generators of these groups have domain $ \mathcal K_\tau^1$, defined analogously to $ \hat{\mathcal K}_\tau^1$. However, our proof that $V$ acts as a derivation on $\hat{\mathcal K}^\infty $ makes use of the Banach algebra structure of $\hat{\mathcal K}_\tau$, so it does not carry over in an obvious way to show that the generator acts as a derivation on $\mathcal K^\infty$ (even though that space is invariant under $U^t$).         

\subsection{\label{secRKHAQuantum}Quantum dynamics and classical-quantum correspondence} 

We now study the quantum dynamics of the circle rotation associated with the RKHAs $\hat{\mathcal K}_\tau$ and its correspondence with the classical dynamics described in Section~\ref{secRKHAClassical}. Starting with the basic definitions, as space of classical observables we consider the RKHA $\hat{\mathcal K}_{\tau,\mathbb R}$ consisting of the real elements of $\hat{\mathcal K}_\tau$ for some $ \tau > 0 $. Note that $\hat{\mathcal K}_{\tau,\mathbb R}$ is a dense subalgebra of the Banach algebra $C_{\mathbb R}(S^1)$ of classical observables employed in Sections~\ref{secQuantum} and~\ref{secCCO},  possessing the additional Hilbert space structure. As spaces of regular quantum states, observables, and generalized quantum states we consider $\mathcal Q(\hat{\mathcal K}_\tau)$, $\mathcal A(\hat{\mathcal K}_\tau)$, and $\mathcal A'(\hat{\mathcal K}_\tau)$, respectively, defined analogously to their counterparts $\mathcal Q(L^2(\mu))$, $\mathcal A(L^2(\mu))$, and $\mathcal A'(L^2(\mu))$ from Section~\ref{secQuantumCircle}. The dynamics on these spaces are governed by unitary operators $U^t : \hat{\mathcal K}_{\tau,\mathbb R} \to \hat{\mathcal K}_{\tau,\mathbb R} $ (the Koopman operators) and isometries $ \hat \Phi^t : \mathcal Q(\hat{\mathcal K}_\tau) \to \mathcal Q(\hat{\mathcal K}_\tau)$, $\hat U^t : \mathcal A(\hat{\mathcal K}_\tau) \to \mathcal A(\hat{\mathcal K}_\tau)$, and $\hat U^{t\prime} : \mathcal A'(\hat{\mathcal K}_\tau) \to \mathcal A'(\hat{\mathcal K}_\tau)$, defined analogously to $ \tilde \Phi^t : \mathcal Q(L^2(\mu)) \to \mathcal Q(L^2(\mu)) $, $ \tilde U^t : \mathcal{A}(L^2(\mu)) \to \mathcal A(L^2(\mu))$, and $ \tilde U^{t\prime}: \mathcal A'(L^2(\mu)) \to \mathcal A'(L^2(\mu))$, respectively. 

Next, similarly  to $\mathbb K_\tau : \mathcal P(S^1) \to \mathcal{ K}_\tau$ and $\Pi: C(S^1)\setminus \{0\} \to \mathcal Q(L^2(\mu))  $, we define an RKHA embedding of probability measures $\hat{\mathbb K}_\tau : \mathcal P(S^1) \to \hat{\mathcal K}_{\tau,\mathbb R}$,
\begin{displaymath}
    \hat{\mathbb K}_\tau(m) = \int_{S^1} \hat\kappa_\tau(\theta,\cdot) \, dm(\theta),
\end{displaymath}
and a map $ \hat \Pi_\tau : \hat{\mathcal K}_\tau \setminus \{ 0 \} \to \mathcal Q(\hat{\mathcal K}_\tau)$, 
\begin{equation}
    \label{eqHatPi}
    \hat\Pi_\tau(f) = \frac{\langle f, \cdot \rangle_{\hat{\mathcal K}_\tau}f }{\lVert f \rVert^2_{\hat{\mathcal K}_\tau}},
\end{equation}
mapping nonzero RKHA functions into pure quantum states. The composition $ \hat \Psi_\tau = \hat \Pi_\tau \circ \hat F_\tau $ then maps classical states in the circle  to pure quantum states in $\mathcal Q(\hat{\mathcal K}_\tau)$. Note that by the reproducing property of $\hat{\mathcal K}_\tau$, $\langle \hat F_\tau(\theta),\cdot \rangle_{\hat{\mathcal K}_\tau}$ is equal to the evaluation functional $\mathbb V_\theta : \hat{\mathcal K}_\tau \to \mathbb C $ at $\theta$, so that
\begin{displaymath}
    \hat \Psi_\tau(\theta) = \frac{\hat \kappa_\tau(\theta,\cdot) \mathbb V_\theta }{\hat \kappa_\tau(\theta,\theta)}.
\end{displaymath}
That is, $\hat \Psi_\tau(\theta)$ acts on functions in $\hat{\mathcal K}_\tau$ by ``reading off'' their value at $\theta$, and multiplying the result by the normalized kernel section $\hat \kappa_\tau(\theta,\cdot) / \hat \kappa_\tau(\theta,\theta) $. Equipped with this map, we define the linear operator $\Omega_\tau : \mathcal{A}(\hat{\mathcal K}_\tau) \to \hat{\mathcal K}_{\tau,\mathbb R} $ mapping quantum mechanical observables in $ \mathcal{A}(\hat{\mathcal K}_\tau)$ to classical observables in $\hat{\mathcal K}_{\tau,\mathbb R}$ according to the formula (cf.\ Proposition~\ref{propS1Quantum}(ii))
\begin{displaymath}
    ( \hat\Omega_\tau A)(\theta) = \mathbb E_{\hat \Psi_\tau(\theta)} A.
\end{displaymath}
In addition, we have the transpose map $\hat \Omega_\tau' : \hat{\mathcal K}_{\tau,\mathbb R}' \to \mathcal A'(\hat{\mathcal K}_\tau)$, $\hat \Omega_\tau' J = J \circ \hat \Omega_\tau$, where we note that $\hat \Omega_\tau'$ can be equivalently defined as a map on $\hat{\mathcal K}_{\tau,\mathbb R}$ by the canonical isomorphism between Hilbert spaces and their duals. 

Defining, further, the unitary map $\hat{\mathcal V}_\tau : \hat{\mathcal K}_\tau \to L^2(\mu)$ by polar decomposition of the integral operator $\hat K_\tau : L^2(\mu) \to \hat{\mathcal K}_\tau$ associated with $\hat \kappa_\tau$ (cf.\ $\mathcal V_\tau $ in~\eqref{eqPolarDecomp}), we can proceed as in Section~\ref{secIso} to construct an isometric embedding $\hat{\mathcal W}_\tau : \hat{\mathcal K}_\tau \to \mathcal H$ of the RKHA $\hat{\mathcal K}_\tau$ into the Hilbert space $\mathcal H $ of sections on Minkowski space, as well as a Banach algebra homomorphism $\hat{\mathcal W}_\tau : \mathcal A(\mathcal H) \to \mathcal A(\hat{\mathcal K}_\tau)$, and linear isometries $\hat{\mathcal W}^+_\tau : \mathcal Q(\hat{\mathcal K}_\tau) \to \mathcal Q(\mathcal H)$ and $\hat{\mathcal W }'_\tau : \mathcal A'(\hat{\mathcal K}_\tau) \to \mathcal A'(\mathcal H)$. These maps have analogous properties to those established in the $L^2(\mu)$ case, as depicted in Figure~\ref{figCommutKH} using commutative diagrams. 

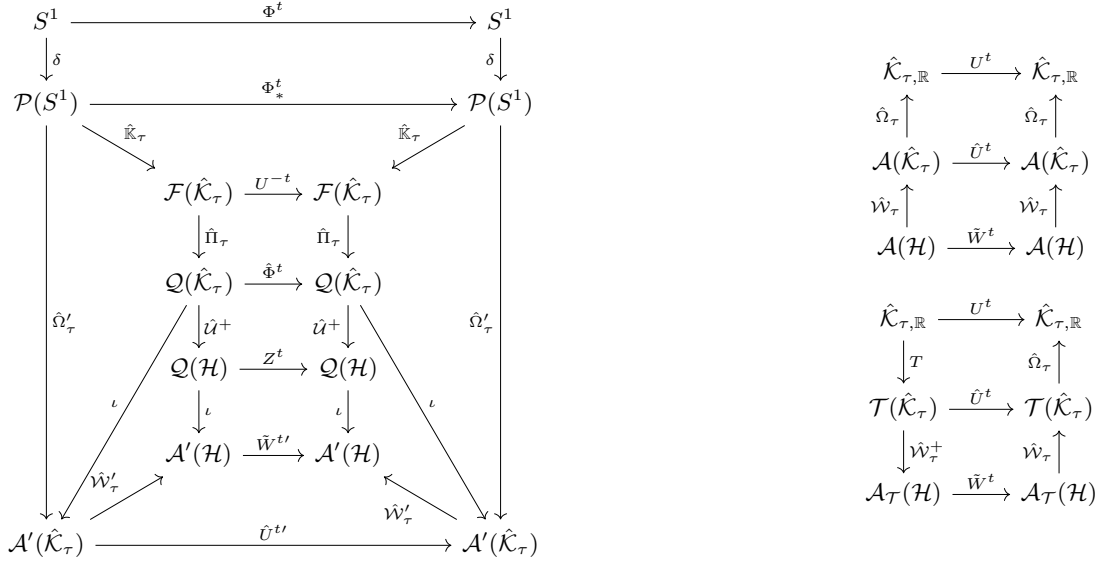
\begin{figure*}
    \begin{minipage}{.65\linewidth}
        \begin{displaymath}
            \begin{tikzcd}
                S^1 \arrow{d}{\delta} \arrow{rrr}{\Phi^t} & & & S^1 \arrow[swap]{d}{\delta} \\
                \mathcal{P}(S^1)  \arrow{rrr}{\Phi^t_*} \arrow{ddddd}{\hat\Omega'_{\tau}} \arrow{dr}{\hat{\mathbb K}_\tau} & & &   \mathcal{P}(S^1)  \arrow[swap]{dl}{\hat{\mathbb K}_\tau} \arrow[swap]{ddddd}{\hat\Omega'_{\tau}} \\
                & \mathcal F(\hat{\mathcal K}_{\tau}) \arrow{d}{\hat\Pi_\tau} \arrow{r}{U^{-t}} & \mathcal F(\hat{\mathcal K}_{\tau}) \arrow[swap]{d}{\hat\Pi_\tau} \\
                & \mathcal{Q}(\hat{\mathcal K}_\tau) \arrow{d}{\hat{\mathcal U}^+} \arrow[swap]{dddl}{\iota} \arrow{r}{\hat \Phi^t} & \mathcal{Q}(\hat{\mathcal K}_\tau) \arrow{dddr}{\iota} \arrow[swap]{d}{\hat{\mathcal U}^+} \\
                & \mathcal Q(\mathcal H) \arrow{d}{\iota} \arrow{r}{Z^t} & \mathcal Q(\mathcal H) \arrow[swap]{d}{\iota}\\
                & \mathcal A'(\mathcal H) \arrow{r}{\tilde W^{t\prime}} & \mathcal A'(\mathcal H) \\
                \mathcal A'(\hat{\mathcal K}_\tau) \arrow{ur}{\hat{\mathcal W}'_\tau} \arrow{rrr}{\hat U^{t\prime}}  & & & \mathcal A'(\hat{\mathcal K}_\tau) \arrow{ul}{\hat{\mathcal W}'_\tau} 
            \end{tikzcd}
        \end{displaymath}
    \end{minipage} \hfill
    \begin{minipage}{.3\linewidth}
        \begin{displaymath}
            \begin{tikzcd}
                \hat{\mathcal K}_{\tau,\mathbb R} \arrow{r}{U^t} & \hat{\mathcal K}_{\tau,\mathbb R} \\
                \mathcal A(\hat{\mathcal K}_\tau) \arrow{r}{ \hat U^t} \arrow{u}{\hat\Omega_\tau} & \mathcal A(\hat{\mathcal K}_\tau) \arrow{u}{\hat\Omega_\tau} \\
                \mathcal A(\mathcal H) \arrow{u}{\hat{\mathcal W}_\tau} \arrow{r}{\tilde W^t} & \mathcal A(\mathcal H) \arrow{u}{\hat{\mathcal W}_\tau}
            \end{tikzcd}
        \end{displaymath}
        \begin{displaymath}
            \begin{tikzcd}
                \hat{\mathcal K}_{\tau,\mathbb R} \arrow{r}{U^t} \arrow{d}{T} & \hat{\mathcal K}_{\tau,\mathbb R} \\
                \mathcal T(\hat{\mathcal K}_\tau) \arrow{d}{\hat{\mathcal W}_\tau^+} \arrow{r}{\hat U^t} & \mathcal T(\hat{\mathcal K}_\tau) \arrow{u}{\hat{\Omega}_\tau}\\
                \mathcal A_{\mathcal T}(\mathcal H) \arrow{r}{\tilde W^t} & \mathcal A_{\mathcal T}(\mathcal H) \arrow{u}{\hat{\mathcal W}_\tau}
            \end{tikzcd}
        \end{displaymath}
    \end{minipage}
    \caption{\label{figCommutKH}As in Fig.~\ref{figCommutL2H}, but for commutative diagrams showing the relationships between the classical/quantum dynamics associated with the RKHAs $\hat{\mathcal K}_\tau$ on the circle and the quantum harmonic oscillator on Minkowski space. In the commutative diagram in the bottom right, $\mathcal T(\hat{\mathcal K}_\tau)$ denotes the closed subalgebra of $\mathcal A(\hat{\mathcal K}_\tau)$ consisting of (bounded) multiplication operators by functions in $\hat{\mathcal K}_{\tau,\mathbb R}$, and $\mathcal A_{\mathcal T}(\mathcal H)$ the image of $\mathcal T(\hat{\mathcal K}_\tau)$ in $\mathcal A(\mathcal H)$ under the map $\hat{\mathcal W}_\tau^+ : A \mapsto \hat{\mathcal W}_\tau A \hat{\mathcal W}_\tau^*$. Note that such a diagram cannot be drawn in the $L^2(\mu)$ setting of Fig.~\ref{figCommutKH}.}
\end{figure*}

Several of these properties are stated as claims in Theorem~\ref{thmRKHS}(i--iii), and we will not repeat their derivation as the arguments are entirely analogous to those in Section~\ref{secQuantum}. Instead, we will focus on the properties of the classical--quantum correspondence associated with $\hat{\mathcal K}_\tau$ which are not present in the $L^2(\mu)$ setting, namely Claims~(iv,v) of Theorem~\ref{thmRKHS}. To that end, observe that by~\eqref{eqHatPi}, a quantum mechanical observable $A \in \mathcal A(\hat{\mathcal K}_\tau)$ maps under $\hat \Omega_\tau$ to the classical observable $f_A \in \hat{\mathcal K}_{\tau,\mathbb R}$ given by
\begin{align*}
    f_A(\theta) &= ( \hat \Omega_\tau A )(\theta) = \tr( \hat \Psi_\tau(\theta) A) \\
    &= \frac{\sum_{j=-\infty}^\infty \langle \hat \phi_{j,\tau}, \hat \kappa_\tau(\theta,\cdot) \mathbb V_\theta( A \hat \phi_{j,\tau} ) \rangle_{\hat{\mathcal K}_\tau}}{\hat\kappa_\tau(\theta,\theta)} \\
    &= \frac{\sum_{j=-\infty}^\infty ( \mathbb V_\theta \hat \phi_{j,\tau})^*(  \mathbb V_\theta(A \hat \phi_{j,\tau} ) ) }{\hat\kappa_\tau(\theta,\theta)} \\
    &= \frac{\sum_{j=-\infty}^\infty \hat\phi_{j,\tau}^*(\theta) (A \hat \phi_{j,\tau})(\theta)}{ \sum_{k=-\infty}^\infty \hat\phi^*_{k,\tau}(\theta) \hat\phi_{k,\tau}(\theta)  }. 
\end{align*}
As a result, whenever $A$ is a multiplication operator, i.e., $ A : f \mapsto gf  $ for some $ g \in \hat{\mathcal K}_{\tau,\mathbb R} $, we have $(A\hat\phi_{j,\tau})(\theta) = g(\theta) \hat \phi_{j,\tau}(\theta)$, and the expression above simplifies to
\begin{displaymath}
    f_A(\theta) = g(\theta). 
\end{displaymath}
Because $A$ was arbitrary, we deduce that $\hat \Omega_\tau $ is compatible with the natural Banach algebra homomorphism $T: \hat{\mathcal K}_{\tau,\mathbb R} \to \mathcal{A}(\hat{\mathcal K}_\tau)$ mapping classical observables to their corresponding multiplication operators. That is, $ \hat \Omega_\tau \circ T = \Id_{\hat{\mathcal K}_{\tau,\mathbb R}}$, as stated in Claim~(iv) of Theorem~\ref{thmRKHS}. The dynamical compatibility condition in Claim~(v) then follows directly from the definition of the operators involved, and leads to an additional commutative diagram, displayed in the bottom right of Fig.~\ref{figCommutKH}.

Figure~\ref{figPsiFrac} displays the evolution of the wavefunction $\hat \psi^{\sigma_x}_{\theta,\tau} $ on Minkowski space underlying the quantum state $ \hat \sigma_{\theta,\tau} = \hat{\mathcal W}_\tau \hat \Psi_\tau(\theta) \hat{\mathcal W}_\tau^*$, defined analogously to $\psi^{\sigma_x}_{\theta,\tau}$ from Section~\ref{secIso}. Here, we have $\hat \psi_{\theta,\tau}^{\sigma_x} = \zeta_{\sigma_x}^{-1} \hat \psi_{\theta,\tau} $, where $ \hat \psi_{\theta,\tau} \in \mathcal H$ is the section given by (cf.~\eqref{eqWF})
\begin{equation}
    \label{eqFracWF}
    \hat \psi_{\theta,\tau} = \sum_{j=0}^\infty e^{-\lvert j \rvert \tau / 2} (\phi_j(\theta) \psi_{j,0} + \phi^*_j(\theta) \psi_{0j}). 
\end{equation}
Th evolution of $ \hat \psi_{\theta,\tau}^{\sigma_x}$ exhibits a qualitatively similar behavior to the evolution of $\psi^{\sigma_x}_{\theta,\tau}$ in Fig.~\ref{figPsi}, but with the notable difference that (due to the slower decay of the exponential coefficients $e^{- \lvert j \rvert \tau / 2}$ in~\eqref{eqFracWF} compared to $e^{- j^2 \tau}$ in~\eqref{eqWF}) $ \hat \psi_{\theta,\tau}^{\sigma_x}$ develops smaller-scale oscillatory features than $ \psi_{\theta,\tau}^{\sigma_x}$.   

\begin{figure*}
    \includegraphics[width=\linewidth]{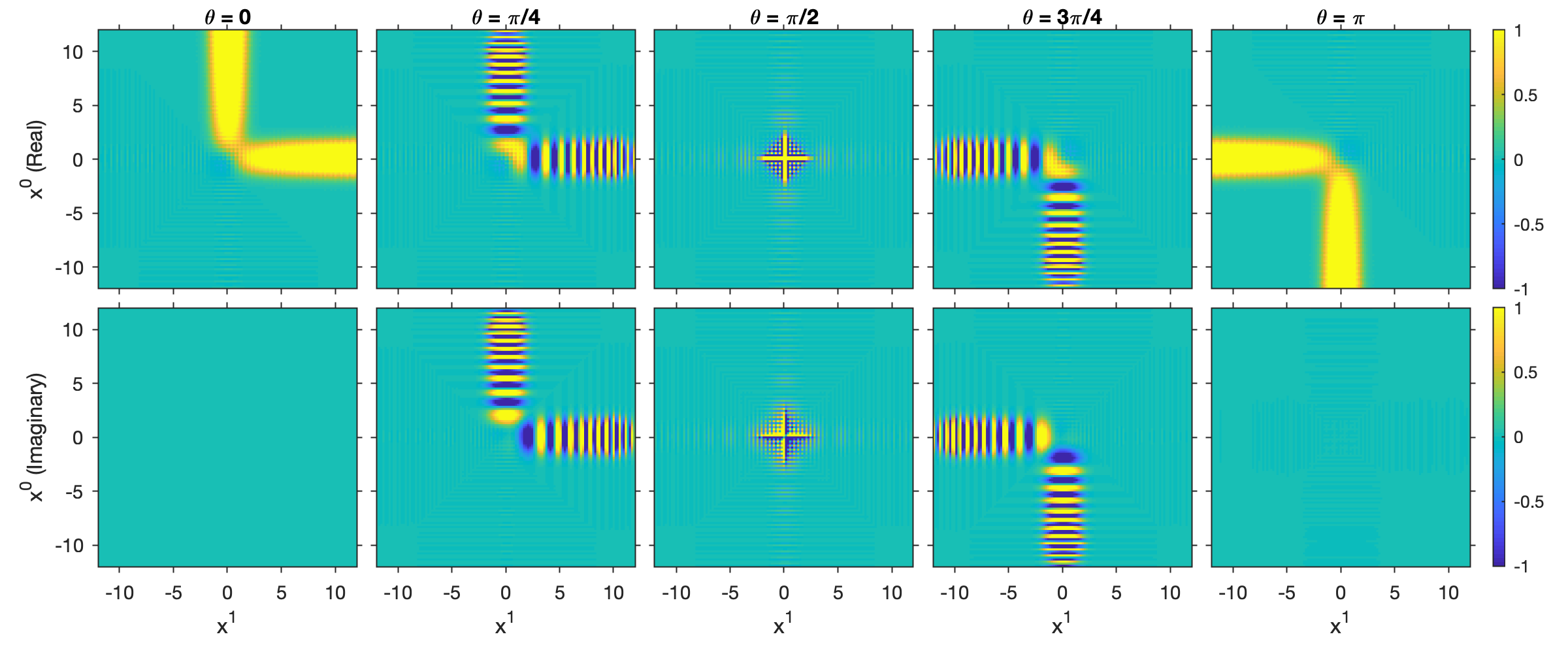}
    \caption{\label{figPsiFrac}As in Fig.~\ref{figPsi}, but for the wavefunction for the RKHA associated with the fractional diffusion on the circle. Notice the smaller-scale oscillatory features developing due to the linear growth of the eigenvalues of $\mathcal L^{1/2}$, as opposed to the quadratic growth of the eigenvalues of $\mathcal  L$.}
\end{figure*}

\section{\label{secConclusions}Concluding remarks}

The classical--quantum correspondences described in this paper, culminating with the RKHA-based formulation in Section~\ref{secRKHA}, provide a mutually consistent description of the dynamics of the circle rotation at three levels, namely the dynamical flow $\Phi^t : S^1 \to S^1$ on state space, the evolution of classical probability measures $ \Phi^t_* : \mathcal P(S^1) \to \mathcal P(S^1)$, and the evolution of quantum states $ \hat \Phi^t : \mathcal Q(\hat{\mathcal K}_\tau) \to \mathcal Q(\hat{\mathcal K}_\tau) $ on a reproducing kernel Hilbert space associated with fractional diffusions, having the additional structure of a $C^*$-algebra. Furthermore, these dynamics embed consistently into the evolution $ Z^t : \mathcal Q(\mathcal H) \to \mathcal Q(\mathcal H)$ of quantum states associated with an $\SO$ gauge theory for the harmonic oscillator on two-dimensional Minkowski space, whose dynamics are generated by a Laplace-type operator acting on sections of an associated $\mathbb C$-line bundle over Minkowski space, induced by a Lorentz-invariant connection satisfying the Yang-Mills equations. In particular, the $\SO$ action on the associated bundle is through a unitary, $\UU$, representation, so that the gauge theory employed in this work can be thought of as a two-dimensional analog of Maxwell electromagnetism. Using this correspondence, we have constructed ladder operators acting on classical observables of the circle rotation, which take the form of order-$1/2$ fractional derivatives, factorizing the positive- and negative- frequency components of the Koopman generator through number operators.       

The geometrical and algebraic structure of the classical--quantum correspondences identified in this work, occurring for a dynamical system as simple as the circle rotation, motivates further study of the connections between operator-theoretic ergodic theory, quantum mechanics, and gauge theory for more general systems, including systems with quasiperiodic or mixing dynamics, as well as systems with spatial structure (e.g., partial differential equation models), where non-Abelian structure groups may play a role.   

\begin{acknowledgements}
    This research was supported by ONR YIP grant N00014-16-1-2649 and NSF grant  DMS 1854383. The author would like to thank Suddhasattwa Das, Igor Mezi\'c, Mihai Putinar, and Joanna Slawinska for stimulating conversations.  He is also grateful to the Department of Computing and Mathematical Sciences at the California Institute of Technology and in particular his host, Andrew M.\ Stuart, for their hospitality and for providing a stimulating environment during a sabbatical, when part of this work was completed. 
\end{acknowledgements}

\appendix

\section{\label{appBundle}Fiber bundles and gauge theory}

In this appendix, we collect various definitions and results on the theory of fiber bundles and gauge theory supporting the analysis in the main text. We begin by considering standard results on general fiber bundles (Appendix~\ref{appBundleGeneral}), and then restrict attention to the principal and associated bundles over two-dimensional Minkowski space studied in this work (Appendix~\ref{appBundleMinkowski}). For additional details on this material we refer the reader to one of the many textbooks in the literature, e.g., Refs.~\cite{BerlineEtAl04,Michor08,RudolphSchmidt17}. 

\subsection{\label{appBundleGeneral}General fiber bundles}

Throughout this section, $ G $ will be a smooth Lie group with Lie algebra $\mathfrak g$, $ P \xrightarrow{\pi} M $ a smooth principal $G$-bundle, and $ E \xrightarrow{\pi_E} M $ a smooth associated bundle with typical fiber $ F$. Moreover, $ P_m = \pi^{-1}(\{m\}) $ and $ E_m = \pi_{E}^{-1}(\{ m \} ) $ will denote the fibers in $P$ and  $E $, respectively, over $ m \in M $. As in the main text, we will use the equivalent notations $R^\Lambda(p) \equiv p \cdot \Lambda$ for the right action at $ p \in P $ by group element $ \Lambda \in G$, and similarly $ L^\Lambda(f ) \equiv \Lambda \cdot f$ will represent the left action at $f \in F$.    

\subsubsection{Fiber-wise isomorphisms and metrics}

In the main text, we make use of the results stated in the following two lemmas. 
\begin{lem}
    \label{lemUnique}
    For any $ p \in P $ with $ \pi(p) = m $, the map $ \varepsilon_p : F \to E_m $, defined by $ \varepsilon_p( f ) = [ p, f ] $ is a bijection. Moreover, for every $ \Lambda \in G $, the property $ \varepsilon_{R^\Lambda(p)}^{-1} = L^{\Lambda^{-1}} \circ \varepsilon_p^{-1} $ holds.    
\end{lem}

\begin{proof}
    Suppose that $ \tilde f \in F $ is such that $ [ p, f ] = [ p, \tilde f] $. Then, by definition of the $ [ \cdot, \cdot ] $ equivalence classes, there exists $ \Lambda \in G $ such that $ p \cdot \Lambda = p$ and $ \Lambda^{-1} \cdot f = \tilde f $. However, because the action of $ G $ on $ P $ is free, $ \Lambda $ is equal to the identity element of $ G$, and $ \tilde f = f $. If now $ p' \in P_m $ and $ f'\in F$ are such that $ [ p, f ] = [ p', f' ] $, then there exists a unique $\Lambda' \in G$ such that $ p' = p \cdot \Lambda' $, and therefore $[p',f']=[p \cdot \Lambda', f' ] = [ p, \Lambda' \cdot f' ] $. It then follows that $ \Lambda \cdot f' = f $, meaning that there is a well defined map from $ E_m $ to $ F $ mapping $ [ p', f' ] $ to the unique $ f $ such that $ [ p', f' ] = [ p, f ] $. It is straightforward to verify that this map is the inverse of $ \varepsilon_p $, proving that $ \varepsilon_p$ is a bijection. 

    Next, for any $ e \in E_m $ and $ p \in P_m $, we have $ \varepsilon_p^{-1} e = f $, where $ f $ is the unique element of $ F $ such that $[p,f] = e$. Thus, for any $ \Lambda \in G $, 
    \begin{align*}
        \varepsilon^{-1}_{R^\Lambda(p)}( e ) &= \varepsilon^{-1}_{R^\Lambda(p)} ([p,f] ) = \varepsilon^{-1}_{p \cdot \Lambda}( [ p \cdot \Lambda, \Lambda^{-1} \cdot f ] ) \\& = \Lambda^{-1} \cdot f = L^{\Lambda^{-1}}( \varepsilon_p^{-1}( e ) ), 
    \end{align*}
    proving the second claim, and completing the proof of the lemma.
\end{proof}

\begin{lem}
    \label{lemG} Suppose that $ E \xrightarrow {\pi_E} M $ is a vector bundle, whose typical fiber $F$ is equipped with an inner product $ \langle \cdot, \cdot \rangle_F$, and $ G $ acts on $ F $ via a unitary representation $ \rho : G \to \mathrm{U}(F) $. Then, the fiber-wise metric $ g_m : E_m \times E_m \to \mathbb{C} $, constructed analogously to~\eqref{eqG} using a local section $ \sigma: U \to P $ whose domain $ U \subseteq M$ contains $ m$, is independent of the choice of $ \sigma $. 
\end{lem}

\begin{proof}
Let $ \tilde \sigma : \tilde U \to P$ be an arbitrary smooth section whose domain $\tilde U$ contains $m$, and $\iota_{\tilde \sigma, E } : \tilde U \times F \to E $ be the associated trivializing map, constructed analogously to $ \iota_{\sigma, E} $. We must show that for any $ e_1,e_2 \in E_m $, $ \tilde g_m( e_1, e_2 ) := \langle \varepsilon_{\tilde \sigma(m)}(e_1), \varepsilon_{\tilde \sigma(m)}( e_2 ) \rangle_F $ is equal to  $ g_m( e_1, e_2 ) = \langle \varepsilon_{\sigma(m)}(e_1), \varepsilon_{\sigma(m)}( e_2 ) \rangle_F $. For that, observe that $ \tilde g_m(e_1, e_2) = \langle \tilde f_1, \tilde f_2 \rangle_F$ and $ g_m(e_1,e_2) = \langle f_1, f_2 \rangle_F$, where, by Lemma~\ref{lemUnique}, $ \tilde f_j $ and $ f_j $ are unique elements in $ F $ such that $ [ \tilde \sigma( m ), \tilde f_j ] =  [ \sigma( m ), f_j ] = e_j $, $ j \in \{ 1, 2 \} $. Because $ G $ acts on $P$ freely, there exists a unique $ \Lambda \in G $ such that $ \tilde \sigma(m) = \sigma( m ) \cdot \Lambda $, and thus
    \begin{displaymath}
        [ \tilde \sigma( m ), \tilde f_j ] = [ \sigma( m ) \cdot \Lambda, \tilde f_j ] = [ \sigma( m ), \Lambda \cdot \tilde f_j ]. 
    \end{displaymath}
    Using again Lemma~\ref{lemUnique} it follows that $ \rho(\Lambda ) \tilde f_j =  \Lambda \cdot \tilde f_j = f_j $, and by unitarity of $ \rho(\Lambda ) $, 
    \begin{align*}
        \tilde g_m( e_1, e_2 ) &= \langle \tilde f_1, \tilde f_2 \rangle_F = \langle \rho( \Lambda ) \tilde f_1, \rho( \Lambda ) \tilde f_2 \rangle_F \\
        &= \langle f_1, f_2 \rangle_F = g_m( e_1, e_2 ),
    \end{align*}
    as claimed. 
\end{proof}

\subsubsection{\label{appConn}Connection 1-forms, covariant derivatives, and curvature tensors}

A \emph{connection 1-form} on a principal bundle $P \xrightarrow{\pi}M$ with structure group $ G $ and Lie algebra $\mathfrak{g} $ is a  Lie-algebra-valued 1-form $ \omega \in \Omega^1(P,\mathfrak{g})$, satisfying the following conditions for every point $ p \in  P $, group element $ \Lambda \in G$, Lie algebra element $\lambda \in \mathfrak g$, and fundamental vector field $W^{\lambda}$, and tangent vector $W \in T_pP$:   
\begin{align}
    \label{eqVert}\omega W^{\lambda}_p &= \lambda_a,\\    
    \label{eqAdj} (R^{\Lambda*} \omega )_p W &= \Ad_{\Lambda^{-1}} \omega_p W.
\end{align}
In~\eqref{eqAdj}, $R^{\Lambda*} : \Omega^1(P, \mathfrak{g} ) \to \Omega^1(P, \mathfrak{g})$  is the pullback map on $\mathfrak{g}$-valued 1-forms associated with the right action $R^\Lambda : P \to P $ by group element $\Lambda \in G$, and $ \Ad_{\Lambda^{-1}}  : \mathfrak{g} \to \mathfrak{g} $  is the representative of $\Lambda$ under the adjoint representation of $G$. The latter, is equal to the  pushforward map on tangent vectors to $ G $ associated with the map $ \AD_{\Lambda} : G \to G $, $\AD_{\Lambda^{-1}} \Lambda' = \Lambda \Lambda' \Lambda^{-1} $, evaluated at the identity $I$; that is, $\Ad_{\Lambda^{-1}} = \AD_{\Lambda^{-1} *, I} $.  If $G$ is Abelian, $\AD_{\Lambda} = \Id_G $ and $\Ad_{\Lambda} = \Id_{\mathfrak g}$ for all $ \Lambda \in G $, and \eqref{eqAdj} reduces to the $G$-invariance condition in~\eqref{eqConnCond}. 

Every such connection 1-form is equivalent to assigning a \emph{horizontal distribution} on $TP$; that is, a smooth assignment of a subspace $H_p P \subseteq T_p P $, called \emph{horizontal subspace}, at every $p \in P$, such that $H_p P \oplus V_p P = T_p P$, and the equivariance condition $ R^\Lambda_* H_p P = H_{R^\Lambda(p)}P $ holds for every $ \Lambda \in G$. In particular, at any $p \in P$, $H_pP $ is equal to the kernel of $ \omega_p$, and we have the vertical and horizontal projection maps $ \ver_p T_p P \to T_p P $ and $ \hor_p T_p P \to T_p P $, where $ \ran \ver_p = V_p P $, $ \ran \hor_p = H_p P $, and
\begin{displaymath}
    \ver_p X = W^{\omega_p X}, \quad \hor_p X = X - \ver_p X.
\end{displaymath}
The pointwise maps $\ver_p$ and $\hor_p$ naturally extend to maps $ \ver : \Gamma(E) \to \Gamma(E)$ and $\hor : \Gamma(E) \to \Gamma(E)$ on sections, respectively; cf.~\eqref{eqHor}.

The horizontal distribution also induces a notion of \emph{parallel transport} of curves on the base space $M$ to curves on the total space $P$. In particular, if $C : ( -1, 1 ) \to M $ is a smooth curve on the base space passing through $ C( 0 ) = m $, then for every point $ p \in P_{\pi^{-1}(m)} $ there exists a unique curve $C^P_p : ( -1, 1 ) \to P $ on the principal bundle, such that (i) $ C^P_p(0) = p $; (ii)  $\pi \circ C^P_p = C$; and (iii) for all $ s \in ( -1, 1 ) $, the tangent vector $\dot C^P_{p,p'} $ to $C^P_p $ at $ p' = C^P_p( s ) $  lies in $H_{p'} P $, and $ \pi_* \dot C^P_{p,p'} = \dot C_{m'} $, where $\dot C_{m'} \in T_{m'}M$ is the tangent vector to $C$ at $ m'=\pi(p') $. Given a vector field $ X \in \Gamma(TM)$ with integral curve $C_m : ( -1, 1 ) \to M $, $ m \in M $, parameterized such that $ C_m(0) = m $, the assignment $ p \mapsto \dot C^P_{p,p} \in T_p P $, $ p \in \pi^{-1}(\{m\})  $, defines a smooth vector field $ Y \in \Gamma(TP)  $ with $ Y_p = \dot C^P_{p,p} $, called the \emph{horizontal lift} of $X$. By construction, the vector field $ Y$ is projectible under $ \pi_* $, satisfying $ \pi_* Y = X$. It should be noted that analogous notions of horizontal subspace, parallel transport, and horizontal lift of vector fields can be defined for any associated bundle to $P \xrightarrow{\pi}M$, but we will not be needing these concepts here. 

A \emph{covariant derivative}, or \emph{connection}, on a real vector bundle $E \xrightarrow{\pi_E} M$  is a linear map $\nabla : \Gamma(E) \to \Omega^1(M,E)$ satisfying the Leibniz rule:
\begin{displaymath}
    \nabla(f s) = df \otimes s + f \nabla s, \quad \forall f \in C^\infty(M), \quad \forall s \in \Gamma(E).
\end{displaymath}
If $ X \in \Gamma(TM)$ is a vector field, $ \nabla$ induces a linear map $ \nabla_X : \Gamma(E) \to \Gamma(E) $, also referred to as covariant derivative, given by $ \nabla_X s = ( \nabla s)( X ) $. This map has the properties
\begin{equation}
    \label{eqCov2}
    \begin{aligned}
        \nabla_{fX}s & = f \nabla_X s, \\
        \nabla_{X+Y} s & = \nabla_X s + \nabla_Y s, \\
        \nabla_X(f s ) &= X( f ) s + f \nabla_X s,
    \end{aligned}
\end{equation}
for all $ f \in C^\infty(M)$, $ Y \in \Gamma(TM)$, and $ s \in \Gamma(E)$. It is straightforward to verify that $\nabla_X$ from~\eqref{eqCov0} satisfies these properties and is thus a covariant derivative operator. If $E\xrightarrow{\pi_E} M$ is complex, then the covariant derivative is constructed as above, but with $X$ and $Y$ taken to be sections of the complexified tangent bundle $T_{\mathbb C} M$. 

The covariant derivative $\nabla : \Gamma(E) \to \Omega^1(M,E) $ on sections induces covariant derivatives $\nabla : \Omega^{k}(M,E) \to \Omega^{k+1}(M,E)$ to $k$-forms with values in $E$. Specifically, given $s \in \Omega^k(M,E)$ and a collection $X_0,\ldots, X_k$ of vector fields in $\Gamma(TM)$, we define
\begin{multline}
    ( \nabla s )(X_0, \ldots, X_k) \equiv \nabla_{X_0,\ldots, X_k} s \\
    \begin{aligned}
        & \sum_{j=0}^k (-1)^j \nabla_{X_j}(s(X_0,\ldots, \hat X_j, \ldots X_k))\\ 
        &+ \sum_{0\leq j < l \leq k}(-1)^{j+l} s([X_j,X_l],X_0, \ldots, \hat X_j, \ldots,\\
    & \qquad \qquad \quad \hat X_l, \ldots, X_k),
    \end{aligned}
    \label{eqCovK}
\end{multline}
where $[\cdot, \cdot]$ is the Lie bracket of vector fields, and $\hat\ $ indicates a missing term. It can then be verified that the Leibniz rule 
\begin{displaymath}
    \nabla(w \wedge s ) = dw \wedge s + (-1)^l w \wedge \nabla s
\end{displaymath}
holds for any $ s \in \Omega^k(M,E)$ and $ w \in \Omega^l(M)$.

If $E\xrightarrow{\pi_E} M $ is an associated vector bundle to a principal bundle $P\xrightarrow{\pi} M$ with typical fiber $F$, onto which $G$ acts linearly through a representation $ \rho : G \to \GLF $ (i.e., $L^\Lambda = \rho(\Lambda)$), the definition for $\nabla $ in~\eqref{eqCovK} is consistent with first defining an exterior covariant derivative $\mathcal D : \Omega^k_G( P, F ) \to \Omega^k_G( P, F )$ on $G$-equivariant, $F$-valued $k$-forms on the principal bundle, and pulling back $\mathcal D $ to an operator on $E$-valued $k$-forms on the base space, analogously to the construction of $\nabla : \Gamma(E)\to \Omega^1(M,E)$ from $\mathcal D : C^\infty_G(P,F) \to \Omega^1_G( P, F )$ in Section~\ref{secConn}. 

Unlike standard exterior derivatives, $\nabla^2 $ and $\mathcal D^2$ are, in general, nonzero operators, which play an important role in defining notions of curvature for vector bundles and principal bundles. Focusing, for now, on the vector bundle $ E \xrightarrow{\pi_E}M$, it can be shown that $ \nabla^2 : \Gamma(E) \to \Omega^2(M,E) = \Gamma(\bigwedge^2 T^*M \otimes E)$, given according to \eqref{eqCovK} by
\begin{displaymath}
    \nabla^2_{X,Y} = \nabla_X \nabla_Y - \nabla_Y \nabla_X - \nabla_{[X,Y]}, \quad X,Y \in \Gamma(TM),
\end{displaymath}
acts as a multiplication operator by a 2-form taking values in the endomorphism bundle $ \End E \to M $ of $E$.  That is, we have
\begin{equation}
    \label{eqCurvR}
    \nabla^2_{X,Y} s = \mathsf R_{X,Y} s
\end{equation}
for some $ \mathsf R \in \Omega^2(M,\End E)$, and despite appearances, $\nabla^2$ is a zeroth-order differential operator since $R_{X,Y}$ is a $(1,1)$ tensor. The tensor field  $\mathsf R $ is called the \emph{curvature tensor}, or \emph{curvature endomorphism}, associated with the connection $\nabla$. It can further be shown that $\mathsf R$ satisfies a \emph{differential Bianchi identity}, 
\begin{displaymath}
    \nabla \mathsf R = 0,
\end{displaymath}
where $\nabla : \Omega^2( M, \End E ) \to \Omega^3(M,\End E)$ is a connection induced canonically from $\nabla : \Omega^2(M,E) \to \Omega^3(M,E)$ on tensor bundles of $E \xrightarrow{\pi_E} M$. If $\mathsf R$ vanishes, the connection $\nabla$ is said to be \emph{flat}.    

It should be noted that $\nabla^2$ is different from the second covariant derivative $\mathsf H : \Gamma(TM) \to \Gamma(T^* M \otimes T^*M \otimes E)$ in Section~\ref{secLapl} in that the former takes values in the product bundle of $E$ and the second exterior power of $T^*M$ (consisting of antisymmetric $(0,2)$ tensors), whereas the latter takes values in the product bundle of $E $ and the second tensor power of $T^*M$ (consisting of general $(0,2)$ tensors). Nevertheless, the two operators are related, in the sense that $ \nabla^2$ represents the antisymmetric component of $\mathsf H$. Specifically, because the Levi-Civita connection is \emph{torsion-free}, i.e., 
\begin{displaymath}
    \nabla^{\text{LC}}_X Y - \nabla^{\text{LC}}_Y X = [ X, Y ], \quad \forall X,Y \in \Gamma(TM),
\end{displaymath}
it follows that 
\begin{displaymath}
    \nabla^2_{X,Y} = \mathsf H_{X,Y} - \mathsf H_{Y,X}.
\end{displaymath}

\subsubsection{Adjoint bundles}

First, we describe the \emph{nonlinear adjoint bundle}  $ \AD P \xrightarrow{\pi_{{\AD P }}} M $  associated to the principal bundle $ P \xrightarrow{\pi} M$. $ \AD P \xrightarrow{\pi_{{\AD P }}} M $ is defined as the fiber bundle over $M $ with typical fiber $F = G$, and left group action on $F $ given by $ \Lambda \cdot f = \AD_\Lambda f $, where $\Lambda, f \in G$. A section $ \varsigma \in \Gamma(\AD P )$ defines a function $\tilde \xi :P \to G $ such that $ \tilde \xi(p)$ is the unique element of $G$ satisfying $\varsigma(m) = [ p, \tilde \xi(p) ] $. Moreover, for every $\Lambda \in G$, we have
\begin{displaymath}
    [p, \tilde \xi(p)] = [ p \cdot \Lambda, \Lambda^{-1} \cdot \tilde \xi(p) ] = [ p \cdot \Lambda, \AD_{\Lambda^{-1}} \tilde \xi(p) ].
\end{displaymath}
Because $ [p \cdot \Lambda, \AD_{\Lambda^{-1}} \tilde \xi(p)] = [ p \cdot \Lambda, \tilde \xi(p \cdot \Lambda)]$, it follows that  $ \tilde \xi(p \cdot \Lambda) = \AD_{\Lambda^{-1}} \tilde \xi(p)$, and thus $ \Lambda \cdot \tilde \xi(p \cdot \Lambda) = \tilde \xi(p) \cdot \Lambda$. The latter relation implies in turn that $ \varphi_{\varsigma} : P \to P $ with $ \varphi_\varsigma(p) = p \cdot \tilde \xi(p)$ satisfies the commutative diagram in~\eqref{eqGaugeCommute}, and is thus a gauge transformation. Conversely, starting from a gauge transformation $ \varphi : P \to P $, one can construct a section $ \varsigma_{\varphi} \in \Gamma(\AD P)$ by reversing the steps leading to $ \varphi_\varsigma$ from $ \varsigma $, so that every gauge transformation $ \varphi $ induces a section $ \varsigma_{\varphi}$. One can further verify that $ \varsigma = \varsigma_{\varphi_{\varsigma}}$ and $ \varphi = \varphi_{\varsigma_{\varphi}}$, so that the gauge group $ \mathcal{G}$ can be identified with $ \Gamma(\AD P)$.

If now $ e = [ p, f ] $, with $ p \in P $ and $ f \in F $, is a point on an associated bundle $ E \xrightarrow{\pi_E} M $ to $ P \xrightarrow{\pi}M $, lying above $ m = \pi(p) = \pi_E(e)$, the action $\varphi_*(e)$ of the pushforward map $ \varphi_* : E \to E $ associated with the gauge transformation $ \varphi$ is given by 
\begin{equation}
    \label{eqGaugeStar}
    \varphi_{*}(e) = [\varphi(p), f] = [ p \cdot \tilde \xi(p), f ] = [ p, \tilde \xi(p) \cdot f ].
\end{equation}
Thus, if $ \sigma :  U \to P $ is a trivializing section defined on $ U \subseteq M $, we have
\begin{equation}
    \label{eqGaugeStarSigma}
    \varepsilon^{-1}_{\sigma(m)}( \varphi_*(e)) = \tilde \xi(\sigma(m)) \varepsilon^{-1}_{\sigma(m)}(e).
\end{equation}

Note that if $G$ is Abelian, $ \AD_\Lambda = \Id $, and every equivalence class $ [ p, \Lambda ] \in \AD P $ lying above a given point $m \in M$ is characterized by a unique group element $\Lambda \in G$. As a result, $ \AD P$ is canonically isomorphic to $ M \times G$, and every section $\varsigma \in \Gamma(\AD P)$ induces a function $\xi : M \to G$ such that $ \varsigma(m) = [ p, \xi(m) ] $, where $p$ is an arbitrary point in $P_m $. In particular, the gauge transformation  $ \varphi_{\varsigma} : P \to P $ associated with $ \varsigma$ is given by $ \varphi_{\varsigma}(p) = p \cdot \xi(\pi(p))$, and~\eqref{eqGaugeStarSigma} becomes
\begin{displaymath}
    \varepsilon^{-1}_{\sigma(m)}( \varphi_*(e)) = \xi(m) \varepsilon^{-1}_{\sigma(m)}(e).
\end{displaymath}

We now consider the \emph{adjoint bundle} $ \Ad P \xrightarrow{\pi_{ {\Ad P}}} M$,  which is a distinct bundle from $ \AD P \xrightarrow{\pi_{{\AD P }}} M $, also playing an important role in gauge theory. In particular, $ \Ad P \xrightarrow{\pi_{ {\Ad P}}} M$ is defined as the vector bundle over $M$ associated to $ P \xrightarrow{\pi} M $ with typical fiber equal to the Lie algebra of $ G$, $F = \mathfrak{g} $, and left action given by the adjoint representation, $ L^\Lambda \lambda = \Ad_\Lambda \lambda  $. This bundle is useful for characterizing the space of connection 1-forms on $P$, denoted by $\mathcal C$. That is, it can be shown that $\mathcal C$ is an infinite-dimensional affine space with translation group isomorphic to $ \Omega^1(M, \Ad P)$; the space of 1-forms on the base space $M$ with values in $ \Ad P$. If the structure group $G$ is Abelian,  $\Ad_\Lambda = \Id $, and the equivalence classes $[p,\lambda ] \in \Ad P$ are characterized by unique Lie algebra elements $ \lambda \in \mathfrak g$. It then follows that $\Ad P $ is canonically isomorphic to $ M \times \mathfrak g$ (analogously to the isomorphism of $\AD P $ to $ M \times G$), and that sections in $\Gamma(\Ad P)$ are canonically identified with $\mathfrak g$-valued functions on $M$ (analogously to the identification of sections in $\Gamma(\AD P)$ with $G$-valued functions).   

\subsubsection{\label{appCurv}Curvature 2-forms and field strengths}

In addition to characterizing the space of connections, the adjoint bundle plays an important role in the context of curvature, as we now describe. First, as already stated in Section~\ref{secYM}, the \emph{curvature 2-form} associated with a connection 1-form $\omega \in \mathcal C$ is the Lie-algebra valued 2-form $\Omega \in \Omega^2(P,\mathfrak g)$ given by the covariant exterior derivative $\Omega = \mathcal D \omega$. It can be shown that $\Omega$ obeys the structure equation
\begin{equation}
    \label{eqCurvStruct}
    \Omega(Y,Z) = d\omega(Y, Z) + [ \omega Y, \omega Z],
\end{equation}
where $ [\cdot, \cdot] : \mathfrak g \times \mathfrak g \to \mathfrak g$ is the Lie algebra commutator. This leads to the \emph{Bianchi identity}, 
\begin{equation}
    \label{eqBianchiP}
    \mathcal D \Omega = \mathcal D^2 \omega =  0,
\end{equation}
which is a non-trivial result since, as noted in Appendix~\ref{appConn}, $\mathcal D^2$ is in general a nonzero operator. Equation~\eqref{eqCurvStruct} manifestly exhibits the fact that $\Omega $ depends nonlinearly on $\omega$ if $G$ is non-Abelian, whereas in the Abelian case that dependence is affine. By examining pullbacks $\sigma^*\Omega \in C^\infty(U,\mathfrak g)$ along trivializing sections $\sigma: U \subseteq M \to P$ of the principal bundle, $\Omega$ can be identified with a 2-form $ F^\omega \in \Omega^2( M, \Ad P ) $ taking values in the adjoint bundle. The 2-form $F^\omega$ is known as the \emph{gauge field strength} associated with $\omega$. 

Let now $ E \xrightarrow{\pi_E} M $ be a vector bundle with typical fiber $F$, acted upon by $G$ through a representation $\rho : G \to \GLF$. Let also $ \varrho: \mathfrak g \to \glf $ be the Lie algebra representation induced by $ \rho $ through its differential at the identity, $ \varrho = \rho_{*,I}$. Then, if $\nabla : \Gamma(E) \to \Omega^1(M,E)$ is the covariant derivative associated with $\omega$, the Bianchi identity in~\eqref{eqBianchiP} is equivalent to
\begin{displaymath}
    \nabla F^\omega = 0,
\end{displaymath}
and the corresponding curvature tensor $\mathsf R \in \Omega^2(M, \End E) $ from~\eqref{eqCurvR} satisfies    
\begin{displaymath}
    ( \mathsf R_{X,Y} s )_m = [ p, \varrho(\Omega(X,Y)) f ].
\end{displaymath}
Here, $ m $ is a point in $M$, $X,Y$ are vector fields in $\Gamma(TM)$, and $s$ is a section in $\Gamma(E)$ with $s(m) = [ p, f ] $. Specializing these results to the case of the adjoint bundle, $ E = \Ad P $, $\varrho $ becomes the adjoint representation of $ \mathfrak g $, and thus
\begin{displaymath}
    ( \mathsf R_{X,Y} s )_m = [ p, \ad_{\Omega(X,Y)} f ].
\end{displaymath}
Moreover, on the domain $U\subseteq M $ of any trivializing section $\sigma : M \to P $, we have (cf.~\eqref{eqCov})
\begin{displaymath}
    \nabla^\sigma_X f = df \cdot X + \ad_{\omega^\sigma X} f, \quad f \in C^\infty(U,\mathfrak g), \quad X \in \Gamma(TU).  
\end{displaymath}
From the last two equations, we deduce that if $G$ is Abelian (i.e., $\ad_\lambda = 0 $ for all $\lambda \in \mathfrak g$), then every connection 1-form $\omega \in \mathcal C$ induces a flat connection  $\nabla$ on the adjoint bundle, which coincides with the standard exterior derivative.

\subsubsection{\label{appYM}Outline of Yang-Mills theory}

In this section, we briefly outline the aspects of Yang-Mills theory leading to the constructions in Section~\ref{secYM}. As in Appendix~\ref{appCurv}, let $\omega \in \mathcal C$ be a connection 1-form on the principal bundle, and let $ \Ad P \xrightarrow{\pi_{\Ad P}} M $ be the adjoint bundle, equipped with the covariant derivative $\nabla$ induced from $\omega$. In Yang-Mills theory, it is assumed that the base space manifold $M$ is equipped with a Riemannian or pseudo-Riemannian metric $\eta$ with volume form $\nu$, and the Lie algebra $\mathfrak g $ is similarly equipped with a non-degenerate $\Ad$-invariant sesquilinear form $ b : \mathfrak g \times \mathfrak g \to \mathbb C $ (cf.\ $b$ from Section~\ref{secMinkowski}). Then, for every $m \in M$, one can define a sesquilinear form $g_m $ on $\Ad P_m $ analogously to~\eqref{eqG}, as well as pointwise sesquilinear forms for $k$-forms in $\Omega^k(M, \Ad P)$ and corresponding norms; the latter, denoted by $\lVert \cdot \rVert_m$.  The Hodge star operator $\star : \Omega^k(M) \to \Omega^{\dim M - k}(M)$ also lifts canonically to an operator $ \star : \Omega^k(M,\Ad P ) \to \Omega^{\dim M - k }(M,\Ad P)$. 

With these definitions, the \emph{Yang-Mills action} associated  with a connection $\omega$ is given by 
\begin{displaymath}
    S_{\text{YM}}(\omega) = \int_M \lVert F^\omega \rVert^2_m \, d\nu(m),
\end{displaymath}
whenever the integral exists. One can verify that $S_{\text{YM}}$ is gauge-invariant. 

If $M$ is compact, then $S_{\text{YM}}(\omega)$ exists for all connections in $\mathcal C$. In the non-compact case, including the Minkowski space studied here, the domain of definition of $S_\text{YM}$ is a proper subset of $\mathcal C$. The Yang-Mills condition states that $\omega$ should be a critical point of this functional; that is, for any section $ s \in \Gamma(\Ad P )$, 
\begin{displaymath}
    \left. \frac{d\ }{d\tau} S_{\text{YM}}(\omega + \tau s ) \right\rvert_{\tau =0} = 0.
\end{displaymath}
This condition leads to the \emph{Yang-Mills equation},
\begin{displaymath}
    \nabla \star F^\omega = 0,
\end{displaymath}
which represents a system of partial differential equations for local gauge fields $\omega^\sigma = \sigma^* \omega $ associated with local trivializing sections $\sigma : U \subseteq M \to P$ of the principal bundle. Due to the nonlinear dependence of $\Omega$ on $\omega$, these equations are nonlinear if $G$ is non-Abelian. Note that if $M$ is non-compact, a connection 1-form $\omega$ can satisfy the Yang-Mills equation while having infinite action. 

\subsection{\label{appBundleMinkowski} Fiber bundles over Minkowski space}

We now restrict attention to the setting where  $P\xrightarrow{\pi}M$ and $E \xrightarrow{\pi_E} M$ are the principal and associated bundles introduced in Section~\ref{secGauge} over two-dimensional Minkowski space, with structure group $ G \simeq \SO$. We begin with three lemmas establishing basic properties of Cartesian charts on $M$, and their induced charts on $P$. 

\begin{lem} \label{lemChartX}Let $x : M \to \mathbb{R}^2$ be an inertial chart with origin $o $ and basis vector fields $X_j$, and $y : P \to \mathbb{R}^3$ the induced chart on the principal bundle. Let also $ L^\Lambda_o $ be a Lorentz transformation associated with group element $\Lambda \in G$. Then, the following hold:  
    \begin{enumerate}[(i),wide]
        \item $ x ' = x \circ L^\Lambda_o$ is an inertial chart, whose coordinates and basis vector fields satisfy
            \begin{align*}
                x^{\prime i}(m) &= x^i(L^\Lambda_o(m)) = \sum_{j=0}^1 {\Lambda^{i}}_j x^j(m), \\ 
                X'_i &= \sum_{i=0}^1 X_j {\Lambda^{-1,j}}_i, 
            \end{align*}
            respectively. Here, ${\Lambda^{i}}_j$ are constant coefficients given by 
            \begin{displaymath}
                {\Lambda^{i}}_j = dx^i \cdot \Lambda X_j = \vec X^i \cdot \Lambda \vec X_j, 
            \end{displaymath}
            and the ${\Lambda^{-1,j}}_i$ are defined similarly with $\Lambda$ replaced by $\Lambda^{-1}$.
        \item The trivializing section $\sigma_{x'} : M \to P$ and maps $ \iota_{\sigma_{x'}} : M \times G \to P $, $ \gamma_{\sigma_{x'}} : P \to G$ associated with $x' $  have the equivariance properties
            \begin{displaymath}
                \sigma_{x'} = R^\Lambda \circ \sigma_x, \quad \iota_{\sigma_{x'}} = R^\Lambda \circ \iota_{\sigma_x}, \quad \gamma_{\sigma_{x'}} = L^{\Lambda^{-1}} \circ \gamma_{\sigma_x}.
            \end{displaymath}
        \item The coordinates and basis vector fields of the coordinate chart $y' : P \to \mathbb{R}^3$ induced by $x' $ satisfy
            \begin{gather*}
                y^{\prime i}(p) = \sum_{j=0}^1 {\Lambda^i}_j y^j(p), \quad Y'_i = \sum_{j=0}^1 Y_j {\Lambda^{-1,j}}_i, \quad i \in \{ 0, 1 \}, \\
                y^{\prime2}(p) = y^2(p) - \vartheta(\Lambda), \quad Y'_2 = Y_2,
            \end{gather*}
            where $\vartheta : G \to \mathbb{R}$ is the coordinate chart on the gauge group from Section~\ref{secMinkowski}. 
    \end{enumerate}
\end{lem}

\begin{proof}
    (i) Let $m \in M $ be arbitrary. Then, by definition of the $x$ chart, $ m = o + \vec v $, where $ \vec v = \sum_{j=0}^1 x^j(m) \vec X_j $, and 
    \begin{displaymath}
        L^\Lambda_o(m) = o + \sum_{j=0}^1 x^j(m) \Lambda \vec X_j = o + \sum_{j,k=0}^1  \vec X_k { \Lambda^k }_j x^j(m), 
    \end{displaymath}
    where ${\Lambda^k}_j = \vec X^k \cdot \Lambda \vec X_j$. Therefore, 
    \begin{displaymath}
        x^{\prime i}(m) = x^i(L^\Lambda_o(m)) = \sum_{j=0}^1 {\Lambda^{i}}_j x^j(m),
    \end{displaymath}
    which confirms the claimed relationship between the $x^{\prime } $ and $x$ coordinates. Moreover, 
    \begin{align*}
        \vec v &= \sum_{j=0}^1 x^{j}(m) \vec X_j = \sum_{i,j=0}^1 {\Lambda^{-1,j}}_i x^{\prime i}(m) \vec X_j \\
        &= \sum_{i=0}^1 x^{\prime i}(m) \left( \sum_{j=0}^1 X_j {\Lambda^{-1,j}}_i \right),   
    \end{align*}
    which implies that $X'_i = \sum_{j=0}^1 X_j {\Lambda^{-1,j}}_i$, as claimed. The equality of $dx^i\cdot \Lambda X_j$ and $ \hat X^i \cdot \Lambda \vec X_j$ follows by definition of inertial charts.

    (ii) The claim about $\sigma_{x'}$ follows directly from the fact that for every $m \in M$,
    \begin{align*}
        \sigma_{x'}(m) &= \{ X'_{0,m}, X'_{1,m} \} = \{ \Lambda^{-1} X_{0,m}, \Lambda^{-1} X_{1,m} \} \\
        &= \sigma_x(m) \cdot \Lambda = R^\Lambda( \sigma_x(m) ). 
    \end{align*}
Moreover, for every $p \in P$, we have $ \iota_{\sigma_{x'}}(p) = ( m, \gamma_{\sigma_{x'}}(p) ) $, where $ \gamma_{\sigma_{x'}}(p)$ is the unique element of $G $ satisfying $\sigma_{x'}(m) \cdot \gamma_{\sigma_{x'}}(p) = p$. But by the result just proved, $\sigma_{x'}(m) = \sigma_x(m) \cdot \Lambda$, and therefore $ \sigma_x(m) \cdot \Lambda \gamma_{\sigma_{x'}}(p) = p$. The latter, implies that $ \Lambda \cdot \gamma_{\sigma_{x'}}(p) = \gamma_{\sigma_x(p)}$, leading to the claim about $\gamma_{\sigma_{x'}}$. To prove the claim about $\iota_{\sigma_{x'}}$, observe that for every $m \in M $ and  $\tilde \Lambda \in G$, 
    \begin{align*}
        \iota_{\sigma_{x'}}(m, \tilde \Lambda) &= \sigma_{x'}(m) \cdot \tilde\Lambda = \sigma_{x}(m) \cdot \Lambda \tilde \Lambda = \sigma_{x}(m) \cdot \tilde \Lambda \Lambda \\
        &= \iota_{\sigma_{x}}(m, \tilde \Lambda ) \cdot \Lambda,
    \end{align*}
    and the claim follows.

    (iii) By definition of the $y'$ chart, for any $p \in P$ we have
    \begin{displaymath}
        y'(p) = ( x' \otimes \vartheta) \iota_{\sigma_{x'}}^{-1}(p) = ( x'(\pi(p)), \vartheta(\gamma_{\sigma_{x'}}(p))).
    \end{displaymath}
    Therefore, for $ i \in \{ 0, 1 \} $, Claim~(i) implies that
    \begin{displaymath}
        y^{\prime i}(p) = x^{\prime i}(\pi(p)) = \sum_{j=0}^1 {\Lambda^i}_j x^j(\pi(p)) = \sum_{j=0}^1 {\Lambda^i}_j y^j(p).
    \end{displaymath}
    Moreover, it follows from Claim~(ii) and~\eqref{eqThetaChart} that
    \begin{align*}
    y^{\prime 2}(p) &= \vartheta(\gamma_{\sigma_{x'}}(p)) = \vartheta(\Lambda^{-1} \gamma_{\sigma_x}(p))) \\
    &= \vartheta(\Lambda^{-1}) + \vartheta(\gamma_{\sigma_x}(p)) = - \vartheta(\Lambda) + y^2(p),
    \end{align*}
    which proves the claimed relation between $y'(p)$ and $y(p)$. Turning to the coordinate basis vector fields, for any $ f \in C^{\infty}(P)$ we have
    \begin{align*}
        Y'_{i,p} f &= \partial_j( f \circ y^{\prime,-1})_{y'(p)} = \partial_j(f \circ y^{-1} \circ y \circ y^{\prime,-1})_{y'(p)} \\
        &= \sum_{j=0}^2 \partial_j(f \circ y^{-1} )_{y(p)} \partial_i(y^j \circ y^{\prime,-1})_{y'(p)} \\
        &= \sum_{j=0}^2 ( Y_{j,p} f )   \partial_i(y^j \circ y^{\prime,-1})_{y'(p)},
    \end{align*}
    Thus, by the results just established,
    \begin{displaymath}
        \partial_i(y^j \circ y^{\prime,-1})_{y'(p)} = 
        \begin{cases}
            {\Lambda^{-1,j}}_i, & i,j \in \{ 0, 1 \}, \\
            1, & i = j = 2, \\ 
            0, & \text{otherwise}, 
        \end{cases}
    \end{displaymath}
    and we conclude that $ Y'_{i,p} f = \sum_{j=0}^1 Y_{j,p} f {\Lambda^{-1,j}}_i $ for $ i \in \{ 0, 1 \} $, and $ Y'_{2,p} f= Y_{2,p}$. This completes the proof of the claim and the lemma. 
\end{proof}

\begin{lem}
    \label{lemChartY}Let $y : P \to \mathbb{R}^3 $ be the coordinate chart on the principal bundle induced by an inertial chart $ x : M \to \mathbb{R}^2$, and let $Y_j$ and $X_j$ be the associated coordinate vector fields. Then, for any $ \Lambda \in G$, the $Y_j $ are invariant under the right action $R^\Lambda $ on the principal bundle, i.e., $ R^\Lambda_* Y_j = Y_j$. Moreover, $Y_2 $ is a fundamental vector field generated by the Lie algebra basis vector $ u $, i.e., $ Y_2 = W^u $.      
\end{lem}

\begin{proof}
    Let $ p = \{ p_0, p_1 \} \in P $ be a point on the principal bundle lying above $ \pi( p ) = m $, and $ p' = R^{\Lambda^{-1}}( p )$. We begin by observing that
    \begin{equation}
        \label{eqYChart}
        \begin{gathered}
            y^0(p') = y^0(p), \quad y^1(p') = y^1(p), \\
            y^2(p') = y^2(p) - \vartheta(\Lambda). 
        \end{gathered}
    \end{equation}
    Indeed, the first expression in the above follows immediately from the facts that $ \pi(p') = m $ and
    \begin{displaymath}
        y^{0}(p') = x^0(\pi(p')) = x^0(\pi(p)) = y^0(p).
    \end{displaymath} 
    The second expression follows similarly. Meanwhile, the third expression follows from the  $G$-equivariance of $\gamma_{\sigma_x}$, viz.
    \begin{align*}
        y^2(p') &= \vartheta(\gamma_{\sigma_x}(p')) = \vartheta(\gamma_{\sigma_x}(R^{\Lambda^{-1}}(p))) \\
        &= \vartheta(L^{\Lambda^{-1}}(\gamma_{\sigma_x}(p))) = \vartheta(\Lambda^{-1}) + \vartheta(\gamma_{\sigma_x}(p)) \\
        &= - \vartheta(\Lambda) + y^2(p),
    \end{align*}
    where the first equality in the second line follows from~\eqref{eqThetaChart}. 
    
    To prove that $ R^\Lambda_* Y_j = Y_j $, we must show that for any $f \in C^\infty(P)$, $( R^\Lambda_* Y_j )_p f $ is equal to $ Y_{j,p} f$. We have,
    \begin{align*}
        (R^\Lambda_* Y_j)_p f &= Y_{j,p'}( f \circ R^\Lambda) = \partial_j( f \circ R^\Lambda \circ y^{-1})_{y(p')} \\
        & = \sum_{k=0}^2 \partial_k(f \circ y^{-1})_{y(p)} \partial_j(y^k \circ R^\Lambda \circ y^{-1})_{y(p')} \\
        & = \sum_{k=0}^2 Y_{k,p} f \,  \partial_j(y^k \circ R^\Lambda \circ y^{-1})_{y(p')}, 
    \end{align*}
    and by~\eqref{eqYChart}, $\partial_j(y^k \circ R^\Lambda \circ y^{-1})_{y(p')} = {\delta^k}_j$, leading to  $( R^\Lambda_* Y_j )_p f = Y_{j,p} f$.

    Finally, to verify that $Y_2$ is the fundamental vector field generated by $u $, it is enough to show that for all $p \in P$, $ dy^j \cdot W^{u}_p = {\delta^j}_2$. Indeed, it follows from~\eqref{eqYChart} that for $j \in \{ 0, 1 \}$ and any $ \epsilon \in \mathbb{R}$, $y^j(R^{\exp(\epsilon u) p}) = y^j(p)$, so that
    \begin{displaymath}
        dy^j_p \cdot W^u_p = \lim_{\epsilon\to0} \frac{y^j(R^{\exp(\epsilon u)}(p)) - y^j(p) }{\epsilon} = 0, 
    \end{displaymath}
    while $y^2(R^{\exp(\epsilon u)} p) = y^2(p) + \epsilon$, so that 
    \begin{displaymath}
        dy^2_p \cdot W^u_p = \lim_{\epsilon\to 0} \frac{y^2(R^{\exp(\epsilon u)}(p)) - y^{2}(p)}{\epsilon} = 1,
    \end{displaymath}
    proving the claim, and completing the proof of the lemma. 
\end{proof}

\begin{lem}
    \label{lemCirc} The $^\odot$ operator from~\eqref{eqCirc} is (i) Lorentz-invariant, i.e., independent of the choice of inertial chart $x $ with origin $ o \in M$; and (ii) $G$-equivariant, in the sense that $ ( R^\Lambda_* Y)^\odot = R^\Lambda_* (Y^\odot) $ for any $\Lambda \in G$, $p \in P$, and $ Y \in T_p P $. 
\end{lem}

\begin{proof}
    (i) Let $x' : M \to \mathbb{R}^2$ be an inertial chart centered at $o $, and $ \Lambda $ the unique element of the structure group $G$ such that $ x' = x \circ L_o^\Lambda $. Let also $ y': P \to \mathbb{R}^3$ be the associated coordinate chart on the principal bundle, with coordinate vector fields $Y'_j$. For any $p \in P$, define the operator $^\boxdot : T_p P \to T_p P$ such that 
    \begin{displaymath}
        Y^{\prime\boxdot}_{0,p} = - Y'_{1,p}, \quad Y^{\prime\boxdot}_{1,p} = - Y'_{0,p}, \quad Y^{\prime\boxdot}_{2,p}= Y'_{2,p}. 
    \end{displaymath}
    We must show that for any $ Y\in T_pP $, $ Y^\boxdot = Y^\odot$. For that, observe first that $^\odot$ and $ ^\boxdot $ have the same matrix elements in their respective defining bases, i.e., 
    \begin{displaymath}
        {A^j}_i := dy^j_p \cdot Y^{\odot}_{i,p} = dy^{\prime j}_p \cdot Y^{\prime\boxdot}_{i,p}, 
    \end{displaymath}
    where
    \begin{displaymath}
        \bm A = [{A^j}_i] = 
        \begin{pmatrix}
            0 & -1 & 0 \\
            -1 & 0 & 0 \\
            0 & 0 & 1 
        \end{pmatrix}.
    \end{displaymath}
    It then follows from these facts and Lemma~\ref{lemChartX}(iii) that
    \begin{align*}
        Y^\boxdot &= \sum_{i=0}^2 ( dy^{\prime,i}_p \cdot Y) Y_{i,p}^{\boxdot} = \sum_{i,j=0}^2 ( dy^{\prime,i}_p \cdot Y) {A^j}_i Y'_{j,p} \\
        &= \sum_{i,j=0}^1 (dy^{\prime,i}_p \cdot Y) {A^{j}}_i Y'_{j,p} + ( dy^{\prime,2}_p \cdot Y ) Y'_{2,p}  \\
        &= \sum_{i,j,k,l=0}^1 ( dy^k_p \cdot Y) {\Lambda^{-1,l}}_j {A^{j}}_i {\Lambda^i}_k Y_{l,p} + ( dy^2_p \cdot Y ) Y_{2,p} \\
        &= \sum_{k,l=0}^1 ( dy^k_p \cdot Y) {A^l}_k Y_{l,p} + ( dy^2_p \cdot Y ) Y_{2,p}= Y^\odot,
    \end{align*}
    proving Claim~(i). 

    (ii) By the $ R^\Lambda_* $-invariance of the $Y_j $  established in Lemma~\ref{lemChartY} and the pointwise definition of $ ^\odot$ in~\eqref{eqCirc}, it follows that
    \begin{displaymath}
        (R^\Lambda_* Y_{j,p})^\odot = (Y_{j,R^\Lambda(p)})^\odot = Y^\odot_{j,R^\Lambda(p)}.
    \end{displaymath}
    Therefore, using again Lemma~\ref{lemChartY}, we obtain
    \begin{align*}
        (R^\Lambda_* Y)^\odot &= \sum_{j=0}^2 (dy^j_{R^\Lambda(p)} \cdot R^\Lambda_* Y)  Y_{j,R^\Lambda(p)} ^\odot  \\
        &= \sum_{j=0}^2 ( R^{\Lambda*}dy^j_{R^\Lambda(p)} \cdot  Y) ( R^\Lambda_* Y_{j,p})^{\odot}\\
        & = \sum_{j=0}^2 (dy^j_p \cdot Y ) ( R^{\Lambda}_*Y_{j,R^\Lambda(p)} )^{\odot} \\ 
        &= R^\Lambda_* ( Y^\odot),
    \end{align*}
    where the second-to-last line follows from the $R^{\Lambda *}$-invariance of the dual vector fields $ dy^j $ (which is in turn a direct consequence of the $R^\Lambda_*$-invariance of the $Y_j$). This proves the claim and the lemma.
\end{proof}

We now turn to the construction of the connection 1-form on the principal bundle, described in Section~\ref{secConn}.

\begin{prop}\label{propConn}The 1-form $\omega \in \Omega^1(P,\mathfrak g)$ defined in~\eqref{eqConn} satisfies~\eqref{eqConnCond}, i.e., it is a connection 1-form on the principal bundle $P\xrightarrow{\pi}M$. Moreover, $\omega$ does not depend on the choice of inertial chart $x: M \to \mathbb{R}^2$ with origin $o$.  
\end{prop}
\begin{proof} 
    First, note that in the induced coordinate chart $ y : P \to \mathbb{R}^3$ on the principal bundle, with  $ y(p)=(y^0(p),y^1(p),y^2(p)) = (x^0(\pi(p)),x^1(\pi(p)),\vartheta(\gamma_{\sigma_x}(p))) $, the function $\tilde h_x$ has the representation
    \begin{displaymath}
        \tilde h_x(p) = \frac{- ( y^0(p) )^2 + ( y^1(p))^2}{2} + y^2(p).
    \end{displaymath}
    As a result, $Y_0 \tilde h_x = -y^0 $, $ Y_1 \tilde h_x = y^1$, and $ Y_2  \tilde h_x =  1 $, leading, in conjunction with~\eqref{eqCirc} and~\eqref{eqConn}, to~\eqref{eqConnComponents}. Moreover, by~\eqref{eqYChart}, for any $ \Lambda \in G$ and  $p \in P$, we have 
    \begin{equation}
        \label{eqTildeHR}
        ( \tilde h_x \circ R^\Lambda )(p) = \tilde h_x(p) + \vartheta(\Lambda). 
    \end{equation}
That is, $\tilde h_x \circ R^\Lambda$ differs from $\tilde h_x$ by a constant on $P$, namely $\vartheta(\Lambda)$, so that $Y(\tilde h_x \circ R^\Lambda ) = Y \tilde h_x$ for any $Y \in T_pP$.

    Next, by~\eqref{eqFundamental}, every fundamental vector field $W^\lambda$ associated with $ \lambda \in \mathfrak{g}$ satisfies $ W^\lambda = W^{\vartheta(\lambda) u} = \vartheta(\lambda) W^u = \vartheta( \lambda ) Y_2 $. Thus, by~\eqref{eqConnComponents},
    \begin{displaymath}
        ( \omega W^\lambda )_p = \vartheta(\lambda)(\omega Y_2 )_p = \vartheta(\lambda) u = \lambda,
    \end{displaymath}
    proving the first condition in~\eqref{eqConnCond}. To verify the second condition, it must be shown that for every $p \in P$, $ Y \in T_pP $, and $ \Lambda \in G$, $(R^{\Lambda *} \omega)_p Y = \omega_p Y$. Indeed, by Lemma~\ref{lemChartY} and~\eqref{eqTildeHR}, we obtain
    \begin{align*}
        (R^{\Lambda *} \omega)_p Y &= \omega_{R^\Lambda(p)} R^{\Lambda}_{*p} Y = R^{\Lambda }_{*p}Y^\odot \tilde h_x = Y^\odot_p(\tilde h_x \circ R^\Lambda) \\
        &= Y^\odot_p  \tilde h_x = \omega_p Y,
    \end{align*}
    as claimed.

    Next, to verify that $\omega$ is independent of the choice of inertial chart $ x $ with origin $ o $, note that~\eqref{eqConn} can be equivalently expressed as 
    \begin{displaymath}
        (\omega Y)_p = ( \pi_{*,p} Y )^\perp h u + ( dy^2 \cdot Y )_p Y_{2,p} ( \vartheta \circ \gamma_{\sigma_x}).
    \end{displaymath}  
    So long as $ o $ is kept fixed, the first term in the right-hand side is independent of $x$ by the definition of $ h $ in~\eqref{eqPot} and Poincar\'e invariance of the $^\perp$ operator (see Section~\ref{secMinkowski}). It thus suffices to show that the second term is $x$-independent. For that, note that if $ x' : M \to \mathbb R^2 $ is the inertial chart with $ x' = x \circ L^\Lambda $ for some $ \Lambda \in G$, then by Lemma~\ref{lemChartX}, $\gamma_{\sigma_{x'}} = L^{\Lambda^{-1}} \circ \gamma_{\sigma_x}$ and $Y'_2 = Y_2 $, so that 
    \begin{multline*}
        ( dy^{\prime2} \cdot Y )_p Y'_{2,p} (\vartheta \circ\gamma_{\sigma_{x'}}) \\ 
        \begin{aligned}
            &= ( dy^{2} \cdot Y )_p Y_{2,p}( \vartheta \circ L^{\Lambda^{-1}} \circ \gamma_{\sigma_x}) \\ 
            &= ( dy^{2} \cdot Y )_p Y_{2,p}( \vartheta \circ  \gamma_{\sigma_x} \circ R^{\Lambda^{-1}}) \\ 
            &= ( dy^{2} \cdot Y )_p Y_{2,p}( \vartheta \circ  \gamma_{\sigma_x} ). 
        \end{aligned}
    \end{multline*}
    Note that to obtain the second-to-last line we used the equivariance property of $\gamma_{\sigma_x}$ in~\eqref{eqGammaEquiv}, and to obtain the last line we made use of the fact that $\vartheta \circ  \gamma_{\sigma_x} \circ R^{\Lambda^{-1}}$ differs from $ \vartheta \circ \gamma_{\sigma_x} $ by a constant shift of $-\vartheta(\Lambda)$ (in accordance with~\eqref{eqThetaChart}), which is annihilated upon application of the tangent vector $Y_{2,p}$. Since $ Y $ and $ p $ were arbitrary, this shows that $\omega$ is independent of the inertial chart $ x $ with origin $ o $, completing the proof of the proposition.
\end{proof}

As a final result in this appendix, we establish how points on the associated vector bundle $E \xrightarrow{\pi_E} M$ transform under gauge transformations.

\begin{lem}
    \label{lemPhase}Let $ \varphi : P \to P $ be a gauge transformation.  Then, for every point $ e \in E_m $ in the associated bundle lying above $ m \in M$, we have $ \varphi_* e = e^{i\alpha \vartheta(\xi(m))/\sqrt{2}}e $, where $ \vartheta : G \to \mathbb{R}$ is the coordinate chart from~\eqref{eqThetaChart}, and $ \xi : M \to G$ the unique function satisfying $ \varphi(p) = p \cdot \xi(\pi(p))  $. Moreover, if $ \sigma \in \Gamma(P)$ is a section of the principal bundle, the following hold:
    \begin{enumerate}[(i),wide]
        \item For any function $ f \in C^\infty(M)$, the corresponding sections $ f_\sigma = \zeta_\sigma f $ and $ f_{\varphi \circ \sigma} = \zeta_{\varphi \circ \sigma} f$ satisfy $ f_{\varphi \circ \sigma} = \varphi_* \circ f_\sigma$, and
            \begin{displaymath}
                f_{\varphi \circ \sigma}(m) = e^{i\alpha \vartheta(\xi(m))/\sqrt{2}} f_\sigma(m), \quad \forall m \in M.
            \end{displaymath}
        \item For any section $ s \in \Gamma(E)$ of the associated bundle, the corresponding $C^\infty(M)$ functions (matter fields) $ s^\sigma = \zeta_\sigma^{-1} s$ and  $ s^{\varphi \circ \sigma} = \zeta_{\varphi \circ \sigma}^{-1} s$ satisfy $ s^{\varphi \circ \sigma} = ( \varphi^\sigma_* )^{-1} \circ s^\sigma $, and 
            \begin{displaymath}
                s^{\varphi \circ \sigma}(m) = e^{-i\alpha\vartheta(\xi(m))/\sqrt{2}} s^{\sigma}(m), \quad \forall m \in M.   
            \end{displaymath}
    \end{enumerate}
\end{lem}

\begin{proof}
    First, the existence and uniqueness of $\xi $ was established in Appendix~\ref{appBundleGeneral}. Let $ e = [ p, z ] $, where $ p \in P_m $ and $ z \in \mathbb{C} $. To prove the claim about $ \varphi_* e$, we use~\eqref{eqGaugeStar} and the fact that $ G$ acts on the fiber $F = \mathbb{C}$ by multiplication according to~\eqref{eqLeftAction}, leading to
\begin{align*}
    \varphi_*(e)& = [ p, \tilde \xi(p) \cdot z ] = [ p, \xi(m) \cdot z ] = [ p, e^{i \alpha \vartheta(\xi(m)) / \sqrt{2}} z ] \\
    &=  e^{i \alpha \vartheta(\xi(m)) / \sqrt{2}} [ p, z ] = e^{i \alpha \vartheta(\xi(m)) / \sqrt{2}} e, 
\end{align*}
as claimed. 

Next, to verify Claim~(i), we have
\begin{align*}
    f_{\varphi\circ \sigma}(m) &= \zeta_{\varphi\circ\sigma} f(m) = \varepsilon_{(\varphi \circ \sigma)(m)}(f(m)) \\
    &= [ (\varphi \circ \sigma)(m), f(m)] = [ \sigma(m) \cdot \xi(m), f(m) ] \\
    &= [ \sigma(m), \xi(m) \cdot f(m)] = e^{i\alpha \vartheta(\xi(m))/\sqrt{2}} f_\sigma(m),
\end{align*}
and the claim follows using the result about $ \varphi_* $ just proved. To verify Claim~(ii), we employ the $G$-equivariance of $ \varepsilon_p^{-1}$ from Lemma~\ref{lemUnique} to obtain
\begin{align*}
    s^{ \varphi \circ \sigma}(m) &= \varepsilon^{-1}_{(\varphi \circ \sigma)(m)} (s(m)) = \varepsilon^{-1}_{\sigma(m) \cdot \xi(m)}(s(m)) \\
    &= \xi(m)^{-1} \cdot \varepsilon^{-1}_{\sigma(m)}s(m) = \xi(m)^{-1} \cdot s^{\sigma}(m) \\
    &= e^{-i\alpha\vartheta(\xi(m))/\sqrt{2}} s^\sigma(m).
\end{align*}
The above, in conjunction with Claim~(i), leads to
\begin{align*}
    s^{\varphi \circ \sigma}(m) &= e^{-i\alpha\vartheta(\xi(m))/\sqrt{2}} \zeta_{\sigma}^{-1} s(m) \\
    &= \zeta_{\sigma}^{-1}(e^{-i\alpha\vartheta(\xi(\cdot))/\sqrt{2}} s)(m) \\
    &= ( \zeta_{\sigma}^{-1} \circ \varphi^{-1}_* \circ s )(m) = ( \zeta_{\sigma}^{-1} \circ \varphi^{-1}_* \circ \zeta_\sigma) s^\sigma(m) \\
    &=  ( (\varphi_*^\sigma)^{-1} \circ s^\sigma )(m),
\end{align*}
and the claim follows.
\end{proof}

\section{\label{appRKHS}Results from RKHS theory}

\subsection{\label{appRKHSBasic}Basic properties of the heat kernel on the circle and the associated RKHSs}

We begin by stating some commonly used terminology in the theory of RKHSs on topological spaces. We will only consider compact spaces, as our main interest here is on RKHSs on the circle equipped with the standard-metric topology.    

Let then $S$ be a compact topological space, and $\kappa : S \times S \to \mathbb{C}$ a Hermitian function, i.e., $\kappa(\theta,\theta')=\kappa(\theta',\theta)^*$ for all $\theta,\theta' \in S$. The function $\kappa$ is said to be \emph{positive-definite} if for any finite sequence $ \theta_1, \ldots, \theta_n \in S $ the kernel matrix $\bm K = [ \kappa(\theta_i,\theta_j)]_{ij}$ is non-negative. By the Moore-Aronszajn theorem \cite{Aronszajn50}, every positive-definite Hermitian function is the reproducing kernel for a unique RKHS $\mathcal{K}$ of complex-valued functions on $S$; that is, a Hilbert space $(\mathcal{K}, \langle \cdot, \cdot \rangle_{\mathcal K})$ such that (i) the kernel sections $\kappa(\theta,\cdot)$ lie in $\mathcal K$ for all $\theta \in S$; and (ii) for every $\theta \in S$, the pointwise evaluation functional $\mathbb V_\theta : \mathcal{K} \to \mathbb C$, $\mathbb V_\theta f = f(\theta)$ is continuous, and satisfies $\mathbb V_\theta f = \langle \kappa(\theta,\cdot), f \rangle_{\mathcal K}$. The latter relation is known as the \emph{reproducing property}. In addition, the kernel $\kappa$ is said to be \cite{MicchelliEtAl06,SriperumbudurEtAl11}:

\begin{itemize}[wide]
    \item \emph{Strictly positive-definite} if $\bm K $ is a strictly positive matrix whenever the $ \theta_1, \ldots, \theta_n$ are all distinct; 
    \item \emph{$C$-universal} if $\kappa $ is continuous, and $\mathcal K$ is a dense subspace of the space of complex-valued, continuous functions on $S$, equipped with the uniform norm; 
    \item \emph{$L^p$-universal} if $\kappa$ is bounded and Borel-measurable, and  $\mathcal K$ is a dense subspace of $L^p(m)$ for any Borel probability measure $m$ on $S$, where $ p \in [ 1, \infty ]$ and $L^p(m)$ is equipped with the standard norm; 
    \item \emph{Characteristic} if $\kappa$ is bounded and Borel-measurable, and the map  $m \mapsto \int_S \kappa(\cdot, \theta) \, d\mu(\theta)$ is injective, where $m$ is any Borel probability measure on $S$. 
\end{itemize}

On a compact space, $C$-universality and $L^p$ universality of a continuous kernel are equivalent notions \cite{SriperumbudurEtAl11}. Moreover, every $C$-universal kernel is strictly-positive definite and characteristic, though the converses of these statements are not true. In addition, we have:   

\begin{lem}
    \label{lemFeature}
    Let $\kappa : S \times S \to \mathbb{C}$ be a Hermitian positive-definite kernel on a metric space $S$, $\mathcal K $ the corresponding RKHS, and $ F : S \to \mathcal{K} $ the feature map such that $F(\theta) = \kappa(\theta, \cdot)$. Then, the following hold.  
    \begin{enumerate}[(i), wide]
        \item $F$ is continuous if and only if $\kappa$ is continuous.  
        \item If $\kappa$ is strictly positive-definite, then $F$ is injective, and $F(\theta)$ and $F(\theta')$ are linearly independent whenever $\theta$ and $\theta'$ are distinct. 
    \end{enumerate}
\end{lem}

\begin{proof}
    (i) The equivalence between continuity of kernels and continuity of feature maps was proved in \cite[][Lemma~2.1]{FerreiraMenegatto13}.

    (ii) It is enough to show that $F(\theta)$ and $F(\theta')$ are linearly independent for any two distinct points $\theta$ and $\theta'$, for this property implies injectivity of $F$. To prove the claim by contradiction, suppose that there exist distinct points $ \theta, \theta' \in S^1 $ such that $F(\theta)$  and $F(\theta')$ are linearly dependent. Then, there exist coefficients $c, c' \in \mathbb{C}$, at least one of which is nonzero, such that $ c F(\theta) + c' F(\theta') = 0$, and thus
    \begin{align*}
        c \langle \kappa(\theta, \cdot), F(\theta) \rangle_{\mathcal K} + c' \langle \kappa( \theta, \cdot ), F(\theta') \rangle_{\mathcal K} &= 0, \\
        c \langle \kappa(\theta', \cdot ), F(\theta) \rangle_{\mathcal K} + c' \langle \kappa(\theta',\cdot), F(\theta') \rangle_{\mathcal K} &= 0. 
    \end{align*}
    Then, by the reproducing property, 
    \begin{displaymath}
        \langle \kappa(\theta, \cdot), F(\theta) \rangle_{\mathcal K} = \langle \kappa(\theta, \cdot), \kappa(\theta, \cdot ) \rangle_{\mathcal K} = \kappa(\theta,\theta'), 
    \end{displaymath}
    and similarly 
    \begin{displaymath}
        \langle \kappa(\theta, \cdot), F(\theta') \rangle_{\mathcal K} = \kappa(\theta,\theta'), \;  \langle \kappa(\theta',\cdot),F(\theta') \rangle_{\mathcal K} = \kappa(\theta',\theta'). 
    \end{displaymath}
    It then follows that the kernel matrix 
    \begin{displaymath}
        \bm K = \begin{pmatrix} \kappa(\theta,\theta) & \kappa(\theta,\theta') \\ \kappa(\theta',\theta) & \kappa(\theta',\theta') \end{pmatrix} 
    \end{displaymath}
    is a non-invertible kernel matrix associated with the distinct points $\theta$ and $\theta'$, contradicting the fact that $\kappa$ is strictly positive-definite.  
\end{proof}
 
With these definitions and results in place, we state the following lemma summarizing some of the basic properties of the heat kernel on the circle.  

\begin{lem}
    \label{lemUniversal}For any $\tau>0$, the heat kernel $\kappa_\tau : S^1 \times S^1 \to \mathbb{R}$ on the circle from~\eqref{eqHeatKernel} is a strictly positive-definite, $C$-universal, $L^p$-universal, characteristic kernel. Moreover, the feature map $F_\tau : S^1 \to \mathcal K_\tau$ is an injective, continuous map, mapping distinct points in $S^1$ to linearly independent functions in $\mathcal K_\tau$. 
\end{lem}

\begin{proof}
    By compactness of $S^1$ and continuity of $\kappa_\tau$, that $ \kappa_\tau $ is $L^p$-universal, strictly positive, and characteristic will follow if it can be shown that it is $C$-universal. Indeed, by \cite[][Corollary~11]{Steinwart01}, a translation-invariant kernel $\kappa: S^1 \times S^1 \to \mathbb{R}$ of the form $ \kappa(\theta,\theta') = \tilde \kappa( \lvert \theta - \theta' \rvert)$, where $ \tilde \kappa : [ 0, 2\pi ] \to \mathbb{R} $ is a continuous function, and $\lvert \cdot \rvert$ denotes arclength distance, is $C$-universal if and only if $ \tilde \kappa $ admits the pointwise absolutely convergent Fourier expansion $\tilde \kappa(\theta) = \sum_{j=0}^\infty \hat \kappa_j \cos( j \theta ) $, where the $ \hat \kappa_j  $ are strictly positive. Clearly, the heat kernel $\kappa_\tau$ from~\eqref{eqHeatKernel} takes this form for any $\tau > 0$, and it therefore follows that it is $ C$-universal. The claim about injectivity and linear independence of $F_\tau$ applied to distinct points then follows from Lemma~\ref{lemFeature}. 
\end{proof}

\begin{lem}
    \label{lemKernelNorm}
    For every $\tau > 0 $, the $L^2(\mu)$ norm of the kernel section $ \kappa_\tau(\theta, \cdot)$ is independent of $ \theta \in S^1$.  
\end{lem}

\begin{proof}
    By shift-invariance of $\kappa_\tau$ and unitarity of the Koopman group on $L^2(\mu)$, for every $ t \in \mathbb R$ we have
    \begin{align*}
        \lVert \kappa_\tau(\theta,\cdot) \rVert_{\mu} &= \lVert \kappa_\tau(\Phi^{\tau}(\theta), \Phi^t(\cdot)) \rVert_{\mu} = \lVert \kappa_{\tau}(\Phi^t(\theta), \cdot) \circ \Phi^t \rVert_\mu \\
        &= \lVert U^t \kappa_\tau(\Phi^t(\theta), \cdot ) \rVert_\mu = \lVert \kappa_\tau( \Phi^t(\theta),\cdot ) \rVert_\mu.
    \end{align*}
    The claim then follows from the fact that for every $ \theta,\theta' \in S^1$ there exists $ t \in \mathbb R$ such that $ \theta' = \Phi^t(\theta)$. 
\end{proof}
   
\subsection{\label{appS1Quantum}Proof of Proposition~\ref{propS1Quantum}}

First, to show that $\Psi_\tau = \Pi \circ F_\tau$ is injective, note that, by Lemma~\ref{lemUniversal}, $F_\tau$ maps distinct points in $S^1$ to linearly independent functions in $ \mathcal F(\mathcal K_\tau) \subset C(S^1) \setminus \{ 0 \} $. Moreover, by definition, $\Pi$ maps each (nonzero) function $ f $ its domain to the rank-1 orthogonal projection operator on $L^2(\mu)$ mapping onto $\spn\{ f \} $. Since $\spn\{f\}$ and $\spn\{f'\}$ are distinct subspaces of $L^2(\mu)$ whenever $f,f'$ are linearly independent continuous functions, these two facts together imply that  $\Psi_\tau$ is injective. 

To show that $\Psi_\tau$ is continuous, observe that 
\begin{displaymath}
    \Psi_{\tau}(\theta) = \langle F_\tau(\theta),\cdot\rangle_{\mu} F_\tau(\theta) / c, 
\end{displaymath}
where $c := \lVert \kappa( \theta, \cdot ) \rVert_\mu^2$ is independent of  $\theta$ by Lemma~\ref{lemKernelNorm}. Then, for any $\theta,\theta'$, we have
\begin{widetext}
\begin{align*}
    c\lVert \Psi_\tau(\theta) - \Psi_\tau(\theta') \rVert_1 &= \lVert \langle F_\tau(\theta),\cdot\rangle_{\mu} F_\tau(\theta) - \langle F_\tau(\theta'), \cdot \rangle_\mu F_\tau(\theta') \rVert_1 \\
    &= \lVert \langle F_\tau(\theta),\cdot \rangle_\mu F_\tau(\theta) - \langle F_\tau(\theta'),\cdot \rangle_\mu F_\tau(\theta) + \langle F_\tau(\theta'),\cdot\rangle_\mu F_\tau(\theta) - \langle F_\tau(\theta'),\cdot \rangle_\mu F_\tau(\theta') \rVert_1 \\ 
    &\leq \lVert \langle F_\tau(\theta) - F_\tau(\theta'), \cdot \rangle_\mu F_\tau(\theta) \rVert_1 + \lVert \langle F_\tau(\theta'), \cdot \rangle_\mu (F_\tau(\theta) - F_\tau(\theta')) \rVert_1 \\
    &=  \lVert F_\tau(\theta) - F_\tau(\theta') \rVert_\mu ( \lVert F_\tau(\theta) \rVert_\mu + \lVert F_\tau(\theta') \rVert_\mu) \\
    &\leq 2c^{1/2} \lVert F_\tau(\theta) - F_\tau(\theta') \rVert_{\mathcal K_\tau},
\end{align*}
and the continuity of $\Psi_\tau$ follows from continuity of $F_\tau $ (Lemma~\ref{lemUniversal}). Note that to obtain the equality in the second-to-last line above we used the fact that the trace norm of a rank-1 operator $ \langle f, \cdot \rangle_\mu g $ on $L^2(\mu) $ is equal to $ \lVert f \rVert_\mu \lVert g \rVert_\mu$, and the inequality in the last line follows from the fact that the $L^2(\mu)$ norm is bounded above by the $ \mathcal{K}_\tau$ norm for any $\tau > 0$. 
\end{widetext}

Next, to check that $ F_\tau \circ \Phi^t = U^{-t} \circ F_\tau$, we compute 
\begin{align*}
    F_\tau(\Phi^t(\theta)) &= \kappa_\tau(\Phi^t(\theta),\cdot) = \kappa_\tau(\theta,\Phi^{-t}(\cdot)) \\
    &= \kappa_\tau(\theta,\cdot) \circ \Phi^{-t} = U^{-t}\kappa_\tau(\theta,\cdot) = U^{-t} F_\tau(\theta),
\end{align*}
verifying the claim. In addition, for any $f \in C(S^1) \setminus \{ 0 \}$, we have 
\begin{align*}
    \tilde \Phi^t(\Pi(f)) &= U^{t*} \Pi(f) U^t = \frac{\langle f, U^t(\cdot)\rangle_\mu U^{t*} f}{\lVert f \rVert^2_{\mu}} \\
    &= \frac{\langle U^{-t}f,\cdot \rangle_\mu U^{-t} f}{ \lVert U^{-t} f \rVert_\mu^2} = \Pi(U^{-t}f),
\end{align*}
showing that $ \Pi \circ U^{-t} = \tilde \Phi^t \circ \Pi$. That $ \Psi_\tau \circ \Phi^t = \tilde \Phi^{t} \circ \Psi_\tau$ then follows from the fact that $\Psi_\tau = \Pi \circ F_\tau$. This completes the proof of Claim~(i).

Turning now to the map $\Omega_\tau$ in Claim~(ii), for any $A \in \mathcal A(L^2(\mu))$, 
\begin{displaymath}
    (\Omega_\tau A)^* = \tr((\Psi_\tau(\cdot))^* A^* ) = \tr(\Psi_\tau(\cdot)A) = \Omega_\tau A,
\end{displaymath}
where the second-to-last equality follows from the fact that for any $\theta \in S^1$, $\Psi_\tau(\theta)$ is a density operator, and thus self-adjoint. Moreover, for any $\theta,\theta' \in S^1$, we have
\begin{align*}
    \lvert \Omega_\tau A(\theta) - \Omega_\tau A(\theta') \rvert &= \lvert \tr( ( \Psi_{\tau}(\theta) - \Psi_\tau(\theta') ) A) \rvert \\
    & \leq  \tr \lvert (\Psi_\tau(\theta)-\Psi_\tau(\theta'))A\rvert \\
    &\leq \lVert \Psi_\tau(\theta) - \Psi_\tau(\theta') \rVert_1 \lVert A \rVert,
\end{align*}
and the continuity of $\Omega_\tau$ follows from the continuity of $\Psi_\tau$ established in Claim~(i). We therefore conclude that $\Omega_\tau$ is well-defined as an operator mapping into $C_{\mathbb R}(S^1)$. In addition, for any $t \in \mathbb{R}$,
\begin{align*}
    \Omega_\tau( \tilde U^t A ) &= \tr( \Psi_\tau(\cdot) U^t A U^{t*} ) = \tr( U^{t*} \Psi_\tau(\cdot) U^t A  ) \\
    &= \tr( \Phi^t_*( \Psi_\tau(\cdot)) A) = \tr(\Psi_\tau(\Phi^t(\cdot)) A) \\
    &= \tr(\Psi_\tau(\cdot) A ) \circ \Phi^t = U^t \tr( \Psi_\tau(\cdot) A) \\
    &= U^t (\Omega_\tau A),
\end{align*}
proving that $\Omega_\tau \circ \tilde U^t = U^t \circ \Omega_\tau$. This completes the proof of Claim~(ii). 

Finally, Claim~(iii) follows directly from the definition of $\Omega'_\tau$ and Claim~(ii).  \qed

\subsection{\label{appRKHS2}Proof of Theorem~\ref{thmRKHS2}}

Starting from Claims~(i) and~(ii), it follows from the Mercer representation of the kernel $\hat \kappa_\tau$ in~\eqref{eqFracHeatKernel} that $ \{ \hat \phi_{j,\tau} \}_{j\in\mathbb Z} $ with $ \hat \phi_{j,\tau} = e^{-\lvert j \rvert \tau / 2} \phi_j$ is an orthonormal basis of $\hat{\mathcal K}_\tau$. Moreover, for every $ t \in \mathbb R$, $ \hat \phi_{j,\tau} \circ \Phi^t = e^{i\alpha t} \hat \phi_{j,\tau}$, so that $U^t$ maps an orthonormal basis of $ \hat{\mathcal K}_\tau$ to an orthonormal basis. We therefore conclude that $ \hat{\mathcal K}_\tau$ and, since $\tau$ is arbitrary, $\hat{\mathcal K}^\infty$ are invariant under $U^t$, proving Claim~(i). Similarly, $U^t : \hat{\mathcal K}_\tau \to \hat{\mathcal K}_\tau $ is unitary since  $ \{ U^t\hat \phi_{j,\tau} \}_{j \in \mathbb Z} $ is an orthonormal basis, and the strong continuity of $ \{ U^t \}_{t\in \mathbb R} $ follows from the fact that $ t \mapsto U^t \hat \phi_{j,\tau} = e^{i\alpha j t} \phi_{j,\tau}$ is a continuous map.    

Next, turning to Claim~(iii), the basis elements $ \hat\phi_{j,\tau} $ are clearly eigenfunctions of the generator $V : D(V) \to \hat{\mathcal K}_\tau$, i.e., 
\begin{align*}
    V \hat \phi_{j,\tau} &= \lim_{t\to0} \frac{ U^t \hat \phi_{j,\tau} - \phi_{j,\tau} }{ t } \\
    &= \lim_{t\to0} \frac{e^{i\alpha j t} - 1 }{t} \hat \phi_{j,\tau} = i \alpha j \hat \phi_{j,\tau},   
\end{align*}
where the first limit is taken with respect to $\hat{\mathcal K}_\tau$ norm. Then, for every $ f = \sum_{j=-\infty}^\infty c_j \hat \phi_{j,\tau} \in \hat{\mathcal K}^1_\tau$ and $ l \in \mathbb N_0$, the action of $ V $ on the partial sum   $ f_l = \sum_{j=-l}^l c_j \hat \phi_{j,\tau} $ is given by $ g_l := V f_l = \sum_{j=-l}^l i \alpha j c_j \hat \phi_{j,\tau}$, and it follows from the definition of $\hat{\mathcal K}^1_\tau$ that $g_l $ is a Cauchy sequence in $\hat{\mathcal K}_\tau$. Thus, $\hat{\mathcal K}_\tau^1$ is a subspace of $D(V)$. Now if $ f =  \sum_{j=-\infty}^\infty c_j \hat \phi_{j,\tau} \in \hat{\mathcal K}_\tau \setminus \hat{\mathcal K}^1_\tau$, for every $C \geq 0 $, there exists $ l_0 \in \mathbb N_0 $ such that, for all $ l \geq l_0 $, $ \sum_{j=-l}^l \lvert j \rvert^2  \lvert c_j \rvert^2 \geq C^2 $. It therefore follows that $ g_l = V f_l $ has norm $ \lVert g_l \rVert_{\hat{\mathcal K}_\tau} > \alpha C$, and thus the sequence $ g_l $ is unbounded. This shows that $( \hat{\mathcal K}_\tau^{1} )^c \subseteq D(V)^c$, and therefore  $\hat{\mathcal K}^1_\tau \supseteq D(V)$. The latter, together with the fact that $ \hat{\mathcal K}^1_\tau \subseteq D(V) $ just shown, implies that $\hat{\mathcal K}^1_\tau =D(V)$, as claimed. 

What remains is to show that $V$ acts on $\hat{\mathcal K}^\infty$ as a derivation. For that, it needs to be shown that (i) $\hat{\mathcal K}^\infty $ is invariant under $V$; and (ii) the Leibniz rule holds, i.e. 
\begin{equation}
    \label{eqLeibnizRKHA}
    V(fg) = (Vf) g + f ( V g ), \quad \forall f,g \in \hat{\mathcal K}^\infty. 
\end{equation}
These claims will follow in turn from the following useful lemma:

\begin{lem}
    \label{lemVBounded}
    For every $ \tau > 0 $ and $ \tau' > \tau$, $\hat{\mathcal K}_{\tau'} $ is a subspace of $\hat{\mathcal K}^1_\tau$. Moreover, the restriction of the generator $ V : D(V) \to \hat{\mathcal K}_\tau$ to $ \hat{\mathcal K}_\tau$, viewed as an operator $ V\rvert_{\hat{\mathcal K}_{\tau'}} : \hat{\mathcal K}_{\tau'} \to \hat{\mathcal K}_\tau $, is bounded.
\end{lem}

A proof of Lemma~\ref{lemVBounded} will be given below. Assuming, for now, that it holds, let $ f = \sum_{j=-\infty}^\infty \hat f_j \phi_j$ and $ g = \sum_{j=-\infty}^\infty \hat g_j \phi_j $ be arbitrary elements of $\hat{\mathcal K}^\infty$, and note that the partial sums $ f_l = \sum_{j=-l}^l \hat f_j \phi_j$ and $ g_l = \sum_{j=-l}^l \hat g_j \phi_j $ converge in any $\hat{\mathcal K}_{\tau'} $ norm. Moreover, because $\hat{\mathcal K}_{\tau'}$ is a Banach algebra, as $l\to\infty$, $f_l g_l$ converges in $\hat{\mathcal K}_{\tau'}$ norm to $ fg$. Choosing, in particular, $ \tau' > \tau$, it follows from Lemma~\ref{lemVBounded} that $Vf_l $, $ Vg_l $, and $ V ( f_l g_l )$ are Cauchy sequences in $\hat{\mathcal K}_\tau$ converging to $Vf$, $Vg$, and $V(fg)$, respectively. The Banach algebra property of $\hat{\mathcal K}_\tau$ then implies  that $( V f_l ) g_l$ and $ f_l (Vg_l)$ are also $\hat{\mathcal K}_\tau$ Cauchy sequences, converging to $ (Vf) g $ and $ f (Vg)$, respectively. 

Now because $\tau$ was arbitrary, the convergence of $ Vf_l$ to $ Vf $ in $\hat{\mathcal K}_\tau$ norm implies that $Vf$ lies in $\hat{\mathcal K}_\tau$ for any $ \tau > 0 $, and therefore that $ \hat{\mathcal K}^\infty$ is invariant under $V$. Moreover, for any  $l\in \mathbb N_0 $, 
\begin{align*}
    V(f_l g_l) &= V\left( \sum_{j,k=-l}^l \hat f_j \hat g_k \phi_j \phi_k \right) = V\left( \sum_{j,k=-l}^l \hat f_j \hat g_k \phi_{j+k} \right) \\
    &= \sum_{j,k=-l}^l i \alpha (j+k) \hat f_j \hat g_k \phi_{j+k} \\
    & = \sum_{j,k=-l}^l i \alpha j \hat f_j \hat g_k \phi_j \phi_k  + \sum_{j,k=-l}^l i \alpha k \hat f_j \hat g_k \phi_j \phi_k  \\
    &= (V f_l) g_l + f_l (V g_l), 
\end{align*}
and taking $ l \to \infty $ limits in $\hat{\mathcal K}_\tau$ norm, we obtain
\begin{align*}
    V(fl) &= V\left( \lim_{l\to\infty} f_l g_l \right) = \lim_{l\to\infty} V(f_l g_l) \\
    &= \lim_{l\to\infty} (V f_l) g_l + \lim_{l\to\infty} f_l (V g_l) \\
    &= (Vf) g + f(Vg),
\end{align*}
verifying~\eqref{eqLeibnizRKHA}. 

Finally, to prove Lemma~\ref{lemVBounded}, observe that $ \hat{\mathcal K}_{\tau'} \subset \hat{\mathcal K}_\tau$  whenever $ \tau' > \tau $, and every $f \in \hat{\mathcal K}_{\tau'}$ admits an expansion  $ f = \sum_{j= -\infty }^\infty c_j \hat \phi_{j,\tau} $, where $\sum_{j=-\infty}^\infty e^{(\tau'-\tau)\lvert j \rvert} \lvert c_j \rvert^2 < \infty$. Since $e^{(\tau'-\tau) \lvert j \rvert} > C \lvert j \rvert^2 $ for some constant $ C > 0$ whenever $ \tau' > \tau$, it follows that $\sum_{j=-\infty}^\infty \lvert j \rvert^2 \lvert c_j \rvert^2 < \infty$, which shows that $ \hat{ \mathcal K }_{\tau'} \subseteq \hat{\mathcal K}_\tau^1$. Similarly, defining $ f_l = \sum_{j=-l}^l c_j \hat \phi_{j,\tau}$ with $c_j $ as above, we have 
\begin{align*}
    \lVert V f_l \rVert_{\hat{\mathcal K}_\tau}^2 &= \sum_{j=-l }^l \alpha^2 \lvert j \rvert^2 \lvert c_j \rvert^2 \\
    &\leq \sum_{j=-l}^l \alpha^2 e^{(\tau'-\tau)\lvert j \rvert} \lvert c_j \rvert^2 / C = \lVert f_l \rVert_{\hat{\mathcal K}_{\tau'}}^2 \alpha^2/C. 
\end{align*}
This shows that $ V\rvert_{\hat{\mathcal K}_{\tau'}} : \hat{\mathcal K}_{\tau'} \to \hat{\mathcal K}_\tau$ is uniformly bounded on the dense subspace of $ \hat{\mathcal K}_{\tau'}$ consisting of finite sums of the form $ \sum_{j=-l}^l c_j \hat \phi_{j,\tau}$, and is therefore bounded on the whole of $ \hat{\mathcal K}_{\tau'}$. \qed

\section{\label{appFractional}Fractional derivatives}

In this appendix, we provide definitions, and outline the basic properties of fractional derivative operators on functions on the circle. Additional details on these topics can be found, e.g., in Refs.~\cite{ButzerWestphal74,SamkoEtAl93}. 

To construct the fractional derivative operator $ \partial^r $ of order $ r\geq 0$, it is convenient to pass to a Fourier representation using the Fourier analysis operator $\mathfrak{F} : L^2(\mu) \to \ell^2$. The latter, is the unitary operator mapping $ f = \sum_{j=-\infty}^\infty c_j \phi_j \in L^2(\mu)$ to the $\ell^2$ sequence $ \hat f = \mathfrak{F} f := (c_j)_{j\in\mathbb{Z}} $ consisting of the expansion coefficients of $ f $ in the Koopman eigenfunction basis $ \{ \phi_j \}_{j=-\infty}^\infty$. Note, in particular, that 
\begin{displaymath}
    c_j = \langle \phi_j, f \rangle_\mu = \frac{1}{2\pi} \int_0^{2\pi} e^{-ij\theta} f(\theta) \, d\theta,
\end{displaymath}
indicating that the elements of $\hat f$ coincide with the standard Fourier expansion coefficients of $ f $.  Then, for any $ r \geq 0 $, we define
\begin{equation}
    \label{eqFracDer}
    \partial^r : D(\partial^r) \to L^2(\mu), \quad \partial^r = \mathfrak{F}^* \hat \partial_r \mathfrak{F}, 
\end{equation}
where $ \hat \partial^r : D(\hat \partial^r) \to \ell^2$ is the multiplication operator with dense domain 
\begin{displaymath}
    D(\hat \partial^r ) = \left\{ (c_j)_{j\in \mathbb{Z}} \in \ell^2 : \sum_{j=-\infty}^\infty j^{2r} \lvert c_j \rvert^2 < \infty \right\} \subseteq \ell^2,
\end{displaymath}
defined as
\begin{displaymath}
    \hat \partial^r (c_j)_{j\in \mathbb{Z}} = ( (i j)^r c_j)_{j\in \mathbb{Z}}.
\end{displaymath}
The domain of $\partial^r $ is then the dense subspace $D(\partial^r) $ of $ L^2(\mu)$ given by
\begin{displaymath}
    D(\partial^r) = \left \{ \sum_{j=-\infty}^\infty c_j \phi_j \in L^2(\mu) : \sum_{j=-\infty}^\infty j^{2r} \lvert c_j \rvert^2 < \infty \right \}.
\end{displaymath}

It follows directly from these definitions that if $ r $ is a non-negative integer,  $\partial^r$ from~\eqref{eqFracDer} coincides with the weak derivative operator with respect to standard angle coordinates $\theta$ on the circle. It can also be verified that the restriction $ \partial^r : W^{p,2}(\mu) \to L^2(\mu)$ of $ \partial^r $ to the Sobolev space $ W^{p,2}(\mu) \subseteq L^2(\mu) $ of order $p = \lceil r \rceil$, where
\begin{displaymath}
    W^{p,2}(\mu) = \left \{ \sum_{j=-\infty}^\infty c_j \phi_j \in L^2(\mu) : \sum_{j=-\infty}^\infty j^{2p} \lvert c_j \rvert^2 < \infty \right \},
\end{displaymath}
is a bounded operator. Here, $W^{p,2}(\mu)$ is equipped with the standard norm $\lVert f \rVert^{2}_{W^{p,2}(\mu)} = \sum_{q=0}^p \sum_{j=-\infty}^\infty j^{2q} \lvert c_j \rvert^2 $, where $  f = \sum_{j=-\infty}^\infty c_j \phi_j $.  

To gain intuition on the behavior of $\partial^r$ in a pointwise (as opposed to Fourier) representation, we state an equivalent definition to~\eqref{eqFracDer}, given by the $L^2(\mu)$ limit
\begin{equation}
    \label{eqFracDerAlt}
    \partial^r f = \lim_{a\to 0^+} \frac{1}{a^r} \mathbb{D}^r_a f, \quad f \in D(\partial^r),
\end{equation}
where $ \mathbb{D}^r_a f$ is the Riemann difference of $ f $, 
\begin{gather*}
    \mathbb{D}^r_a f( \theta ) = \sum_{n=0}^\infty (-1)^n \binom{r}{n} f( \theta - a n ), \\
    \binom{r}{n} = \frac{r (r -1) \cdots (r-n + 1)}{n!}.
\end{gather*}
In the fractional calculus, the operator in~\eqref{eqFracDerAlt} is known as the Liouville-Gr\"unwald, or Gr\"unwald-Letnikov derivative. From this definition, it is clear that whenever $ r $ is a non-negative integer $ \partial^r f(\theta) $ depends on the behavior of $f$ in an infinitesimally small neighborhood of $\theta$, for the Riemann difference $ \mathbb{D}^r_a$ contains finitely many terms. On the other hand, for non-integer $ r$, $\mathbb{D}^r_a$ contains infinitely many terms, and $ \partial^r f(\theta)$ depends on the behavior of $f$ on distant points from $\theta$. In other words, $\partial^r$ is a non-local operator whenever $r$ is not an integer.

With these definitions in place, we have the following:
\begin{prop}
    \label{propLadder}The ladder operators $A_\pm$ and $A_{\pm}^+$ from~\eqref{eqLadder} satisfy~\eqref{eqLadderFrac}.
\end{prop}

\begin{proof}
    We will prove the claim only for $A_-$ and $A_-^+$, as the results for $A_+$ and $A_+^+$ follow similarly. For that, note first that the set 
    \begin{displaymath}
        \left\{ \ldots, \frac{\phi_{-2}}{\sqrt{2}}, \frac{\phi_{-1}}{\sqrt{1}}, \phi_0, \frac{\phi_1}{\sqrt{1}}, \frac{\phi_2}{\sqrt{2}}, \ldots \right\}   
    \end{displaymath}
     is an orthonormal basis of the Sobolev space $W^{1,2}(\mu)$. As a result, because $ \partial^{1/2} : W^{1,2} \to L^2(\mu)$ is bounded, it is enough to show that
    \begin{equation}
        \label{eqPropFrac}
        \begin{aligned}
            \mathcal{U}^* A_0 \mathcal{U} \phi_j & = i^{1/2} L^* \partial_-^{1/2} \phi_j, \\
            \mathcal{U}^* A_0^+ \mathcal{U} \phi_j &= i^{1/2} \partial_-^{1/2} L \phi_j, 
        \end{aligned}
    \end{equation}
    for all $ j \in \mathbb{Z}$. Indeed, using~\eqref{eqLadderPsi}, for any non-positive integer $ j $ we obtain   
    \begin{align*}
        \mathcal{U}^* A_0 \mathcal{U} \phi_j &= \mathcal{U}^* A_0 \psi_{-j,0} =\mathcal{U}^*( \sqrt{-j} \psi_{-j-1,0}) \\
        &= \sqrt{-j} \phi_{j+1} = i^{1/2}\phi_1 (  \sqrt{ i j} \phi_j )  \\
        &=i^{1/2} L^* \partial^{1/2} \phi_j = i^{1/2} L^* \partial_-^{1/2} \phi_j  
    \end{align*}
    and
    \begin{align*}
        \mathcal{U}^* A_0^+ \mathcal{U} \phi_j &= \mathcal{U}^* A_0^+ \psi_{-j,0} = \mathcal{U}^* ( \sqrt{-j+1} \psi_{-j+1,0}) \\
        &= \sqrt{-j+1}\phi_{j-1} =  i^{1/2} \sqrt{i(j-1)}\phi_{-1} \phi_j \\
        &= i^{1/2} \partial^{1/2} L \phi_j = i^{1/2} \partial^{1/2}_- L \phi_j. 
    \end{align*}
    On the other hand, if $ j $ is positive,
    \begin{align*}
        \mathcal{U}^* A_0 \mathcal{U} \phi_j &= \mathcal{U}^* A_0 \psi_{0j} = 0 = i^{1/2} L^* \partial_-^{1/2} \phi_j \\
        \mathcal{U}^* A_0^+ \mathcal{U} \phi_j &= \mathcal{U}^* A_0^+ \psi_{0j} = \mathcal{U}^*\psi_{1j} = 0 = i^{1/2} L^* \partial_-^{1/2} \phi_j. 
    \end{align*}
    We therefore conclude that~\eqref{eqPropFrac} holds, as claimed.
\end{proof}

%

\begin{thebibliography}{34}%
\makeatletter
\providecommand \@ifxundefined [1]{%
 \@ifx{#1\undefined}
}%
\providecommand \@ifnum [1]{%
 \ifnum #1\expandafter \@firstoftwo
 \else \expandafter \@secondoftwo
 \fi
}%
\providecommand \@ifx [1]{%
 \ifx #1\expandafter \@firstoftwo
 \else \expandafter \@secondoftwo
 \fi
}%
\providecommand \natexlab [1]{#1}%
\providecommand \enquote  [1]{``#1''}%
\providecommand \bibnamefont  [1]{#1}%
\providecommand \bibfnamefont [1]{#1}%
\providecommand \citenamefont [1]{#1}%
\providecommand \href@noop [0]{\@secondoftwo}%
\providecommand \href [0]{\begingroup \@sanitize@url \@href}%
\providecommand \@href[1]{\@@startlink{#1}\@@href}%
\providecommand \@@href[1]{\endgroup#1\@@endlink}%
\providecommand \@sanitize@url [0]{\catcode `\\12\catcode `\$12\catcode
  `\&12\catcode `\#12\catcode `\^12\catcode `\_12\catcode `\%12\relax}%
\providecommand \@@startlink[1]{}%
\providecommand \@@endlink[0]{}%
\providecommand \url  [0]{\begingroup\@sanitize@url \@url }%
\providecommand \@url [1]{\endgroup\@href {#1}{\urlprefix }}%
\providecommand \urlprefix  [0]{URL }%
\providecommand \Eprint [0]{\href }%
\providecommand \doibase [0]{http://dx.doi.org/}%
\providecommand \selectlanguage [0]{\@gobble}%
\providecommand \bibinfo  [0]{\@secondoftwo}%
\providecommand \bibfield  [0]{\@secondoftwo}%
\providecommand \translation [1]{[#1]}%
\providecommand \BibitemOpen [0]{}%
\providecommand \bibitemStop [0]{}%
\providecommand \bibitemNoStop [0]{.\EOS\space}%
\providecommand \EOS [0]{\spacefactor3000\relax}%
\providecommand \BibitemShut  [1]{\csname bibitem#1\endcsname}%
\let\auto@bib@innerbib\@empty
\bibitem [{\citenamefont {Baladi}(2000)}]{Baladi00}%
  \BibitemOpen
  \bibfield  {author} {\bibinfo {author} {\bibfnamefont {V.}~\bibnamefont
  {Baladi}},\ }\href@noop {} {\emph {\bibinfo {title} {Positive transfer
  operators and decay of correlations}}},\ \bibinfo {series} {Advanced Series
  in Nonlinear Dynamics}, Vol.~\bibinfo {volume} {16}\ (\bibinfo  {publisher}
  {World scientific},\ \bibinfo {address} {Singapore},\ \bibinfo {year}
  {2000})\BibitemShut {NoStop}%
\bibitem [{\citenamefont {Eisner}\ \emph {et~al.}(2015)\citenamefont {Eisner},
  \citenamefont {Farkas}, \citenamefont {Haase},\ and\ \citenamefont
  {Nagel}}]{EisnerEtAl15}%
  \BibitemOpen
  \bibfield  {author} {\bibinfo {author} {\bibfnamefont {T.}~\bibnamefont
  {Eisner}}, \bibinfo {author} {\bibfnamefont {B.}~\bibnamefont {Farkas}},
  \bibinfo {author} {\bibfnamefont {M.}~\bibnamefont {Haase}}, \ and\ \bibinfo
  {author} {\bibfnamefont {R.}~\bibnamefont {Nagel}},\ }\href@noop {} {\emph
  {\bibinfo {title} {Operator Theoretic Aspects of Ergodic Theory}}},\ \bibinfo
  {series} {Graduate Texts in Mathematics}, Vol.\ \bibinfo {volume} {272}\
  (\bibinfo  {publisher} {Springer},\ \bibinfo {year} {2015})\BibitemShut
  {NoStop}%
\bibitem [{\citenamefont {Koopman}(1931)}]{Koopman31}%
  \BibitemOpen
  \bibfield  {author} {\bibinfo {author} {\bibfnamefont {B.~O.}\ \bibnamefont
  {Koopman}},\ }\href {\doibase 10.1073/pnas.17.5.315} {\bibfield  {journal}
  {\bibinfo  {journal} {Proc. Natl. Acad. Sci.}\ }\textbf {\bibinfo {volume}
  {17}},\ \bibinfo {pages} {315} (\bibinfo {year} {1931})}\BibitemShut
  {NoStop}%
\bibitem [{\citenamefont {Koopman}\ and\ \citenamefont {von
  Neumann}(1931)}]{KoopmanVonNeumann32}%
  \BibitemOpen
  \bibfield  {author} {\bibinfo {author} {\bibfnamefont {B.~O.}\ \bibnamefont
  {Koopman}}\ and\ \bibinfo {author} {\bibfnamefont {J.}~\bibnamefont {von
  Neumann}},\ }\href {\doibase 10.1073/pnas.18.3.255} {\bibfield  {journal}
  {\bibinfo  {journal} {Proc. Natl. Acad. Sci.}\ }\textbf {\bibinfo {volume}
  {18}},\ \bibinfo {pages} {255} (\bibinfo {year} {1931})}\BibitemShut
  {NoStop}%
\bibitem [{\citenamefont {Mauro}(2002)}]{Mauro02}%
  \BibitemOpen
  \bibfield  {author} {\bibinfo {author} {\bibfnamefont {D.}~\bibnamefont
  {Mauro}},\ }\href {\doibase 10.1142/S0217751X02009680} {\bibfield  {journal}
  {\bibinfo  {journal} {Int. J. Mod. Phys. A}\ }\textbf {\bibinfo {volume}
  {17}},\ \bibinfo {pages} {1301} (\bibinfo {year} {2002})}\BibitemShut
  {NoStop}%
\bibitem [{\citenamefont {Bondar}\ \emph {et~al.}(2012)\citenamefont {Bondar},
  \citenamefont {Cabera}, \citenamefont {Lompay}, \citenamefont {Ivanov},\ and\
  \citenamefont {Rabitz}}]{BondarEtAl12}%
  \BibitemOpen
  \bibfield  {author} {\bibinfo {author} {\bibfnamefont {D.~I.}\ \bibnamefont
  {Bondar}}, \bibinfo {author} {\bibfnamefont {R.}~\bibnamefont {Cabera}},
  \bibinfo {author} {\bibfnamefont {R.~R.}\ \bibnamefont {Lompay}}, \bibinfo
  {author} {\bibfnamefont {M.~Y.}\ \bibnamefont {Ivanov}}, \ and\ \bibinfo
  {author} {\bibfnamefont {H.~A.}\ \bibnamefont {Rabitz}},\ }\href {\doibase
  10.1103/PhysRevLett.109.190403} {\bibfield  {journal} {\bibinfo  {journal}
  {Phys. Rev. Lett.}\ }\textbf {\bibinfo {volume} {109}},\ \bibinfo {pages}
  {190403} (\bibinfo {year} {2012})}\BibitemShut {NoStop}%
\bibitem [{\citenamefont {Klein}(2018)}]{Klein18}%
  \BibitemOpen
  \bibfield  {author} {\bibinfo {author} {\bibfnamefont {U.}~\bibnamefont
  {Klein}},\ }\href {\doibase 10.1007/s40509-017-0113-2} {\bibfield  {journal}
  {\bibinfo  {journal} {Quantum Stud.: Math. Found.}\ }\textbf {\bibinfo
  {volume} {5}},\ \bibinfo {pages} {219} (\bibinfo {year} {2018})}\BibitemShut
  {NoStop}%
\bibitem [{\citenamefont {Bondar}\ \emph {et~al.}(2019)\citenamefont {Bondar},
  \citenamefont {Gay-Balmaz},\ and\ \citenamefont {Tronci}}]{BondarEtAl19}%
  \BibitemOpen
  \bibfield  {author} {\bibinfo {author} {\bibfnamefont {D.~I.}\ \bibnamefont
  {Bondar}}, \bibinfo {author} {\bibfnamefont {F.}~\bibnamefont {Gay-Balmaz}},
  \ and\ \bibinfo {author} {\bibfnamefont {C.}~\bibnamefont {Tronci}},\ }\href
  {\doibase 10.1098/rspa.2018.0879} {\bibfield  {journal} {\bibinfo  {journal}
  {Proc. Roy. Soc. A}\ }\textbf {\bibinfo {volume} {475}},\ \bibinfo {pages}
  {20180879} (\bibinfo {year} {2019})}\BibitemShut {NoStop}%
\bibitem [{\citenamefont {Giannakis}(2019)}]{Giannakis19b}%
  \BibitemOpen
  \bibfield  {author} {\bibinfo {author} {\bibfnamefont {D.}~\bibnamefont
  {Giannakis}},\ }\href {\doibase 10.1103/PhysRevE.100.032207} {\bibfield
  {journal} {\bibinfo  {journal} {Phys. Rev. E}\ }\textbf {\bibinfo {volume}
  {100}},\ \bibinfo {pages} {032207} (\bibinfo {year} {2019})}\BibitemShut
  {NoStop}%
\bibitem [{\citenamefont {Mezi\'c}(2019)}]{Mezic19}%
  \BibitemOpen
  \bibfield  {author} {\bibinfo {author} {\bibfnamefont {I.}~\bibnamefont
  {Mezi\'c}},\ }in\ \href@noop {} {\emph {\bibinfo {booktitle} {SIAM Conference
  on Applications of Dynamical Systems (DS19)}}}\ (\bibinfo {address}
  {Snowbird, Utah},\ \bibinfo {year} {2019})\BibitemShut {NoStop}%
\bibitem [{\citenamefont {Morgan}(2019{\natexlab{a}})}]{Morgan19}%
  \BibitemOpen
  \bibfield  {author} {\bibinfo {author} {\bibfnamefont {P.}~\bibnamefont
  {Morgan}},\ }\href {\doibase 10.1088/1402-4896/ab0c53} {\bibfield  {journal}
  {\bibinfo  {journal} {Phys. Scr.}\ }\textbf {\bibinfo {volume} {94}},\
  \bibinfo {pages} {075003} (\bibinfo {year} {2019}{\natexlab{a}})}\BibitemShut
  {NoStop}%
\bibitem [{\citenamefont {Morgan}(2019{\natexlab{b}})}]{Morgan19b}%
  \BibitemOpen
  \bibfield  {author} {\bibinfo {author} {\bibfnamefont {P.}~\bibnamefont
  {Morgan}},\ }\href@noop {} {\enquote {\bibinfo {title} {Unary classical
  mechanics},}\ } (\bibinfo {year} {2019}{\natexlab{b}}),\ \Eprint
  {http://arxiv.org/abs/1901.00526} {1901.00526} \BibitemShut {NoStop}%
\bibitem [{\citenamefont {Slawinska}\ \emph {et~al.}(2019)\citenamefont
  {Slawinska}, \citenamefont {Ourmazd},\ and\ \citenamefont
  {Giannakis}}]{SlawinskaEtAl19}%
  \BibitemOpen
  \bibfield  {author} {\bibinfo {author} {\bibfnamefont {J.}~\bibnamefont
  {Slawinska}}, \bibinfo {author} {\bibfnamefont {A.}~\bibnamefont {Ourmazd}},
  \ and\ \bibinfo {author} {\bibfnamefont {D.}~\bibnamefont {Giannakis}},\ }in\
  \href
  {https://www.climatechange.ai/CameraReady/30/CameraReadySubmission/manuscript.pdf}
  {\emph {\bibinfo {booktitle} {Workshop on ``Climate Change: How Can AI
  Help?'', 36th International Conference on Machine Learning (ICML)''}}}\
  (\bibinfo {address} {Long Beach, California},\ \bibinfo {year}
  {2019})\BibitemShut {NoStop}%
\bibitem [{\citenamefont {Ferreira}\ and\ \citenamefont
  {Menegatto}(2013)}]{FerreiraMenegatto13}%
  \BibitemOpen
  \bibfield  {author} {\bibinfo {author} {\bibfnamefont {J.~C.}\ \bibnamefont
  {Ferreira}}\ and\ \bibinfo {author} {\bibfnamefont {V.~A.}\ \bibnamefont
  {Menegatto}},\ }\href@noop {} {\bibfield  {journal} {\bibinfo  {journal}
  {Ann. Funct. Anal.}\ }\textbf {\bibinfo {volume} {4}},\ \bibinfo {pages} {64}
  (\bibinfo {year} {2013})}\BibitemShut {NoStop}%
\bibitem [{\citenamefont {Rosenberg}(1997)}]{Rosenberg97}%
  \BibitemOpen
  \bibfield  {author} {\bibinfo {author} {\bibfnamefont {S.}~\bibnamefont
  {Rosenberg}},\ }\href@noop {} {\emph {\bibinfo {title} {The {L}aplacian on a
  {R}iemannian Manifold}}},\ \bibinfo {series} {London Mathematical Society
  Student Texts}, Vol.~\bibinfo {volume} {31}\ (\bibinfo  {publisher}
  {Cambridge University Press},\ \bibinfo {address} {Cambridge},\ \bibinfo
  {year} {1997})\BibitemShut {NoStop}%
\bibitem [{\citenamefont {Das}\ and\ \citenamefont
  {Giannakis}(2019)}]{DasGiannakis19c}%
  \BibitemOpen
  \bibfield  {author} {\bibinfo {author} {\bibfnamefont {S.}~\bibnamefont
  {Das}}\ and\ \bibinfo {author} {\bibfnamefont {D.}~\bibnamefont
  {Giannakis}},\ }\href@noop {} {\enquote {\bibinfo {title} {Reproducing kernel
  {H}ilbert algebras on compact {A}belian {L}ie groups},}\ } (\bibinfo {year}
  {2019}),\ \Eprint {http://arxiv.org/abs/1912.11664} {1912.11664} \BibitemShut
  {NoStop}%
\bibitem [{\citenamefont {Sriperumbudur}\ \emph {et~al.}(2011)\citenamefont
  {Sriperumbudur}, \citenamefont {Fukumizu},\ and\ \citenamefont
  {Lanckriet}}]{SriperumbudurEtAl11}%
  \BibitemOpen
  \bibfield  {author} {\bibinfo {author} {\bibfnamefont {B.~K.}\ \bibnamefont
  {Sriperumbudur}}, \bibinfo {author} {\bibfnamefont {K.}~\bibnamefont
  {Fukumizu}}, \ and\ \bibinfo {author} {\bibfnamefont {G.~R.}\ \bibnamefont
  {Lanckriet}},\ }\href@noop {} {\bibfield  {journal} {\bibinfo  {journal} {J.
  Mach. Learn. Res.}\ }\textbf {\bibinfo {volume} {12}},\ \bibinfo {pages}
  {2389} (\bibinfo {year} {2011})}\BibitemShut {NoStop}%
\bibitem [{\citenamefont {Paulsen}\ and\ \citenamefont
  {Raghupathi}(2016)}]{PaulsenRaghupathi16}%
  \BibitemOpen
  \bibfield  {author} {\bibinfo {author} {\bibfnamefont {V.~I.}\ \bibnamefont
  {Paulsen}}\ and\ \bibinfo {author} {\bibfnamefont {M.}~\bibnamefont
  {Raghupathi}},\ }\href@noop {} {\emph {\bibinfo {title} {An Introduction to
  the Theory of Reproducing Kernel {H}ilbert Spaces}}},\ \bibinfo {series}
  {Cambridge Studies in Advanced Mathematics}, Vol.\ \bibinfo {volume} {152}\
  (\bibinfo  {publisher} {Cambridge University Press},\ \bibinfo {address}
  {Cambridge},\ \bibinfo {year} {2016})\BibitemShut {NoStop}%
\bibitem [{\citenamefont {Stone}(1932)}]{Stone32}%
  \BibitemOpen
  \bibfield  {author} {\bibinfo {author} {\bibfnamefont {M.~H.}\ \bibnamefont
  {Stone}},\ }\href@noop {} {\bibfield  {journal} {\bibinfo  {journal} {Ann.
  Math}\ }\textbf {\bibinfo {volume} {33}},\ \bibinfo {pages} {643} (\bibinfo
  {year} {1932})}\BibitemShut {NoStop}%
\bibitem [{\citenamefont {ter Elst}\ and\ \citenamefont
  {Lema\'nczyk}(2017)}]{TerElstLemanczyk17}%
  \BibitemOpen
  \bibfield  {author} {\bibinfo {author} {\bibfnamefont {A.~F.~M.}\
  \bibnamefont {ter Elst}}\ and\ \bibinfo {author} {\bibfnamefont
  {M.}~\bibnamefont {Lema\'nczyk}},\ }\href {\doibase 10.1017/etds.2015.111}
  {\bibfield  {journal} {\bibinfo  {journal} {Ergodic Theory Dyn. Syst.}\
  }\textbf {\bibinfo {volume} {37}},\ \bibinfo {pages} {1} (\bibinfo {year}
  {2017})}\BibitemShut {NoStop}%
\bibitem [{\citenamefont {Mezi\'c}(2005)}]{Mezic05}%
  \BibitemOpen
  \bibfield  {author} {\bibinfo {author} {\bibfnamefont {I.}~\bibnamefont
  {Mezi\'c}},\ }\href {\doibase 10.1007/s11071-005-2824-x} {\bibfield
  {journal} {\bibinfo  {journal} {Nonlinear Dyn.}\ }\textbf {\bibinfo {volume}
  {41}},\ \bibinfo {pages} {309} (\bibinfo {year} {2005})}\BibitemShut
  {NoStop}%
\bibitem [{\citenamefont {Sakurai}(1993)}]{Sakurai93}%
  \BibitemOpen
  \bibfield  {author} {\bibinfo {author} {\bibfnamefont {J.~J.}\ \bibnamefont
  {Sakurai}},\ }\href@noop {} {\emph {\bibinfo {title} {Modern Quantum
  Mechanics}}},\ edited by\ \bibinfo {editor} {\bibfnamefont {S.~F.}\
  \bibnamefont {Tuan}}\ (\bibinfo  {publisher} {Addison-Wesley},\ \bibinfo
  {address} {Reading, Massachusetts},\ \bibinfo {year} {1993})\BibitemShut
  {NoStop}%
\bibitem [{\citenamefont {Berline}\ \emph {et~al.}(2004)\citenamefont
  {Berline}, \citenamefont {Getzler},\ and\ \citenamefont
  {Vergne}}]{BerlineEtAl04}%
  \BibitemOpen
  \bibfield  {author} {\bibinfo {author} {\bibfnamefont {N.}~\bibnamefont
  {Berline}}, \bibinfo {author} {\bibfnamefont {E.}~\bibnamefont {Getzler}}, \
  and\ \bibinfo {author} {\bibfnamefont {M.}~\bibnamefont {Vergne}},\
  }\href@noop {} {\emph {\bibinfo {title} {Heat Kernels and {D}irac
  Operators}}}\ (\bibinfo  {publisher} {Springer-Verlag},\ \bibinfo {address}
  {Berlin},\ \bibinfo {year} {2004})\BibitemShut {NoStop}%
\bibitem [{\citenamefont {Michor}(2008)}]{Michor08}%
  \BibitemOpen
  \bibfield  {author} {\bibinfo {author} {\bibfnamefont {P.~W.}\ \bibnamefont
  {Michor}},\ }\href@noop {} {\emph {\bibinfo {title} {Topics in Differential
  Geometry}}},\ \bibinfo {series} {Graduate Studies in Mathematics},
  Vol.~\bibinfo {volume} {93}\ (\bibinfo  {publisher} {American Mathematical
  Society},\ \bibinfo {address} {Providence},\ \bibinfo {year}
  {2008})\BibitemShut {NoStop}%
\bibitem [{\citenamefont {Rudolph}\ and\ \citenamefont
  {Schmidt}(2017)}]{RudolphSchmidt17}%
  \BibitemOpen
  \bibfield  {author} {\bibinfo {author} {\bibfnamefont {G.}~\bibnamefont
  {Rudolph}}\ and\ \bibinfo {author} {\bibfnamefont {M.}~\bibnamefont
  {Schmidt}},\ }\href@noop {} {\emph {\bibinfo {title} {Differential Geometry
  and Mathematical Physics: Part II. Fibre Bundles, Topology and Gauge
  Fields}}},\ Theoretical and Mathematical Physics\ (\bibinfo  {publisher}
  {Springer},\ \bibinfo {address} {Dordrecht},\ \bibinfo {year}
  {2017})\BibitemShut {NoStop}%
\bibitem [{\citenamefont {Yang}\ and\ \citenamefont
  {Mills}(1954)}]{YangMills54}%
  \BibitemOpen
  \bibfield  {author} {\bibinfo {author} {\bibfnamefont {C.~N.}\ \bibnamefont
  {Yang}}\ and\ \bibinfo {author} {\bibfnamefont {R.~L.}\ \bibnamefont
  {Mills}},\ }\href {\doibase 10.1103/PhysRev.96.191} {\bibfield  {journal}
  {\bibinfo  {journal} {Phys. Rev.}\ }\textbf {\bibinfo {volume} {96}},\
  \bibinfo {pages} {191} (\bibinfo {year} {1954})}\BibitemShut {NoStop}%
\bibitem [{\citenamefont {Jaffe}\ and\ \citenamefont
  {Witten}(2000)}]{JaffeWitten00}%
  \BibitemOpen
  \bibfield  {author} {\bibinfo {author} {\bibfnamefont {A.}~\bibnamefont
  {Jaffe}}\ and\ \bibinfo {author} {\bibfnamefont {E.}~\bibnamefont {Witten}},\
  }\href {http://www.claymath.org/sites/default/files/yangmills.pdf} {\enquote
  {\bibinfo {title} {Quantum {Y}ang--{M}ills},}\ }\bibinfo {howpublished} {Clay
  Mathematics Institute} (\bibinfo {year} {2000})\BibitemShut {NoStop}%
\bibitem [{\citenamefont {Hall}(2013)}]{Hall13}%
  \BibitemOpen
  \bibfield  {author} {\bibinfo {author} {\bibfnamefont {B.~C.}\ \bibnamefont
  {Hall}},\ }\href@noop {} {\emph {\bibinfo {title} {Quantum Theory for
  Mathematicians}}},\ \bibinfo {series} {Graduate Texts in Mathematics}, Vol.\
  \bibinfo {volume} {267}\ (\bibinfo  {publisher} {Springer},\ \bibinfo
  {address} {New York},\ \bibinfo {year} {2013})\BibitemShut {NoStop}%
\bibitem [{\citenamefont {Aizenbud}\ and\ \citenamefont
  {Gourevitch}(2011)}]{AizenbudGourevitch11}%
  \BibitemOpen
  \bibfield  {author} {\bibinfo {author} {\bibfnamefont {A.}~\bibnamefont
  {Aizenbud}}\ and\ \bibinfo {author} {\bibfnamefont {D.}~\bibnamefont
  {Gourevitch}},\ }\href@noop {} {\enquote {\bibinfo {title} {Schwartz
  functions on {N}ash manifolds},}\ } (\bibinfo {year} {2011}),\ \Eprint
  {http://arxiv.org/abs/0704.2891} {0704.2891} \BibitemShut {NoStop}%
\bibitem [{\citenamefont {Aronszajn}(1950)}]{Aronszajn50}%
  \BibitemOpen
  \bibfield  {author} {\bibinfo {author} {\bibfnamefont {N.}~\bibnamefont
  {Aronszajn}},\ }\href@noop {} {\bibfield  {journal} {\bibinfo  {journal}
  {Trans. Amer. Math. Soc.}\ }\textbf {\bibinfo {volume} {68}},\ \bibinfo
  {pages} {337–404} (\bibinfo {year} {1950})}\BibitemShut {NoStop}%
\bibitem [{\citenamefont {Micchelli}\ \emph {et~al.}(2006)\citenamefont
  {Micchelli}, \citenamefont {Xu},\ and\ \citenamefont
  {Zhang}}]{MicchelliEtAl06}%
  \BibitemOpen
  \bibfield  {author} {\bibinfo {author} {\bibfnamefont {C.~A.}\ \bibnamefont
  {Micchelli}}, \bibinfo {author} {\bibfnamefont {Y.}~\bibnamefont {Xu}}, \
  and\ \bibinfo {author} {\bibfnamefont {H.}~\bibnamefont {Zhang}},\
  }\href@noop {} {\bibfield  {journal} {\bibinfo  {journal} {J. Mach. Learn.
  Res.}\ }\textbf {\bibinfo {volume} {7}},\ \bibinfo {pages} {2651} (\bibinfo
  {year} {2006})}\BibitemShut {NoStop}%
\bibitem [{\citenamefont {Steinwart}(2001)}]{Steinwart01}%
  \BibitemOpen
  \bibfield  {author} {\bibinfo {author} {\bibfnamefont {I.}~\bibnamefont
  {Steinwart}},\ }\href@noop {} {\bibfield  {journal} {\bibinfo  {journal} {J.
  Mach. Learn. Res.}\ }\textbf {\bibinfo {volume} {2}},\ \bibinfo {pages} {67}
  (\bibinfo {year} {2001})}\BibitemShut {NoStop}%
\bibitem [{\citenamefont {Butzer}\ and\ \citenamefont
  {Westphal}(1974)}]{ButzerWestphal74}%
  \BibitemOpen
  \bibfield  {author} {\bibinfo {author} {\bibfnamefont {P.~L.}\ \bibnamefont
  {Butzer}}\ and\ \bibinfo {author} {\bibfnamefont {U.}~\bibnamefont
  {Westphal}},\ }in\ \href@noop {} {\emph {\bibinfo {booktitle} {Fractional
  Calculus and Its Applications}}},\ \bibinfo {series} {Lecture Notes in
  Mathematics}, Vol.\ \bibinfo {volume} {457},\ \bibinfo {editor} {edited by\
  \bibinfo {editor} {\bibfnamefont {B.}~\bibnamefont {Ross}}}\ (\bibinfo
  {publisher} {Springer-Verlag},\ \bibinfo {address} {Berlin},\ \bibinfo {year}
  {1974})\ pp.\ \bibinfo {pages} {116--145}\BibitemShut {NoStop}%
\bibitem [{\citenamefont {Samko}\ \emph {et~al.}(1993)\citenamefont {Samko},
  \citenamefont {Kilbas},\ and\ \citenamefont {Marichev}}]{SamkoEtAl93}%
  \BibitemOpen
  \bibfield  {author} {\bibinfo {author} {\bibfnamefont {S.~G.}\ \bibnamefont
  {Samko}}, \bibinfo {author} {\bibfnamefont {A.~A.}\ \bibnamefont {Kilbas}}, \
  and\ \bibinfo {author} {\bibfnamefont {O.~I.}\ \bibnamefont {Marichev}},\
  }\href@noop {} {\emph {\bibinfo {title} {Fractional Integrals and
  Derivatives. Theory and Applications}}}\ (\bibinfo  {publisher} {Gordon and
  Breach},\ \bibinfo {address} {Yverdon},\ \bibinfo {year} {1993})\BibitemShut
  {NoStop}%
\end{thebibliography}

\end{document}